\documentclass{amsart}
\usepackage{graphics}
\usepackage{graphicx}
\usepackage{amsfonts,amssymb,amsopn,amscd,amsmath,enumerate}
\usepackage[usenames,dvipsnames]{xcolor}
\usepackage{mathrsfs}
\usepackage{multirow}
\usepackage{tikz}
\usetikzlibrary{plotmarks}
\allowdisplaybreaks

\usepackage[usenames,dvipsnames]{xcolor}
\definecolor{darkblue}{rgb}{0.0, 0.0, 0.55}
\definecolor{bordeaux}{rgb}{0.34, 0.01, 0.1}

\usepackage[pagebackref,colorlinks,linkcolor=bordeaux,citecolor=darkblue,urlcolor=black,hypertexnames=true]{hyperref}

 \newtheorem{theorem}{Theorem}[section]
 
 {\rm}
 \newtheorem{proposition}[theorem]{Proposition}
 \newtheorem{conjecture}[theorem]{Conjecture}

 \newtheorem{definition}[theorem]{Definition}{\rm}
 
 \newtheorem{remark}[theorem]{Remark}

\theoremstyle{definition}
 \newtheorem{example}[theorem]{Example}{\rm}
 
\numberwithin{equation}{section}

\DeclareMathOperator{\Sym}{Sym}

\DeclareMathOperator{\Trace}{tr}

\newif\ifcomment
\commentfalse
\commenttrue
\begin{document}
\def\cA{\mathcal A}
\def\cH{\mathcal H}
\def\red{\color{red}}
\def\bl{\color{blue}}
\def\ora{\color{orange}}
\def\green{\color{green}}
\def\br{\color{brown}}
\newcommand{\realtofloat}{\mathtt{Real2Float}}
\newcommand{\sparsepop}{\mathtt{SparsePOP}}
\newcommand{\ncsostools}{\mathtt{NCSOStools}}
\newcommand{\ncpoltosdpa}{\mathtt{Ncpol2sdpa}}
\newcommand{\gloptipoly}{\mathtt{Gloptipoly}}
\def\s{\mathbb{s}}
\def\la{\langle}
\def\ra{\rangle}
\def\e{{\rm e}}
\def\x{\mathbf{x}}
\def\by{\mathbf{y}}
\def\bz{\mathbf{z}}
\def\F{\mathcal{F}}
\def\R{\mathbb{R}}
\def\Mbb{\mathbb{M}}
\def\Sbb{\mathbb{S}}
\def\T{\mathbb{T}}
\def\N{\mathbb{N}}
\def\K{\mathbb{K}}
\def\bK{\overline{\mathbf{K}}}
\def\Q{\mathbf{Q}}
\def\M{\mathbf{M}}
\def\O{\mathbf{O}}
\def\C{\mathbb{C}}
\def\P{\mathbf{P}}
\def\Z{\mathbb{Z}}
\def\H{\mathcal{H}}
\def\A{\mathbf{A}}
\def\W{\mathbf{W}}
\def\bfone{\mathbf{1}}
\def\V{\mathbf{V}}
\def\AA{\overline{\mathbf{A}}}
\def\L{\mathbf{L}}
\def\bS{\mathbf{S}}
\def\H{\mathcal{H}}
\def\Bbb{\mathbb{B}}
\def\Dbb{\mathbb{D}}
\def\d{\hat{d}}
\def\b{\mathcal{B}}
\def\cc{\mathcal{C}}
\def\co{{\rm co}\,}
\def\cp{{\rm CP}}
\def\tg{\tilde{f}}
\def\tx{\tilde{\x}}
\def\supmu{{\rm supp}\,\mu}
\def\supnu{{\rm supp}\,\nu}
\def\m{\mathcal{M}}
\def\k{\mathcal{K}}
\def\la{\langle}
\def\ra{\rangle}
\def\opt{\text{opt}}
\def\cyc{\overset{\text{cyc}}{\sim}}
\def\smileL{\overset{\smallsmile}{L}}
\def\blambda{{\boldsymbol{\lambda}}}
\def\bsigma{{\boldsymbol{\sigma}}}
\def\RX{\R \langle \underline{X} \rangle}
\def\CX{\C \langle \underline{X} \rangle}
\def\TX{\T \langle \underline{X} \rangle}
\def\KX{\K \langle \underline{X} \rangle}
\def\uX{\underline X}
\def\uY{\underline Y}
\def\uA{\underline A}
\def\uB{\underline B}
\newcommand{\RXI}[1]{\R \langle \underline{X}(I_{#1}) \rangle }
\def\SigmaX{\Sigma \langle \underline{X} \rangle}
\newcommand{\SigmaXI}[1]{\Sigma \langle \underline{X}(I_{#1}) \rangle }
\def\SymRX{\Sym \RX}
\def\SymCX{\Sym \CX}
\def\SymTX{\Sym \TX}
\def\SymT{\Sym \T}
\def\SymKX{\Sym \KX}
\def\ov{\overline{o}}
\def\und{\underline{o}}
\newcommand{\victorm}[1]{\todo[inline,color=purple!30]{VM: #1}}
\newcommand{\victormshort}[1]{\todo[inline,color=purple!30]{VM: #1}}
\newcommand{\victorv}[1]{\Ig{#1}}
\newcommand{\victorvshort}[1]{\todo[color=brown!30]{IK: #1}}
\newcommand{\serban}[1]{{\color{red} Janez: #1}}
\newcommand{\serbanshort}[1]{\todo[color=red!30]{TdW: #1}}

\colorlet{commentcolour}{green!50!black}
\newcommand{\comment}[3]{%
\ifcomment%
	{\color{#1}\bfseries\sffamily(#3)%
	}%
	\marginpar{\textcolor{#1}{\hspace{3em}\bfseries\sffamily #2}}%
	\else%
	\fi%
}
\newcommand{\ViM}[1]{
	\comment{magenta}{I}{#1}
}
\newcommand{\ViV}[1]{
	\comment{green}{J}{#1}
}
\newcommand{\Se}[1]{
	\comment{blue}{V}{#1}
}
\newcommand{\idea}[1]{\textcolor{red}{#1(?)}}

\newcommand{\Expl}[1]{
	{\tag*{\text{\small{\color{commentcolour}#1}}}%
	}
}

\newcommand{\sparsegns}{\texttt{SparseGNS}}
\newcommand{\sparseeiggns}{\texttt{SparseEigGNS}}
\newcommand{\eigmin}{\texttt{NCeigMin}}
\newcommand{\eigminsparse}{\texttt{NCeigMinSparse}}
\def\nsdp{n_{\text{sdp}}}
\def\msdp{m_{\text{sdp}}}

\title[Noncommutative Christoffel-Darboux Kernels]{Noncommutative Christoffel-Darboux Kernels}

\author{Serban T. Belinschi \and Victor Magron \and Victor Vinnikov}

\date{}

\begin{abstract}
We introduce from an analytic perspective Christoffel-Darboux kernels associated to bounded, tracial noncommutative distributions.
We show that properly normalized traces, respectively norms, of evaluations of such kernels on finite dimensional matrices
yield classical plurisubharmonic functions as the degree tends to infinity, and show that they are comparable to certain noncommutative 
versions of the Siciak extremal function. We prove estimates for Siciak functions associated to free products of distributions, and
use the classical theory of plurisubharmonic functions in order to propose a notion of support for noncommutative distributions.
We conclude with some conjectures and numerical experiments.
\end{abstract}



\maketitle

\section{Introduction}
\label{sec:intro}
The goal of this paper is to study the  noncommutative analog of {\em Christoffel-Darboux kernel} associated to a certain class of distributions occuring 
in free probability, and to investigate the related asymptotic properties.
In the classical commutative setting, one considers a finite signed Borel measure $\mu$ on $\mathbb R^n$. Integrating with respect to $\mu$ defines
an inner product on the space of polynomials. The  Christoffel-Darboux kernel \cite{simon2008christoffel} is the reproducing  kernel associated to the 
Hilbert space containing all polynomials up to a given degree. One way to compute this kernel is to rely on the elements of the orthonormal basis of 
this Hilbert space \cite{dunkl2014orthogonal}.

One appealing feature of the Christoffel-Darboux kernel and the related Christoffel function is their ability to capture some properties of $\mu$, such as 
the support or the density of its absolutely continuous component, from the only a priori knowledge of its moments. We refer the interested reader to 
\cite{mate1980bernstein,mate1991szego} for the univariate case. Recent research efforts \cite{lasserre2019empirical} have focused on the multivariate
case. When the measure $\mu$ is uniform or empirical and when the support of $\mu$ satisfies specific compactness conditions, the sequence of level 
sets of the Christoffel function associated to $\mu$ converges to its support with respect to the Hausdorff distance.
The rate of convergence for estimating the support of a measure from a finite, independent,
sample based on such empirical Christoffel-Darboux kernels is analyzed in \cite{vu2019rate}.
Current applications of the Christoffel function include  statistical leverage scores \cite{pauwels2018relating}, sorting out typicality  
\cite{pauwels2016sorting,lasserre2019empirical}, and detection of outliers \cite{beckermann2020perturbations}. 
An extension to the case of certain singular measures is addressed in \cite{pauwels2020data}, for instance when $\mu$ is the Hausdorff measure 
supported on the unit sphere. The framework from \cite{invsdp} allows one to approximate attractors of dynamical systems, while relying on 
Christoffel-Darboux kernels associated to the moment matrix of the measure which is invariant with respect to the dynamics.
Further applications \cite{marx2019tractable} consider approximations of (possibly discontinuous) functions  arising from weak (or measure-valued) 
solutions of optimal control problems or entropy solutions to non-linear hyperbolic PDEs. 
Note that in the two latter cases, the measure $\mu$ can be singular continuous. 

One important motivation for the use of Christoffel-Darboux kernels is to extract the support of measures arising, e.g., in the above-mentioned dynamical 
systems or in a polynomial optimization problem (POP). In the latter case, one minimizes a polynomial over the intersection of finitely many level sets of 
polynomials, i.e., over a basic closed \emph{semialgebraic set}. The minimizers of the POP belong to the support of atomic measures. 
Solving this problem is NP-hard in general~\cite{Laurent:Survey}. {\em Lasserre's hierarchy}~\cite{Las01sos} is a  well established framework to 
approximate the value of POPs. This methodology consists of approximating the optimal value of the initial POP by considering a hierarchy of semidefinite 
programs \cite{anjos2011handbook}, involving moment matrices of growing sizes. By Putinar's Positivstellensatz \cite{Putinar1993positive}, if the 
quadratic module generated by the polynomials describing the semialgebraic set is archimedean, the hierarchy of semidefinite bounds converges from 
below to the minimum of the polynomial over this  semialgebraic set. A somehow more delicate problem is to compute or at least approximate the 
minimizers of the POP. For this, one can extract the support of the atomic measure thanks to \cite{henrion2005detecting}, which provides a numerical 
algorithm based on linear algebra. However, this procedure can be applied only if finite convergence occurs and assuming that the resulting obtained 
moment matrix is flat \cite{curto1998flat}. 
 
In the free noncommutative context, one can use \emph{sum of hermitian squares} decompositions of positive polynomials \cite{Helton02,McCullSOS} 
to perform  eigenvalue optimization of noncommutative polynomials over noncommutative semialgebraic sets. The noncommutative analogue of Lasserre's 
hierarchy \cite{Helton04,navascues2008convergent,pironio2010convergent,cafuta2012constrained,nctrace} allow one to approximate as closely as desired 
the optimal value of such eigenvalue minimization problems. Further efforts  \cite{pironio2010convergent,cafuta2012constrained,nctrace} have been 
pursued to derive a hierarchy of semidefinite relaxations to optimize the trace of a given polynomial under positivity constraints.
Sparsity exploiting hierarchies \cite{klep2021sparse,wang2020exploiting} allow one to reduce the associated computational burden. 
The case of more general trace polynomials has been investigated in \cite{tracencpop}. 
Algorithms similar to the one from 
\cite{henrion2005detecting} allow one to extract optimizers of eigenvalue or trace minimization problems; see, e.g., 
\cite{pironio2010convergent},~\cite[Chapter 21]{anjos2011handbook},~\cite[Theorem~1.69]{burgdorf16} and~\cite{nctrace}.
As for the commutative case, when the extraction procedure fails (i.e.,~without flatness of the associated moment matrix), one still 
hopes to approximate such optimizers by considering them as elements in the support of a tracial noncommutative distribution. 
This support would be approximated by computing the levelsets of the noncommutative Christoffel-Darboux kernel associated to this distribution. 

One of the equivalent definitions of the Christoffel-Darboux kernel is via a sum of orthonormal polynomials. This has been done before by 
Constantinescu \cite{OP1} in the noncommutative context. Further explorations of systems of multivariate orthogonal and orthonormal 
noncommutative polynomials have been undertaken by several authors, for instance \cite{A1,A2,A3,OP2}. We bring as novelty to this 
study the structure of operator spaces \cite{Vern} and the application of the classical analysis of plurisubharmonic functions \cite{GZ} --
which we believe has never been considered before in the noncommutative context.

\section{Preliminaries}
\label{sec:prelim}
\subsection{Hermitian matrices, words and noncommutative polynomials}
Given $k,n\in\mathbb N$, let us denote by $\mathbb M_k(\mathbb C)$ (resp.~$\Sbb_k$) the space of all complex (resp.~hermitian) matrices of 
order $k$, and by $\Sbb_k^n$ the set of $n$-tuples $\underline{A} = (A_1,\dots,A_n)$ of hermitian matrices $A_i$ of order $k$. Let $I_k\in\mathbb M_k(\mathbb C)$ stand  
for the identity matrix. 

For a fixed $n\in\mathbb N$, we consider a finite alphabet $\{X_1,\dots,X_n\}$ and denote by $\langle\underline{X}\rangle$ the set of
all possible words of finite length formed with the letters $X_1,\dots,X_n$. The empty word is denoted by 1. The length of a word $w$ is defined to be
the number of letters (counted with repetition) that form $w$, and is denoted by $|w|$ (for example, $|X_1^5X_2X_1|=7$). 
The length of the empty word is zero.  For $d\in\mathbb N$, $\langle\underline{X}\rangle_d$ is the 
subset of all words of length at most $d$.  We endow $\langle\underline{X}\rangle$ with the graded lexicographic order $\leq_{\text{gl}}$, 
and its strict version $<_{\text{gl}}$ (that is, $w<_{\text{gl}}w'$ iff $w\leq_{\text{gl}}w'$ and $w\neq w'$).
The graded lexicographic order is defined as follows: first, $X_1 <_{\text{gl}}X_2 <_{\text{gl}}\cdots <_{\text{gl}}X_n$. Second, if $w,w'\in\langle
\underline{X}\rangle$ are such that $|w| < |w'|$, then $w <_{\text{gl}} w'$. Third, if $|w|=|w'|$, then one applies the usual lexicographic order: if the first letter of
$w$ is less than the first letter of $w'$, then $w <_{\text{gl}} w'$; if the first letters of $w$ and $w'$ coincide, one goes on to compare the second letters of $w$ and $w'$, and so on.
For example, $X_n<_{\text{gl}}X^2_1<_{\text{gl}}X_1X_2<_{\text{gl}}X_2X_1<_{\text{gl}}X_2^2<_{\text{gl}}X_1^7<_{\text{gl}}X_1^6X_n<_{\text{gl}}
X_1^5X_nX_1$. 

The set $\langle\uX\rangle$ is in fact a monoid, known as the free monoid with $n$ generators, where the multiplication is the juxtaposition of words and
the neutral element is the empty word 1. Thus, the complex vector space spanned by it, which we denote by $\mathbb C\langle\underline{X}\rangle$, 
has a complex algebra structure. This algebra is known also as the algebra of polynomials in noncommutative indeterminates $\underline{X}=(X_1,\dots,X_n)$
(that is, we view the letters as indeterminates). We denote by $\mathbf W_d(\uX)$ the vector of all words of 
$\langle\uX\rangle_d$ ordered w.r.t. $\leq_{\text{gl}}$. The dimension of $\mathbb C\langle\underline{X}\rangle_d$ equals the length of $\mathbf W_d(\underline{X})$, which is 
$\bsigma(n,d):=\sum_{i=0}^dn^i=\frac{n^{d+1}-1}{n-1}$. We equip the set $\mathbb C\langle\uX\rangle$ with the involution $\star$ that conjugates 
the elements of $\mathbb C$, fixes $\{X_1,\dots,X_n\}$ pointwise and reverses words, in such a way that $\mathbb C\langle\uX\rangle$ is the 
$\star$-algebra freely generated by $n$ self-adjoint letters $X_1,\dots,X_n$. 
We will systematically identify the set of words in $n$ letters and the set of monomials in $n$ selfadjoint 
indeterminates, and use this identification in notations as well. It will sometimes be convenient to write a word as $\underline{X}^w$ instead of just $w$;
that is, if $w$ has $p$ letters, we let $w=\ell_1\ell_2\ell_3\cdots\ell_p$, and we identify it with the degree $p$ monomial $\underline{X}^w=X_{\ell_1}X_{\ell_2}X_{\ell_3}
\cdots X_{\ell_p}$. 
The set of all \textit{self-adjoint elements} of $\mathbb C\langle\underline{X}\rangle$ is defined as $\text{Sym}\mathbb C\langle\uX\rangle:=
\{f\in\mathbb C\langle\uX\rangle\colon f=f^\star\}$.
%
\subsection{Tracial functionals and moment matrices}
Let $\tau\colon\mathbb C\langle\uX\rangle\to\C$ be a positive tracial functional, that is, a $\mathbb C$-linear map such that 
$\tau(f^\star f)\ge0$ and $\tau(fg)=\tau(gf)$ for all $f,g\in\mathbb C\langle\underline{X}\rangle$. 
The functional $\tau$ is called {\em faithful} if $\tau(f^\star f) = 0$ implies $f = 0$. 
Let $\M_d(\tau)$ be the moment matrix of $\tau$, i.e., the 
symmetric matrix indexed by words of $\langle\uX\rangle_d$ ordered according to $\leq_{\text{gl}}$, which is defined by $\M_d(\tau)_{u,v}:=\tau(u^\star v)$. 
Since $\tau$ is positive, the moment matrix $\M_d(\tau)$ is positive semi-definite. 
If, in addition, $\tau$ is faithful, then $\M_d(\tau)$ is positive definite,
and thus invertible. As some of the applications we envision are to the theory of von Neumann algebras \cite[Chapter 5]{Takesaki01}, we consider a special subcategory of the 
set of such traces, which we shall call {\em bounded traces}. These are traces $\tau$ that satisfy the additional condition that there exists a real 
number $M>0$ such that $|\tau(w)|<M^d$ for all $d\in\mathbb N$, $w\in\langle\underline{X}\rangle_d$. If $M$ is known/given, we say that 
$\tau$ is bounded by $M$. If in addition $\tau(1)=1$, then $\tau$ is called a {\em bounded tracial state}.
Bounded tracial states correspond to noncommutative probability distributions; see \cite[Proposition 5.2.14 (d)]{anderson2010introduction} 
for more details about how they appear in the free probability literature.
\subsection{Plurisubharmonic functions}\label{PSHintro}
One of the main tools we use in our paper is the theory of plurisubharmonic functions. We introduce here the basic notions 
necessary for our purposes. The main references we use are \cite{GZ,Klimek,ST}. 

There are many equivalent characterizations of plurisubharmonic functions. We choose here the classical path: given a domain  $G\subseteq\mathbb C$, a function
$f\colon G\to[-\infty,+\infty)$ is {\em subharmonic} if it is  not identically equal to $-\infty$, it is upper semicontinuous, and it satisfies the sub-mean value inequality:
$f(z)\leq(2\pi)^{-1}\int_{-\pi}^\pi f(z+re^{i\theta})\,{\rm d}\theta$ for any $z\in G$ and $r>0$ such that the closed disk of center $z$ and radius $r$ is included in $G$.

Given a domain $D\subseteq\mathbb C^n$, a function $u\colon D\to[-\infty,+\infty)$ is {\em plurisubharmonic} if it is upper semicontinuous, not identically $-\infty$, 
and for any $z\in D$ and $b\in\mathbb C^n$, the function $\zeta\mapsto u(z+\zeta b)$ is subharmonic or identically $-\infty$ on each component of $\{\zeta\in\mathbb C
\colon z+\zeta b\in D\}$. 

An extremely useful characterization of plurisubharmonic functions is given in terms of their second derivatives: essentially, $u$ is plurisubharmonic if and only if for any 
$\xi\in\mathbb C^n$, $\langle\mathcal L_u(\underline{z})\xi,\xi\rangle:=\sum_{1\le j,k\le n}\xi_j\bar\xi_k\frac{\partial^2 u}{\partial z_j\partial \bar{z}_k}\ge0$ in the sense of 
(Schwartz) distributions -- hence $\langle\mathcal L_u(\cdot)\xi,\xi\rangle$ is a positive measure on $\mathbb C^n$. For the precise statement and proof of this result, see, for 
instance, \cite[Proposition 1.43]{GZ}, or \cite[Section 2.9]{Klimek}.

Plurisubharmonic functions behave well with respect to taking limits. Two such results will be implicitly used later in our paper (see \cite[Propositions 1.28, 1.39, and 1.40]{GZ}):
\begin{enumerate}
\item Let $\{u_k\}_{k\in\mathbb N}$ be a decreasing sequence of plurisubharmonic functions in $D$. If $\lim_{k\to\infty}u_k$ is not identically $-\infty$, then it is plurisubharmonic.
\item Let $(u_i)_{i\in I}$ be a family of plurisubharmonic functions in a domain $D$, which is locally uniformly bounded from above, and let $u=\sup_{i\in I}u_i$. Then
the upper semicontinuous regularization $u^*$ of $u$ is plurisubharmonic in $D$ ($u^*$ can be obtained as $u^*(z)=\limsup_{\zeta\to z}u(\zeta)$).\label{2}
\end{enumerate}
The set $\{u^*> u\}$ is generally small -- always of zero Lebesgue measure. In cases of interest to us, this result can be strenghtened. For this, we introduce the notion of
pluripolar sets. A set $F\subset\mathbb C^n$ is called pluripolar if if for all $a\in F$ there exists a neighborhood $W$ of $a$ in $\mathbb C^n$ and a function $v$ which is
plurisubharmonic on $W$ such that $F\cap W\subseteq\{v=-\infty\}$ (see \cite[Definition 1.4, Appendix B.1]{ST} and comments following it).
It is known that the set on which a plurisubharmonic function is equal to $-\infty$ cannot be large. For instance, it has volume zero both in $\mathbb C^n$ and in any (maximal
dimension) Euclidean sphere or torus included in $\mathbb C^n$ -- see \cite[Proposition 1.34]{GZ}, and, for a much deeper understanding, \cite[Section 4.4]{GZ}. A property
is said to hold {\em quasi-everywhere} on a set $S$ if it holds on $S\setminus F$ for a pluripolar set $F$. Theorem 1.7 of \cite[Appendix B.1]{ST}, due to Bedford and Taylor,
states that if the functions $u_i,i\in I$, from \eqref{2} above are defined on all of $\mathbb C^n$ and have at most logarithmic growth at infinity, then the set
$\{u^*> u\}$ is actually pluripolar. Moreover, according to Theorem 1.6 from the same reference, $u^*$ itself has at most logarithmic growth at infinity.

\section{Noncommutative Christoffel-Darboux Kernels}
\label{sec:ncchristoffel}

\subsection{Orthonormal polynomials}\label{ort}
Assume the linear functional $\tau\colon\mathbb C\langle\underline{X}\rangle\to\mathbb C$ is positive. Then $(u,v)\mapsto\tau(v^\star u)$ defines a positive sesquilinear
form on $\mathbb C\langle\underline{X}\rangle$ which transforms it into a pre-Hilbert space. Completing it with respect to the seminorm $\|u\|=\tau(u^\star u)^\frac12$
and factoring out the kernel $\{\|u\|=0\}$ yields a Hilbert space which we  denote by $L^2(\tau)$. Although most of the following statements still hold under
weaker hypotheses, we assume from now on that $\tau$ is a faithful tracial state. In that case $\langle\underline{X}\rangle$ is linearly independent and spans $L^2(\tau)$ as
Hilbert space. For all $d\in\mathbb N$, let us define the family of orthonormal polynomials, 
$\{P_w\}_{w\in\langle\uX\rangle_d}$, ordered according to the lexicographic order, and  satisfying for all $v,w\in\langle\uX\rangle_d$:
\begin{align}
\label{eq:orthonormal}
\tau(P_v^\star\,P_w)=\delta_{v=w}\,, \quad\tau(P_v^\star\,w)=0\,, \text{ if } w <_{\text{gl}} v \,, \quad\tau(P_w^\star\, w)>0\,,\quad P_1=1.
\end{align}
Of course, $\mathrm{span}\{w\colon w<_{\text{gl}}v\}=\mathrm{span}\{P_w\colon w<_{\text{gl}}v\}$ for all $v\in\langle\uX\rangle$. Such a 
family can be constructed easily by using the classical Gram-Schmidt orthonormalization process applied to the basis of monomials $\langle\underline{X}\rangle$. 
Written in the language of the $\tau$-induced 
inner product, it is given recursively by the initial condition $P_1(\uX)=1\in\mathbb C$ and the general expression 
$$
P_w(\underline{X})=
\frac{w\displaystyle-\sum_{v<_{\text{gl}}w}\tau(P^\star_v(\uX)w)P_v(\uX)}{\displaystyle\left(\tau(w^\star w)-\sum_{v<_{\text{gl}}w}|
\tau(P^*_v(\uX)w)|^2\right)^\frac12}, \quad w\in\langle\underline{X}\rangle_d.
$$
It is immediately clear from the construction that 
$\tau(P_w^\star P_w)=1$, $\tau(P_v^\star P_w)=0$ 
if $v<_{\text{gl}}w$, and $$\tau(w^\star P_w)=
\frac{\tau(w^\star w)-\sum_{v<_{\text{gl}}w}\tau(P^*_v(\uX)w)\tau(w^\star P_v(\uX))}{\left(\tau(w^\star w)-\sum_{v<_{\text{gl}}w}|
\tau(P^*_v(\uX)w)|^2\right)^\frac12}>0.$$ Of course, since $\tau(P_w^\star w)=\overline{\tau(w^\star P_w)}$, we also have $\tau(P_w^\star w)>0$. 
Finally, if $v_0<_{\text{gl}}w,$ then $v_0$ is a linear combination of elements $P_v,v\le_{\text{gl}}v_0$, so that $\tau(P_wv^\star_0)=
\sum_{v\le_{\text{gl}}v_0}\bar{e}_v\tau(P_wP_v^\star)=0.$ In the context, we should emphasize that faithfulness of $\tau$ automatically implies 
the denominator in the expression of $P_w(\underline{X})$ above is nonzero. If $\tau$ is not faithful, it may well happen that $\tau(w^\star w)=
\sum_{v<_{\text{gl}}w}|\tau(P^*_v(\uX)w)|^2$; in that case, we may perform the orthonormalization procedure above and the vast majority of the
analysis that follows below on a quotient space, as in \cite{beckermann2020perturbations}: one ``drops'' in the Gram-Schmidt process any monomial $w$ 
which is a linear combination of elements $v<_{\text{gl}}w$ in the pre-Hilbert space induced by $\tau$. 

\begin{remark}\label{tres-uno}
\begin{trivlist}
\item[(1)]
Since $\tau(P^\star)=\overline{\tau(P)}$, it follows immediately from the above that the family $\{P_w^\star\}_{w\in\langle\uX\rangle_d}$ is itself 
an orthonormal basis for $\mathbb C\langle\uX\rangle_d$. Moreover, $\tau(P_ww^\star)\allowbreak>0$ if $\tau$ is faithful, and $\tau(P_vw^\star)=0$ whenever $w<_\text{gl}v$.
However, generally $P_w^\star\neq P_{w^\star}$ 
and it is not even 
clear whether $P_w^\star\in\{P_w\colon w\in\langle\uX\rangle_d\}$. It is nevertheless clear that 
$\mathrm{span}\{P_w\colon w\in\langle\uX\rangle_d\setminus\langle\uX\rangle_{d-1}\}
=\mathrm{span}\{P_w^\star\colon w\in\langle\uX\rangle_d\setminus\langle\uX\rangle_{d-1}\}$, $d\in\mathbb N$.

\item[\ (2)] The fact that both $\{P_w\}_{w\in\langle\uX\rangle_d}$ and $\{P_w^\star\}_{w\in\langle\uX\rangle_d}$ are orthonormal bases implies that any 
correspondence sending one to the other is a unitary transformation of $L^2(\mathbb C\langle\underline{X}\rangle_d,\tau)$. Thus, there exists a 
unitary matrix $U=(U_{v,w})_{v,w\in\langle\underline{X}\rangle_d}$ such that $P^\star_{v_0}(\underline{X})=\sum_{v\in\langle\uX\rangle}U_{v_0,v}
P_v(\underline{X}).$ (As seen above, $U$ leaves $\mathrm{span}\{P_w\colon w\in\langle\uX\rangle_d\setminus\langle\uX\rangle_{d-1}\}$ invariant 
for all $d$, so it is block-diagonal.) The matrix $U$ being unitary is equivalent to stating that $UU^\star=U^\star U=I_{\bsigma(n,d)}$. In particular, 
$P_{v_0}(\underline{X})=\sum_{v\in\langle\uX\rangle}(U^\star)_{v_0,v}P_v^\star(\underline{X})=\sum_{v\in\langle\uX\rangle}\overline{U_{v,v_0}}
P_v^\star(\underline{X})$, and $\delta_{v=v_0}=\sum_{w\in\langle\underline{X}\rangle_d}U_{v,w}\overline{U_{v_0,w}}$.
Moreover, the orthonormality relations yield $\tau(P^\star_{v_0}(\underline{X})P^\star_{v_1}(\underline{X}))=\sum_{v\in\langle\uX\rangle}U_{v_0,v}$ 
$\tau(P_v(\underline{X})P^\star_{v_1}(\underline{X}))=U_{v_0,v_1}.$ Under the assumption of traciality for $\tau$, we obtain that $U_{v_0,v_1}=U_{v_1,v_0}$ for all 
$v_0,v_1\in\langle\underline{X}\rangle_d$, so that the unitary matrix $U$ (and necessarily its adjoint $U^\star$ too) is symmetric.
This simple observation will be useful in the definition of the Christoffel-Darboux kernel.
\end{trivlist}
\end{remark}

Polynomials $P_w$ are usually not selfadjoint, and the construction offered above cannot be expected to provide 
selfadjoint orthonormal bases from a non-selfadjoint set of monomials. While the orthonormal basis $\{P_w\}_{w\in\langle\underline{X}\rangle}$
turns out to be sufficient for our purposes, we would like to specify as an aside that it is quite easy to create a family of selfadjoint orthonormal 
polynomials out of the hermitization of our monomials. This follows directly from the nature of the Gram-Schmidt process, if we place a convenient
modification of the lexicographic order on the set of hermitized monomials. 
It is well-known that if $\underline{X}^w$ is not selfadjoint, then $\text{Span}_\mathbb C\{\underline{X}^w,\underline{X}^{w^\star}\}=
\text{Span}_\mathbb C\left\{\frac{\underline{X}^w+\underline{X}^{w^\star}}{2},\frac{\underline{X}^w-\underline{X}^{w^\star}}{2i}\right\},$ and 
$\frac{\underline{X}^{w^\star}+\underline{X}^w}{2},\frac{\underline{X}^w-\underline{X}^{w^\star}}{2i}$ are both selfadjoint. For any 
$d\in\mathbb N$, take the set $\langle\underline{X}\rangle_d\setminus\langle\underline{X}\rangle_{d-1}$ and proceed 
through the set of monomials according to the lexicographic order. We proceed the following way: when we see a selfadjoint element, we leave it 
alone; when we see a non-selfadjoint monomial $\underline{X}^w$, we replace it with its {\em real part} $\Re\underline{X}^w=
\frac{\underline{X}^w+\underline{X}^{w^\star}}{2}$. At the same time, we replace $\underline{X}^{w^\star}$ (which, according to this procedure, 
must necessarily be greater than $\underline{X}^w$) with the the {\em imaginary part} of $\underline{X}^w$, namely $\Im\underline{X}^w=
\frac{\underline{X}^w-\underline{X}^{w^\star}}{2i}$, to be placed in our ordered list in the slot previously occupied by $\underline{X}^{w^\star}$  (of course, the first 
element in the list is selfadjoint, namely $X^d_1$). This process yields a basis of $\text{Span}_\mathbb C(\langle\underline{X}\rangle_d\setminus\langle\underline{X}\rangle_{d-1})$ 
which is entirely composed of selfadjoint elements, for each $d\in\mathbb N$ (the case $d=0$ corresponds to the selfadjoint 1 and the case $d=1$ to the selfadjoints $X_1,\dots,X_n$).
Thus, our leave it ordered according to the same lexicographic order, where the position of selfadjoint monomials remains
unchanged, and if $w^\star {}_{\text{gl}}\!\!>w$, then the position of $\Re\underline{X}^w$ is at $w$, and the position of $\Im\underline{X}^w$ is at $w^\star$.
The Gram-Schmidt procedure applied to this basis yields a family $\{S_w\}_{w\in\langle\underline{X}\rangle}\subset\text{Sym}
\mathbb C\langle\uX\rangle\subset\mathbb C\langle\underline{X}\rangle$ of selfadjoint orthogonal polynomials. This is a known fact
(it can be viewed as a reformulation of the general fact that the Gram-Schmidt procedure applied to a basis in a real Hilbert space with respect
to the real inner product or with respect to its complexification yields the same result). However, it can also easily be argued by induction after the ordered set of words. 
With the convention, valid only in this paragraph, that $w\in\langle\uX\rangle$ denotes not the monomial $\uX^w$, 
but the selfadjoint basis element it indexes according to the above procedure, the general formula of the $w_0^{\rm th}$ orthonormal polynomial is, as before,
$$
S_{w_0}(\underline{X})=
\frac{w_0\displaystyle-\sum_{v<_{\text{gl}}w_0}\tau(S_v(\uX)w_0)S_v(\uX)}{\displaystyle\left(\tau(w_0^\star w_0)-\sum_{v<_{\text{gl}}w_0}|
\tau(S_v(\uX)w_0)|^2\right)^\frac12}, \quad w_0\in\langle\underline{X}\rangle_d.
$$
By hypothesis, $w_0\!=\!w_0^\star,S_v(\uX)\!=\!S_v^\star(\uX)$. Moreover, $\overline{\tau(S_v(\uX)w_0)}\!=\!\tau\left((S_v(\uX)w_0)^\star\right)
\!=\tau(w_0^\star S_v^\star(\uX))=\tau(w_0S_v(\uX))=\tau(S_v(\uX)w_0)$, so that $\tau(S_v(\uX)w_0)\in\mathbb R$. Thus, 
$S_{w_0}(\underline{X})=S_{w_0}^\star(\underline{X})$. As $S_1(\uX)=1$ is selfadjoint, this completes our argument.


Before going forward to the study of the  Christoffel-Darboux kernel, let us list a few simple or (by now) well-known examples of orthogonal polynomials:
\begin{example}\label{ortex}\noindent
\begin{enumerate}
\item It is well-known, and easy to check, that any compactly supported Borel probability measure on $\mathbb R$ whose support contains an infinity of 
points admits a family of orthonormal polynomials from $\mathbb C\langle X\rangle=\mathbb C[X].$ The state $\tau$ corresponding to it is simply the 
integration of the polynomial with respect to the given probability measure. $\tau$ is in fact faithful if and only if the support of the corresponding 
probability measure is an infinite set. Otherwise, the number of linearly independent monomials equals the number of points in the support.
The reader can find a vast and fascinating literature on the subject by searching in \cite{dunkl2014orthogonal} and references therein.

\item Multivariate orthogonal polynomials are another obvious classical example, to which reference \cite{dunkl2014orthogonal} is mainly dedicated. Given 
a Euclidean space $\mathbb R^n$, a compactly supported Borel probability measure on it admits a family of orthonormal polynomials from 
$\mathbb C[\uX]=\mathbb C\langle\uX\rangle/\langle X_iX_j=X_jX_i\colon1\le i,j,\le n\rangle$ with respect to the trace $\tau$ defined by the 
integration with respect to the given probability measure. If the support of our probability is not concentrated on any finite union of algebraic curves 
(for instance if its interior in $\mathbb R^n$ is nonempty), then the trace is faithful on $\mathbb C[\uX]$. However, it is quite clear that such a trace 
does {\em not} satisfy our condition of faithfulness: since $-(X_1X_2-X_2X_1)^2=(X_1X_2-X_2X_1)(X_2X_1-X_1X_2)=(X_2X_1-X_1X_2)^\star
(X_2X_1-X_1X_2)$ is a positive polynomial in $\mathbb C\langle\uX\rangle\setminus\{0\}$ but it is the zero polynomial in $\mathbb C[\uX]$, we have 
$\tau(-(X_1X_2-X_2X_1)^2)=\tau(0)=0$, so $\tau$ is far from being faithful on $\mathbb C\langle\uX\rangle$ whenever $n\ge2$. Thus, the 
noncommutative version of orthogonal polynomials differs drastically from the classical one in more than one variable.

\item An example from \cite[Section 3]{A1} will be useful later in the paper. Assume that $\tau$ is bounded and the components $X_1,\dots,X_n$ of 
$\uX$ are free \cite{V1} with respect to $\tau$. Assume moreover that $\tau$ is faithful, which, in this case, is equivalent to requiring that the - classical -
distribution of $X_j$ with respect to $\tau$ has infinite support in $\mathbb R$. We denote by $P^{(j)}_k(X_j),k\in\mathbb N,$ the orthogonal polynomials 
associated to the distribution of $X_j$ with respect to $\tau$. Given $w\in\langle\uX\rangle_d$, we write it as $w=(w_1,w_2,\dots,w_d)$. We partition 
$w$ in intervals given by consecutively repeating indices: $w=(\pi_1,\pi_2,\dots,\pi_\ell)$ for some $\ell\in\{1,\dots,d\}$. For example, if $n=3$, $d=10$, 
and $w=(2,2,2,1,3,3,1,1,1,1)$, then $w=(\pi_1,\pi_2,\pi_3,\pi_4)$ where $\pi_1=(2,2,2),\pi_2=(1),\pi_3=(3,3),\pi_4=(1,1,1,1)$. We associate to
each $\pi_j$ the orthonormal polynomial $P^{(\iota_j)}_{|\pi_j|}(X_{\iota_j})$, where $|\pi_j|$ is the cardinality of the block $\pi_j$, and $\iota_j\in\{1,
\dots,n\}$ is the index that is contained by $\pi_j$. Then $P_w(\uX)=P^{(\iota_1)}_{|\pi_1|}(X_{\iota_1})P^{(\iota_2)}_{|\pi_2|}(X_{\iota_2})\cdots 
P^{(\iota_\ell)}_{|\pi_\ell|}(X_{\iota_\ell})$, $w\in\langle\uX\rangle$, is a system of orthonormal polynomials for $\tau$. (For example, with 
$w=(2,2,2,1,3,3,1,1,1,1)$, we would obtain $P_w(X_1,X_2,X_3)=P_3^{(2)}(X_2)P_1^{(1)}(X_1)P_2^{(3)}(X_3)P_4^{(1)}(X_1)$.) This follows from 
the definition of free independence and the condition of orthonormality imposed on each family $P^{(j)}_k(X_j)$, $1\le j\le n,k\in
\mathbb N.$ The case when, say, $\mathbb C\langle X_1,\dots,X_p\rangle$ and $\mathbb C\langle X_{p+1},\dots,X_n\rangle$ are free with respect to 
$\tau$ is treated conceptually precisely the same way, and the orthonormal polynomials associated to $\tau$ on $\mathbb C\langle\uX\rangle$ are 
again alternating products of orthonormal polynomials corresponding to the restrictions of $\tau$ to $\mathbb C\langle X_1,\dots,X_p\rangle$ and 
$\mathbb C\langle X_{p+1},\dots,X_n\rangle$, respectively, but the notation becomes more involved. Faithfulness of $\tau$ on either of
$\mathbb C\langle X_1,\dots,X_p\rangle$ or $\mathbb C\langle X_{p+1},\dots,X_n\rangle$ is not necessary for the above result to hold.

\item Several other interesting examples can be found in, or adapted from, \cite{A1,A2}, including free Meixner laws or orthogonal polynomials 
corresonding to other noncommutative independences.
\end{enumerate}
\end{example}

\subsection{Operator space structure and the Christoffel-Darboux kernel}
For any two vector spaces $\mathcal{ V,W}$ over $\mathbb C$, the algebraic tensor product vector space over $\mathbb C$, 
$\mathcal V\otimes\mathcal W$, is well-defined. If $\mathcal{V,W}$ are complex algebras, this vector space has itself two algebra 
structures: $\mathcal V\otimes\mathcal W$  with the multiplication $(\sum_i v_i\otimes w_i)(\sum_j \tilde{v}_j\otimes \tilde{w}_j)=
\sum_{i,j}v_i\tilde{v}_j\otimes w_i\tilde{w}_j$ and $\mathcal V\otimes\mathcal W^{\rm op}$ with the multiplication
$(\sum_i v_i\otimes w_i)(\sum_j \tilde{v}_j\otimes \tilde{w}_j)=\sum_{i,j}v_i\tilde{v}_j\otimes \tilde{w}_jw_i$. This second 
multiplication is important because it allows in some circumstances the identification 
$\mathcal V\otimes\mathcal V^{\rm op}\simeq\mathscr L(\mathcal V,\mathcal V)$, the space of linear operators from $\mathcal V$
to itself, via $\sum_i v_i\otimes \tilde{v}_i\mapsto\left[x\mapsto\sum_i v_ix\tilde{v}_i\right]$. This identification works well for the space 
$\Mbb_k(\mathbb C)$ and, when $\mathcal V$ is a von Neumann algebra, there are completions of the range of 
$\mathcal V\otimes\mathcal V^{\rm op}$ in $\mathscr L(\mathcal V,\mathcal V)$ that 
are quite important in the theory of operator spaces (we refer to \cite{ER,Vern,Pisier} as fundamental sources for operator spaces theory,
and as an example of the application of the usefulness of this view of $\mathcal V\otimes\mathcal V^{\rm op}$ in free probability, to
the now-classical paper \cite{GS}).

Let us apply the above to $\mathcal{V}=\Mbb_k(\mathbb C)$ and $\mathcal W=\mathbb C\langle\uX\rangle$. We have the natural identification 
$\Mbb_k(\mathbb C)\otimes\mathbb C\langle\uX\rangle\simeq\Mbb_k(\mathbb C\langle\uX\rangle)$. This is a star algebra with the adjoint operation 
$[(P_{ij}(\uX))_{i,j=1}^k]^\star=(P_{ji}^\star(\uX))_{i,j=1}^k$. Writing $(P_{ij}(\uX))_{i,j=1}^k$ as $(P_{ij}(\uX))_{i,j=1}^k=
\left(\sum_w\alpha_{ij}^{(w)}P_w(\uX)\right)_{i,j=1}^k$, $\alpha_{ij}^{(w)}=\tau\left(P_{ij}(\uX)P_w^\star(\uX)\right)$, yields
\begin{eqnarray*}
(P_{ij}(\uX))_{i,j=1}^k=\left(\sum_w\alpha_{ij}^{(w)}P_w(\uX)\right)_{i,j=1}^k & = & \sum_{i,j=1}^k \sum_w\alpha_{ij}^{(w)}e_{ij}\otimes 
P_w(\uX)\\
& = & \sum_w\left(\sum_{i,j=1}^k\alpha_{ij}^{(w)}e_{ij}\right)\otimes P_w(\uX)\\
& = & \sum_wC_w((P_{ij}(\uX))_{i,j=1}^k)\otimes P_w(\uX),
\end{eqnarray*}
where $C_w((P_{ij}(\uX))_{i,j=1}^k)\in\Mbb_k(\mathbb C),C_w((P_{ij}(\uX))_{i,j=1}^k)=\sum_{i,j=1}^k\alpha_{ij}^{(w)}e_{ij}$. 
Written in this form, the star operation is simply $\left(\sum_wC_w((P_{ij}(\uX))_{i,j=1}^k)\otimes P_w(\uX)\right)^\star
=\sum_wC_w((P_{ij}(\uX))_{i,j=1}^k)^\star\otimes P_w^\star(\uX)$. We define an $\Mbb_k(\mathbb C)$-valued sesquilinear form
$\left\langle(P_{ij}(\uX))_{i,j=1}^k,(Q_{ij}(\uX))_{i,j=1}^k\right\rangle=
(\text{Id}_{\Mbb_k(\mathbb C)}\otimes\tau)\left((P_{ij}(\uX))_{i,j=1}^k[(Q_{ij}(\uX))_{i,j=1}^k]^\star\right)
=\sum_{w}C_w((P_{ij}(\uX))_{i,j=1}^k)C_w((Q_{ij}(\uX))_{i,j=1}^k)^\star.$ With these notations, if $A\in\Mbb_k(\mathbb C),$ ${\bf P}, {\bf Q}\in
\Mbb_k(\mathbb C\langle\uX\rangle),$ then $\langle{\bf P}A,{\bf Q}\rangle=\langle{\bf P},{\bf Q}A^*\rangle$, 
$\langle A{\bf P},{\bf Q}\rangle=A\langle{\bf P},{\bf Q}\rangle$, $\langle{\bf P},A{\bf Q}\rangle
=\langle{\bf P},{\bf Q}\rangle A^\star$. This makes $\Mbb_k(\mathbb C\langle\uX\rangle)$ into a $\Mbb_k(\mathbb C)$-Hilbert 
module (see \cite[Chapter 14]{Vern}). From now on, we denote matrix-valued polynomials by boldface letters. 

\begin{remark}\label{N}
For fixed $0\neq k\in\mathbb N$, a polynomial ${\bf P}\in\Mbb_k(\mathbb C\langle\uX\rangle),\mathbf P(\underline{X})=\sum_{w\in\langle\uX\rangle}
c_w\otimes P_w(\uX)$ is a noncommutative function \cite[Section 2.1]{kaliuzhnyi2014foundations}. Indeed, consider an arbitrary unital $C^*$-algebra $\mathcal A$ 
(we will mostly need the case $\mathcal A=\Mbb_l(\mathbb C)$ for some $l$, possibly different from $k$) 
and an $\Mbb_k(\mathbb C)$-$\mathcal A$-bimodule $\mathcal B$ $($usually $\mathbb M_{k\times l}(\mathbb C))$. One defines
the function 
$$
\mathcal A^n\times\mathcal B\ni(\underline{a},b)\mapsto\mathbf P(\underline{a},b)=\sum_{w\in\langle\uX\rangle}c_wbP_w(\underline{a}).
$$
Amplification to $\iota\times\iota$ matrices is done the obvious way: if 
$(\underline{a},b)\in(\mathbb M_\iota(\mathcal A))^n\times\mathbb M_\iota(\mathcal B)$, then 
$$
\mathbf P(\underline{a},b)=\sum_{w\in\langle\uX\rangle}(c_w\otimes I_\iota)bP_w(\underline{a})=\sum_{w\in\langle\uX\rangle}
[\mathrm{diag}(\underbrace{c_w,\dots,c_w}_{\iota\text{ \rm times}})]bP_w(\underline{a}).
$$
\end{remark}
For reasons that will become clear shortly, we prefer to write $\mathbf P(\underline{a})(b)$ instead of $\mathbf P(\underline{a},b)$. 
Indeed, evaluation of polynomials as above is essential in the rest of the paper. For a $\mathbf P(\underline{X})\in
\mathbb M_k(\mathbb C\langle\uX\rangle)$, the evaluation in an $n$-tuple of matrices $\uA\in\mathbb M_k(\mathbb C)^n$ allows us, as 
already mentioned, to view $\mathbf P$ as a linear map on $\mathbb M_k(\mathbb C)$:
if $\mathbf P(\underline{X})=\sum_{w\in\langle\uX\rangle}c_w\otimes\uX^w$, then $\mathbf P(\underline{A})=
\sum_{w\in\langle\uX\rangle}c_w\otimes\uA^w\colon\mathbb M_k(\mathbb C)\to\mathbb M_k(\mathbb C),$ 
$\mathbf P(\underline{A})(C)=\sum_{w\in\langle\uX\rangle}c_wC\uA^w$. In this context, recall that the star operation 
on $\mathbb C\langle\uX\rangle$ extends to $\mathbb M_k(\mathbb C\langle\uX\rangle)$ the obvious way: if, say, $\mathbf P(\underline{X})
=\sum_{w\in\langle\uX\rangle}c_w\otimes\uX^w$, then $\mathbf P^\star(\underline{X})=
\sum_{w\in\langle\uX\rangle}c^\star_w\otimes\uX^{w^\star}=\sum_{w\in\langle\uX\rangle}c^\star_{w^\star}\otimes\uX^w$.
Thus, when performing evaluations on possibly non-selfadjoint tuples, $\mathbf P(\underline{A})^\star
=\left[\sum_{w\in\langle\uX\rangle}c_w\otimes\uA^w\right]^\star=\sum_{w\in\langle\uX\rangle}c_w^\star\otimes(\uA^w)^\star
=\sum_{w\in\langle\uX\rangle}c_w^\star\otimes(\uA^\star)^{w^\star}=\sum_{w\in\langle\uX\rangle}c^\star_{w^\star}
\otimes(\uA^\star)^w=\mathbf P^\star(\uA^\star)$. It is useful to record this equality, together with the one corresponding to scalar polynomials:
\begin{equation}\label{starry}
\mathbf P(\underline{A})^\star=\mathbf P^\star(\uA^\star),\quad P(\uA)^\star=P^\star(\uA^\star), \quad \uA\in\mathbb M_k(\mathbb C)^n,k\in
\mathbb N.
\end{equation}
It is remarkably convenient that this star operation coincides with the adjoint operation when $\mathbb M_k(\mathbb C)$ is viewed as a Hilbert space with 
the Hilbert-Schmidt norm. Indeed, with Tr denoting the non-normalized trace,
\begin{align*}
\left\langle C,\mathbf P(\underline{A})(D)\right\rangle_\mathrm{HS} 
& = \mathrm{Tr}_k\left(\left[\sum_{w\in\langle\uX\rangle}c_wD\uA^w\right]^\star C\right)
= \mathrm{Tr}_k\left(\sum_{w\in\langle\uX\rangle}(\uA^w)^\star D^\star c^\star_wC\right)\\
&= \mathrm{Tr}_k\left(D^\star \sum_{w\in\langle\uX\rangle} c^\star_{w^\star}C(\uA^\star)^w\right)= 
\mathrm{Tr}_k\left(D^\star\mathbf P^\star(\uA^\star)(C)\right)\\
&=\left\langle\mathbf P^\star(\uA^\star)(C),D\right\rangle_\mathrm{HS},\quad 
C,D\in\mathbb M_k(\mathbb C),\uA\in\mathbb M_k(\mathbb C)^n,k\in\mathbb N.
\end{align*}
It should be emphasized that re-normalizing Tr changes nothing in the above. Also, the same algebraic computations
apply equally well for a tracial state on a C$^{\star}$-algebra. In particular, for a II${}_1$ factor $\mathcal M$ with normal 
faithful trace-state tr, if $\mathbf P(\uX)\in\mathcal M\otimes\mathbb C\langle\uX\rangle$, then $[\mathbf P(\underline{a})]^\star(c)=
\mathbf P^\star(\underline{a}^\star)(c)$ for any $\underline{a}\in\mathcal M^n,c\in L^2(\mathcal M,\mathrm{tr})$.

When viewed as an operator on the Hilbert space $\mathbb M_k(\mathbb C)$, the norm of $\mathbf P(\underline{A})$ is defined the usual way:
$\|\mathbf P(\underline{A})\|=\sup\{\|\mathbf P(\underline{A})(C)\|_\mathrm{HS}\colon\|C\|_\mathrm{HS}=1\}.$ This is essentially the 
same norm as the $C^*$-norm of $\mathbf P(\underline{A})$ when viewed as acting on $\mathbb C^k\otimes\mathbb C^k$. 
Indeed, with $\mathbf P(\underline{A})=\sum_{w\in\langle\uX\rangle}c_w\otimes\uA^w$, the following two quantities are equal:
\begin{equation}\label{HS}
\sup_{\|C\|_\mathrm{HS}=1}\|\mathbf P(\underline{A})(C)\|_\mathrm{HS}^2=
\sup_{\|C\|_\mathrm{HS}=1}\mathrm{Tr}_k\left(\sum_{v,w\in\langle\uX\rangle}(\uA^v)^\star C^\star c^\star_vc_wC\uA^w\right),
\end{equation}
\begin{equation}\label{H}
\sup\|\mathbf P(\underline{A})\sum_{i,j=1}^k\alpha_{i,j}e_i\otimes e_j\|_2^2=\sup\sum_{v,w\in\langle\uX\rangle}\sum_{i,j,p,q=1}^k\alpha_{i,j}
\bar{\alpha}_{p,q}\langle c_ve_i,c_we_p\rangle\langle\uA^ve_j,\uA^we_q\rangle,
\end{equation}
where the supremum in the second equality is taken after all arrays $\alpha_{i,j},1\le i,j\le k$ such that $\sum_{i,j=1}^k|\alpha_{i,j}|^2=1$. (As before, the result in
\eqref{HS} does not change regardless of whether one uses Tr${}_k$ or tr${}_k$ in defining the HS norm.) 
The first expression identifies with the second by simply splitting $C$ into matrix units. This form identifies $\mathbf P(\uA)$ acting on 
$(\mathbb M_k(\mathbb C),\langle\,\cdot\,,\,\cdot\,\rangle_\mathrm{HS})$ as an element in the tensor product of von Neumann algebras 
$\mathbb M_k(\mathbb C)\otimes\mathbb M_k(\mathbb C)^{\rm op}$ and $\mathbf P(\uA)$ acting on $(\mathbb C^k\otimes\mathbb C^k,\langle\,\cdot\,,\,\cdot\,\rangle)$ 
as an element of $\mathbb M_k(\mathbb C)\otimes\mathbb M_k(\mathbb C).$ In this case it follows that $\mathbf P(\uA)$ has the same norm 
regardless of whether it is viewed in one or in the other. 
It should also be noted that the spectral properties of $\mathbf P(\uA)$ do not change 
regardless of whether it is viewed as acting on $\mathbb C^k\otimes\mathbb C^k$, on $(\mathbb M_k(\mathbb C),\|\cdot\|_\mathrm{HS}),$ or on 
$(\mathbb M_k(\mathbb C),\|\cdot\|)$. However, the notion of positivity as an operator on the Hilbert space $(\mathbb M_k(\mathbb C),\|\cdot\|_\mathrm{HS})$ 
is different from the notion of positivity as an operator on the Banach space $(\mathbb M_k(\mathbb C),\|\cdot\|)$ endowed with the order determined by the cone of 
positive semidefinite matrices.

For every $d \in \N$, we define the bi-variate polynomial $\kappa_{\tau,d}$ as follows
\begin{align}
\label{eq:kappa}
\kappa_{\tau,d} (\uX,\uY) := \sum_{w \in \langle \uX \rangle_d } P_w(\uX)\otimes P_w^\star(\uY) \,.
\end{align}
If $\tau$ is {\em not faithful}, then the above notation $w\in\langle\uX\rangle_d$ should be understood as running only through
a subset of {\em linearly independent} monomials in the completion $L^2(\tau):=L^2(\mathbb C\langle\underline{X}\rangle,\tau)$ of 
$\mathbb C\langle\uX\rangle$ with respect to $\|P\|_2=\tau(P^\star P)^\frac12$ which span $\mathbb C\langle\underline{X}\rangle_d$ -- we do not ``repeat'' any summand.
Placing the adjoints on the first or on the second tensor in the above definition is a matter of choice. Indeed, as noted in Section \ref{ort}, 
$\{P^\star_v\}_{v\in\langle\uX\rangle_d}$ is an orthonormal basis as well, so that the correspondence between the elements of 
$\{P_w\}_{w\in\langle\uX\rangle_d}$ and $\{P_v^\star\}_{v\in\langle\uX\rangle_d}$ is unitary.
With the notations from Section \ref{ort},
\begin{align}
\sum_{w \in \langle \uX \rangle_d }P_w(\uX)\otimes P_w^\star(\uY)
&=\!\sum_{w\in\langle\uX\rangle_d}\left(\sum_{v_0\in\langle\uX\rangle_d}\!(U^*)_{w,v_0}P^\star_{v_0}(\uX)\right)
\!\otimes\!\left(\sum_{v\in\langle \uX \rangle_d }\!U_{w,v}P_v(\uY)\right)\nonumber\\
& = \!\sum_{v,v_0\in\langle\uX\rangle_d}\left(\sum_{w\in\langle\uX\rangle_d}U_{w,v}(U^*)_{w,v_0}\right)P_{v_0}^\star(\uX)\otimes P_v(\uY)
\nonumber\\
& = \!\sum_{v,v_0\in\langle\uX\rangle_d}\delta_{v=v_0}P_{v_0}^\star(\uX)\otimes P_v(\uY)=\!\sum_{v\in\langle\uX\rangle_d}P_{v}^\star(\uX)\otimes 
P_v(\uY).\label{symmetry}
\end{align}
(As shown in Remark \ref{tres-uno}, $U$ is symmetric, so that $U_{w,v}(U^*)_{w,v_0}=U_{v,w}(U^*)_{w,v_0}$.)

One may view $\kappa_{\tau,d}$ as a  reproducing kernel for the finite-dimensional space $L^2(\mathbb C\langle\underline{X}\rangle_d,\tau)$: we 
tautologically have
\begin{align*}
(\mathrm{Id}_{\mathbb C\langle\underline{X}\rangle}\otimes\tau)(\kappa_{\tau,d}(\uX,\uY)(1\otimes P(\uY))) & = \sum_{v\in\langle\uX\rangle_d}
(\mathrm{Id}_{\mathbb C\langle\underline{X}\rangle}\otimes\tau)(P_{v}(\uX)\otimes P^\star_v(\uY)P(\underline{Y}))\\
& = \sum_{v\in\langle\uX\rangle_d}P_{v}(\uX)\otimes\tau(P^\star_v(\uY)P(\underline{Y}))1\\
& = \sum_{v\in\langle\uX\rangle_d}\tau(P^\star_v(\uY)P(\underline{Y}))P_{v}(\uX)\otimes 1\\
&=  P(\underline{X})\otimes1,
\end{align*}
for any $P\in\mathbb C\langle\underline{X}\rangle_d$.
More remarkably, with the previously introduced Hilbert space $L^2(\tau):=L^2(\mathbb C\langle\underline{X}\rangle,\tau)$ as the completion 
of $\mathbb C\langle\uX\rangle$ with respect to $\|P\|_2=\tau(P^\star P)^\frac12$, by letting $d\to\infty$ in the above, we may write the $L^2$-limit
\begin{equation}\label{reproduction}
\lim_{d\to\infty}(\mathrm{Id}_{\mathbb C\langle\underline{X}\rangle}\otimes\tau)(\kappa_{\tau,d}(\uX,\uY)(1\otimes f(\underline{Y})))=
f(\underline{X})\otimes1,\quad f\in L^2(\tau).
\end{equation}
In this (very weak) sense, one may state that $\kappa_{\tau}(\uX,\uY)=\lim_{d\to\infty}\kappa_{\tau,d}(\uX,\uY)$ exists.
The reproducing kernel $\kappa_{\tau,d} (\uX,\uY)$ is called the {\em  noncommutative Christoffel polynomial} associated to $\tau$ and $d$.
Note that 
\begin{align*}
(\tau\otimes\tau)(\kappa_{\tau,d}(\uX,\uY)^\star\kappa_{\tau,d}(\uX,\uY)) & = \sum_{v,w\in\langle \uX \rangle_d }
(\tau\otimes\tau)(P_v^\star(\uX)P_w(\uX)\otimes P_w^\star(\uY)P_v(\uY))\\
& = \sum_{v,w\in\langle \uX \rangle_d }\tau(P_v^\star(\uX)P_w(\uX))\tau(P_v(\uY)P_w^\star(\uY))\\
& = \sum_{w\in\langle \uX \rangle_d }\tau(P_w^\star(\uX)P_w(\uX))\tau(P_w(\uY)P_w^\star(\uY))\\
&=  \sum_{w\in\langle \uX \rangle_d }\!1\cdot1=\bsigma(n,d), 
\end{align*}
by orthonormality of the $\{P_w \}_{w \in \langle \uX \rangle _d}$.

The bi-variate polynomial $\kappa_{\tau,d}$ is in fact a globally defined completely positive noncommutative kernel in the sense of \cite[Sections 2.3 and 2.4]{BMV}
(for exact definitions, which are not relevant for our curent paper, we refer to conditions (2.3), (2.4), and (2.10) in this reference); indeed, for any strictly
positive $k,k'\in\mathbb N$, $\kappa_{\tau,d}$ acts on $\Mbb_{k\times k'}(\mathbb C)$ via $\Mbb_k(\mathbb C)^n\times\Mbb_{k'}
(\mathbb C)^n \to \mathscr L(\Mbb_{k\times k'}(\mathbb C),\Mbb_{k\times k'}(\mathbb C))$, 
$(\uA,\uB)\!\mapsto\!\left[X\mapsto\sum_{w \in \langle \uX \rangle_d } P_w(\uA)XP_w^\star(\uB)\right]$
and thus $\kappa_{\tau,d}(\uA,\uA^\star)\ge0$, i.e. it is completely positive, for all $\uA\in\mathbb M_k(\mathbb C)^n.$
Note that we obtain a different action on $\Mbb_{k\times k'}(\mathbb C)$ for each choice of $\underline{A}$ and $\underline{B}$. 
We also note for future use that, since $P_1(\uX)=1$ (see \eqref{eq:orthonormal}), we have $\kappa_{\tau,d}(\uA,\uA)(C)=
\sum_{w\in\langle\uX\rangle_d}P_w(\uA)CP_w^\star(\uA)\succeq P_1(\uA)CP_1^\star(\uA)=C$ whenever evaluated on elements 
of a unital $C^*$-algebra in which the coordinates of $\uA$ are selfadjoint and $C\succeq0$.

For each $\uA = (A_1,\dots,A_n) \in \Mbb_k(\mathbb C)^n$, let us define $\Lambda_{\tau,d}(\uA) :=  \kappa_{\tau,d} (\uA,\uA^\star)(I_k)^{-1}$, 
as an operator in $\Mbb_k(\mathbb C)$. The function $\Lambda_{\tau,d}$ is called the {\em  noncommutative Christoffel function} associated to 
$\tau$ and $d$. We see from the above that $\Lambda_{\tau,d}$ is well-defined on $\mathcal A^n$ for any unital $C^*$-algebra $\mathcal A$, 
and in particular it is well defined on all of $\Mbb_k(\mathbb C)^n$, $k\in\mathbb N$, and 
$\Lambda_{\tau,d}(\uA)\preceq I_k,\uA\in \Mbb_k(\mathbb C)^n$.

We now prove the noncommutative analog of \cite[Theorem~3.1]{lasserre2019empirical}, which can be found in  e.g.~\cite{dunkl2014orthogonal} 
and~\cite{nevai1986geza}. 
We emphasize that the $\min$ appearing in \eqref{eq:optim} is a minimum of a matrix-valued quantity. 
\begin{theorem}
\label{th:christoffeloptim}
Let $\uA\in\Mbb_k(\mathbb C)^n$ be fixed, arbitrary. Then 
\begin{equation}
\label{eq:optim}
\displaystyle\Lambda_{\tau,d}(\uA)=\min
\left\lbrace\,({\rm Id}_{\Mbb_k(\mathbb C)}\otimes\tau)({\bf P}{\bf P}^\star)\,: {\bf P}\in\Mbb_k(\mathbb C\langle\uX\rangle_d),{\bf P}(\uA)(I_k)=I_k
\right\rbrace.
\end{equation}
\end{theorem}
It is remarkable that the set in the right-hand side of \eqref{eq:optim} has a minimum, given that the set of positive operators is not a lattice. 
Moreover, under the assumption that $\tau$ is bounded, it can be realized as the distribution of a tuple $\underline{a}_\tau$ of selfadjoints in a 
von Neumann algebra with respect to a normal faithful trace tr, and then $({\rm Id}_{\Mbb_k(\mathbb C)}\otimes\tau)({\bf P}{\bf P}^\star)=
({\rm Id}_{\Mbb_k(\mathbb C)}\otimes\text{tr})({\bf P}(\underline{a}_\tau){\bf P}(\underline{a}_\tau)^\star)$,
the conditional expectation of a positive element in a von Neumann algebra.
\begin{proof}[Proof of Theorem \ref{th:christoffeloptim}]
Let us define $c_w(\uA)=\Lambda_{\tau,d}(\uA)P_w(\uA)\in\mathbb M_k(\mathbb C)$, for all $ w\in\langle\uX\rangle_d$ and 
${\bf P}_d(\uX):=\sum_{w\in\langle\uX\rangle_d}c_w(\uA)\otimes P_w^\star(\uX)\in\Mbb_k(\mathbb C\langle\uX\rangle_d)$.
(If $\tau$ is not faithful, we again drop monomials from the set $\langle\uX\rangle_d$ so that it becomes a basis for $L^2(\tau,\mathbb C\langle\uX\rangle_d)$.) 
This polynomial ${\bf P}_d$ is feasible for \eqref{eq:optim} since one has
\[
{\bf P}_d(\uA)\!(I_k\!)\!=\!\!\!\!\!\!\sum_{w\in\langle\uX\rangle_d}\!\!\!\!\!\!c_w(\uA)\!I_k\!P_w(\uA)^\star\!\!=\!\Lambda_{\tau,d}(\uA)\!\!\!\!\!\!\sum_{w \in \langle\uX\rangle_d}\!\!\!
\!\!\!P_w(\uA)\!I_k\!P_w(\uA)^\star\!\!=\!\!\Lambda_{\tau\!,d}(\uA)\kappa_{\tau\!,d}(\uA,\uA^\star)(I_k\!)\!\!=\!\!I_k.
\]
In addition, one has (we use the op algebra structure on the  second tensor)
\begin{align*}
{\bf P}_d(\uX){\bf P}_d^\star(\uX) & = \sum_{v,w \in \langle \uX \rangle_d} c_w(A) c_v(A)^\star\otimes P_v(\uX)P_w^\star(\uX),
\end{align*}
so
\begin{align*}
({\rm Id}_{\Mbb_k(\mathbb C)}\otimes\tau) ({\bf P}_d(\uX){\bf P}_d^\star(\uX)) & = \sum_{w \in \langle \uX \rangle_d} c_w(A) c_w(A)^\star \\
& = \sum_{w \in \langle \uX \rangle_d} \Lambda_{\tau,d}(\uA)P_w(\uA) P_w(\uA)^\star\Lambda_{\tau,d}(\uA) \\
& = \Lambda_{\tau,d}(\uA) \kappa_{\tau,d} (\uA,\uA^\star)(I_k)\Lambda_{\tau,d}(\uA) = \Lambda_{\tau,d}(\uA)\,,
\end{align*}
proving that the set in the right hand-side of \eqref{eq:optim} contains $\Lambda_{\tau,d}(\uA)$.

To investigate the minimality of the element $\Lambda_{\tau,d}(\uA)$, recall the Hilbert $\mathbb M_k(\mathbb C)$-module structure on $\mathbb M_k
(\mathbb C)^{|\langle\uX\rangle_d|}=\mathbb M_k(\mathbb C)^{\bsigma(n,d)}$: $\displaystyle\langle (a_w)_w,(b_w)_w\rangle_d
=\sum_{w\in\langle\uX\rangle_d}a_wb_w^*$. Consider an arbitrary polynomial ${\bf P}(\uX)=
\displaystyle\sum_{w\in\langle\uX\rangle_d}a_w\otimes P_w(\uX)$ such that ${\bf P}(\uA)(I_k)=I_k$. By the definition of the 
$\mathbb M_k(\mathbb C)$-valued inner product, this last expression is equivalent to stating $\langle(a_w)_w,(P_w(\uA)^\star)_w\rangle_d=I_k$.
With the Hilbert module structure (see \cite[Chapter 14]{Vern}), we are guaranteed the positivity 
\begin{align}
&\begin{bmatrix}
\langle(a_w)_w,(a_w)_w\rangle_d&\langle(a_w)_w,(P_w(\uA)^\star)_w\rangle_d\\\langle(P_w(\uA)^\star)_w,(a_w)_w\rangle_d
&\langle(P_w(\uA)^\star)_w,(P_w(\uA)^\star)_w\rangle_d
\end{bmatrix}\label{Bi}\\&=\begin{bmatrix}
\langle(a_w)_w,(a_w)_w\rangle_d & I_k\\ I_k & \langle(P_w(\uA)^\star)_w,(P_w(\uA)^\star)_w\rangle_d
\end{bmatrix}\succeq0.\nonumber
\end{align}
As $\langle(P_w(\uA)^\star)_w,(P_w(\uA)^\star)_w\rangle_d^{-1}=\left[\sum P_w(\uA)^\star P_w(\uA)\right]^{-1}=\Lambda_{\tau,d}(\uA)$ (see 
\eqref{symmetry}) is known to exist, the positivity of the above matrix is equivalent, by considering the Schur complement \cite[Theorem 7.7.6]{horn}, or
\cite[Lemma 3.1]{Vern}, to the relation
$$
I_k\!\langle(P_w(\uA)^\star)_w,\!(P_w(\uA)^\star)_w\rangle_d^{-1}\!I_k\!\preceq\langle(a_w)_w,(a_w)_w\rangle_d
=\!\!\!\!\sum_{w\in\langle\uX\rangle_d}\!\!\!\!a_wa_w^\star
=({\rm Id}_{\Mbb_k(\mathbb C)}\otimes\tau)({\bf P}{\bf P}^\star\!),
$$
as claimed.
\end{proof}
We mention below some re-interpretations of Theorem 
\ref{th:christoffeloptim}; one remarkable point is that the normalization $\mathbf P(\uA)(I_k)=I_k$ is in fact quite arbitrary, even though useful.

\begin{remark}
\begin{trivlist}
\item[{\rm(1)}]
The positivity of the matrix in \eqref{Bi} imposes automatically the domination relation\!
$\langle(a_w\!)_w,\!(P_w(\uA)^\star)_w\rangle_d\Lambda_{\tau\!,d}(\uA)\langle(P_w(\uA)^\star)_w,\!(a_w\!)_w\rangle_d\!\!\preceq\!\!\langle(a_w
\!)_w,\!(a_w\!)_w\rangle_d$ for {\em all} vectors $(a_w)_w\in\mathbb M_k(\mathbb C)^{|\langle\uX\rangle_d|}$. Identifying the optimal solution 
$(a_w)_{w\in\langle\underline{X}\rangle_d}=(\Lambda_{\tau,d}(\uA)P_w^\star(\uA))_{w\in\langle\underline{X}\rangle_d}$ comes simply to identifying 
the element which transforms $\preceq$ into $=$ while not creating a null space through  conjugation with $\langle(a_w)_w,(P_w(\uA)^\star)_w
\rangle_d,$ i.e. while imposing that this element stays invertible $($of course, the most convenient choice to impose is 
$\langle(a_w)_w,(P_w^\star(\uA))_w\rangle_d=I_k)$. 
\item[\, {\rm(2)}] Under the assumption that $\langle(a_w)_w,\!(a_w)_w\rangle_d$ is invertible, positivity of the leftmost matrix in relation \eqref{Bi} is 
equivalent, via the Schur complement, to
$$
\langle(P_w(\uA)^\star)_w,\!(a_w)_w\rangle_d  
\langle(a_w)_w,\!(a_w)_w\rangle_d^{-1}
\langle(a_w)_w,\!(P_w(\uA)^\star)_w\rangle_d\!  \preceq \!
\langle(P_w(\uA)^\star)_w,\!(P_w(\uA)^\star)_w\rangle_d.
$$
After substituting $\langle(P_w(\uA)^\star)_w,(P_w(\uA)^\star)_w\rangle_d =  \kappa_{\tau,d}(\uA,\uA^\star)(I_k) $, $\langle(a_w)_w,\!(a_w)_w\rangle_d = ({\rm Id}_{\mathbb M_k(\mathbb C)}\otimes\tau)(\mathbf{P}(\uX){\bf P}^\star(\uX))$, and $\langle(a_w)_w,(P_w(\uA)^\star)_w\rangle_d = \mathbf{P}(\uA)(I_k)$ we obtain
$$
\mathbf{P}(\uA)(I_k)^\star({\rm Id}_{\mathbb M_k(\mathbb C)}\otimes\tau)(\mathbf{P}(\uX){\bf P}^\star(\uX))^{-1}\mathbf{P}(\uA)(I_k)\preceq
\kappa_{\tau,d}(\uA,\uA^\star)(I_k).
$$
By employing an arbitrarily small perturbation, the above holds also if $\langle(a_w\!)_w,\!(a_w\!)_w\rangle_d$ is not invertible, in the sense that the 
left-hand side stays bounded by $\kappa_{\tau,d}(\uA,\uA^\star)(I_k)$ when the perturbation tends to zero. This relation can be viewed as the 
noncommutative equivalent of the inequality
$|P(\underline{x})|^2\leq\kappa_{\mu,d}(\underline{x},\underline{x})\int|P(t)|^2\,{\rm d}\mu(t)$, $\underline{x}\in\mathbb R^n$. Moreover, Equation 
\eqref{eq:optim} is then rewritten as
\begin{eqnarray}
\lefteqn{\kappa_{\tau,d}(\uA,\uA^\star)(I_k)}\nonumber\\
& & =\max\{\mathbf P(\uA)(I_k)^\star\mathbf P(\uA)(I_k)\colon{\bf P}\in\Mbb_k(\mathbb C\langle\uX\rangle_d),
({\rm Id}_{\Mbb_k(\mathbb C)}\otimes\tau)({\bf P}{\bf P}^\star)=I_k\}.\nonumber 
\end{eqnarray}
Indeed, the Schur complement in the left-hand side of the previous domination relation remains invariant under multiplication of $\mathbf P$ to the left 
with $G\otimes 1$ for any $G\in GL_k(\mathbb C)$.
\item[\, {\rm(3)}] One may of course extend the definition of $\Lambda_{\tau,d}$ to be simply $\Lambda_{\tau,d}(\uA,\uB)(C)=\left(\kappa_{\tau,d}(\uA,
\uB^\star)(C)\right)^{-1}$ for all $\uA,\uB,C$ for which the inverse is well-defined. While that does not seem to be a useful extension, we would like to 
point out that $\Lambda_{\tau,d}(\uA)(C):=\left(\kappa_{\tau,d}(\uA,\uA^\star)(C)\right)^{-1}$ {\em is} well defined for any $C\succ0$, and moreover 
$\Lambda_{\tau,d}(\uA)(C)=\min\{({\rm Id}_{\Mbb_k(\mathbb C)}\otimes\tau)({\bf P}{\bf P}^\star)\colon{\bf P}\in\Mbb_k(\mathbb C\langle\uX
\rangle_d),\mathbf P(\uA)(C^{1/2})=I_k\}$.
\item[\, {\rm(4)}]  A brief inspection of the proof of Theorem \ref{th:christoffeloptim} shows that the result holds as well when one replaces $\mathbb 
M_k(\mathbb C)$ with any finite factor $\mathcal A$ $($and in particular, $\mathbb M_k(\mathbb C\langle\uX\rangle)$ with $\mathcal A\otimes\mathbb 
C\langle\uX\rangle)$. The essential element of the proof is that $\mathcal A^{\bsigma(n,d)}$ accepts the $\mathcal A$-Hilbert module structure 
$\langle(a_w)_w,(b_w)_w\rangle_d=\sum a_wb_w^\star$. 
\item[\, {\rm(5)}] We record for future use the expression of the minimizer in \eqref{eq:optim}: for any operator algebra $\mathcal A$ on a Hilbert space,
\begin{equation}
{\bf P}_d(\uX)=(\Lambda_{\tau,d}(\underline{a})\otimes 1)
\sum_{w\in\langle\uX\rangle_d}P_w(\underline{a})\otimes P_w^\star(\uX)\in\mathcal A\otimes\mathbb C\langle\uX\rangle_d,
\quad\underline{a}\in(\mathcal A^{\rm sa})^n.
\end{equation}
In particular, for $\mathcal A\!=\!\Mbb_k(\mathbb C)$, we obtain the formula in Theorem \ref{th:christoffeloptim}.
The optimization problem stated in \eqref{eq:optim} is the noncommutative analog of  \cite[(3.2)]{lasserre2019empirical}.
\item[\, {\rm(6)}] The view of the evaluation $\mathbf P(\uA)(I_k)$ as a matrix-valued inner product has been essential in the proof of Theorem \ref{th:christoffeloptim}. 
Another re-interpretation of this evaluation might be more intuitive, namely as a matrix product. Tautologically, for $\mathbf P(\uX)=
\sum_{w\in\langle\uX\rangle_d}c_w\otimes\uX^w$, we have
\begin{equation}\label{vec}
\begin{array}{r}
\mathbf P(\uA)(I_k)=\mathfrak c\mathfrak A=\begin{bmatrix} c_0 & c_{X_1} & \cdots & c_{X_n^{d-1}X_{n-1}} & c_{X_n^d}\end{bmatrix}\\
\ \\
\ \\
\ \\
\ 
\end{array}\!\!\!
\begin{bmatrix} 1 \\ A_1 \\ \vdots \\ A_n^{d-1}A_{n-1}\\ A_n^d\end{bmatrix}.
\end{equation}
The matrix $\mathfrak c$ is the $k\times\bsigma(n,d)k$ complex matrix formed by the row of $k\times k$  blocks $c_w,w\in\langle\uX\rangle_d$. 
Similarly, $\mathfrak A\in\mathbb M_{k\bsigma(n,d)\times k}(\mathbb C)$ is the column matrix having as entries the $k\times k$ blocks $\uA^w,w\in
\langle\uX\rangle_d$. The matrix entries of the row vector $\mathfrak c$ are ordered lexicographically from left to  right, and the matrix entries $\uA^w$ 
of $\mathfrak A$ are ordered lexicographically from top to bottom. It is clear that one may amplify $\mathfrak c$ to a matrix $\mathsf c\in
\mathbb M_{k\bsigma(n,d)}(\mathbb C)$ by adding rows of zeros under $\mathfrak c$ and similarly one may add columns of zeros to the right of 
$\mathfrak A$ in order to form a square matrix $\mathsf A$ of the same size. Then
\begin{equation}\label{mat}
\mathbf P(\uA)(I_k)=\left[\mathsf c\mathsf A\right]_{1,1}
\end{equation}
is the $(1,1)$ entry of the product $\mathsf c\mathsf A\in\mathbb M_{\bsigma(n,d)}(\mathbb M_{k}(\mathbb C))$, or as the upper left $k\times k$ block
of $\mathbb M_{k\bsigma(n,d)}(\mathbb C)$.
\end{trivlist}
\   \label{Remark33}
\end{remark}

Let $D_d(\tau)$ be the lower triangular matrix whose rows are the coefficients of the polynomials $P_w$ ordered by $\leq_{\text{gl}}$. 
Then one can prove that the moment matrix of $\tau$ satisfies $\M_d(\tau)^{-1} =  D_d(\tau)^T D_d(\tau)$. 
As a consequence we obtain the following:
\begin{proposition}
\label{prop:inverse}
Recall that $\mathbf W_d(\uX)$ is the vector of all words of 
$\langle\uX\rangle_d$.
Let $\tau\colon\mathbb C\langle\uX\rangle\to\mathbb C$ be a faithful tracial state.
For any $n$-tuples of selfadjont operators $\uA,\uB$ and any bounded operator $C$ on a Hilbert space, one has 
\[
\kappa_{\tau,d}(\uA,\uB)(C)=\mathbf W_d(\uA)^TD_d(\tau)^TCD_d(\tau)\mathbf W_d^\star(\uB) .\] 
In particular, $\kappa_{\tau,d}(\uA,\uB)(1) = 
\mathbf W_d(\uA)^T \M_d(\tau)^{-1} \mathbf W^\star_d(\uB)$.
\end{proposition}
\subsection{Noncommutative Siciak extremal functions}\label{Sec34}

Our next objective is to investigate the asymptotic properties of $\Lambda_{\tau,d}$ as $d\to\infty$ when evaluated on various domains.
In order to perform this analysis, we are forced to place some restrictions both on the nature of $\tau$ and the 
domains on which we evaluate $\Lambda_{\tau,d}$. This necessity should probably not be too surprising, as even in the
much more studied classical context, a full description of this asymptotic behavior is not known. We are unaware of any
previous work in the noncommutative context.

%
%
%
For a fixed positive, bounded trace-state $\tau\colon\mathbb C\langle\underline{X}\rangle\to\mathbb C$,
we construct via the GNS construction \cite[Chapter I, Definition 9.15]{Takesaki01} the finite von Neumann algebra 
generated by the $n$ selfadjoint elements (the images of $X_1,\dots,X_n\in\mathbb C\langle\underline{X}\rangle$ via the GNS representation) 
whose joint law equals $\tau$ \cite[Proposition 5.2.14(d)]{anderson2010introduction}. We denote it by $W^*(\tau)$.

In the following, the faithfulness of $\tau$ will often be assumed for convenience. We will however try to indicate what changes occur
in various arguments if $\tau$ is not assumed to be faithful.

Our methods parallel to a significant extent the classical ones \cite{GZ,Klimek,ST}. The way to investigate the asymptotic behavior of the 
Christoffel-Darboux kernel in classical analysis involves the so-called Siciak extremal function. Let us start by proposing free noncommutative versions 
of this object. We consider $\underline{a}_\tau$ to be a tuple of selfadjoint variables in a von Neumann algebra $\mathcal M$ which is distributed 
according to $\tau$ with respect to the normal faithful tracial state tr on $\mathcal M$ (in particular, $W^*(\tau)$ is isomorphic as a von Neumann algebra to
the von Neumann subalgebra of $\mathcal M$ generated by the $n$-tuple $\underline{a}_\tau$). For each $d,k\in\mathbb N$, $\uA\in\Mbb_k(\mathbb C)^n$, we 
define
\begin{align}\label{330}
\Phi_{\tau,d}^2(\uA)=\sup\left\{{{\rm tr}_k(\mathbf P(\uA)(I_k)^\star\mathbf P(\uA)(I_k))}\colon\mathbf P\in\mathbb M_k(\mathbb C\langle\uX
\rangle_d),\|\mathbf P(\underline{a}_\tau)\|\le1\right\},\\
\Phi_{\tau,d}^\infty(\uA)=\sup\left\{\|\mathbf P(\uA)(I_k)\|^2\colon\mathbf P\in\mathbb M_k(\mathbb C\langle\uX\rangle_d),
\|\mathbf P(\underline{a}_\tau)\|\le1\right\}.\label{331}
\end{align}
Recall that $\mathbf P(\underline{a}_\tau)\in\mathbb M_k(W^*(\tau))$, so that the norm $\|\mathbf P(\underline{a}_\tau)\|$ is the operator norm on 
$\mathbb M_k(W^*(\tau))\simeq\mathbb M_k(\mathbb C)\otimes W^*(\tau)\subseteq\mathbb M_k(\mathbb C)\otimes\mathcal M$. Note that in 
\eqref{330} we have defined $\Phi^2_{\tau,d}$ as the trace of the product of the two matrices obtained by applying the linear operator $\mathbf P(\uA)$ 
to $I_k$. However, we would have obtained precisely the same result had we defined it as $\mathrm{tr}_k\left([\mathbf P^\star(\uA^\star)\mathbf 
P(\uA)](I_k)\right)$. Clearly $\Phi_{\tau,d}^2(\uA)\leq\Phi_{\tau,d}^\infty(\uA)$ for all $\uA$. Let us list next various properties of these functions which 
will be of use to us later.
\subsubsection{Plurisubharmonicity}\label{sspsh}
Both $\Phi_{\tau,d}^2(\uA)$ and $\Phi_{\tau,d}^\infty(\uA)$ are suprema of norms of analytic polynomials, hence entire plurisubharmonic functions 
(see Section \ref{PSHintro}) in the classical sense when viewed as functions on the complex Euclidean space $\mathbb C^{nk^2}\simeq(\mathbb M_k(\mathbb C))^n$. Moreover, 
$\|\mathbf P(\uA)(I_k)\|=\lim_{m\to\infty}{\rm tr}_k\left([\mathbf P(\uA)(I_k)^\star\mathbf P(\uA)(I_k)]^{m}\right)^{1/2m}$.

\subsubsection{Monotonicity}\label{monot} If $\underline{X}^{'}\subseteq\underline{X}$ is a subset, then 
$\Phi^\bullet_{\tau|_{\mathbb C\langle\underline{X}^{'}\rangle},d}(\uA^{'})\leq\Phi^\bullet_{\tau,d}(\uA^{'},\uA\setminus\uA^{'})$, for  
$\bullet\in\{2,\infty\},\underline{A}=(\uA^{'},\uA\setminus\uA^{'})\in\mathbb M_k(\mathbb C\langle\uX\rangle)^n,d\in\mathbb N$. 
This follows easily by noting that a polynomial $\mathbf P\in\mathbb C\langle\underline{X}^{'}\rangle_d$ belongs tautologically
to $\mathbb C\langle\underline{X}\rangle_d$ as well, and that the embedding $W^*(\tau|_{\mathbb C\langle\underline{X}^{'}\rangle})
\hookrightarrow W^*(\tau)$ is completely isometric.

\subsubsection{Compactness}\label{comp}
The set of polynomials $\left\{\mathbf P\!\in\mathbb M_k(\mathbb C\langle\uX\rangle_d)\colon\!
\|\mathbf P(\underline{a}_\tau)\|\le1\right\}$ is compact, in the sense that for any given basis $\{b_1,\dots,b_{\bsigma(n,d)}\}$ of 
$\mathbb C\langle\uX\rangle_d$, the set $\{(c_1,\dots,c_{\bsigma(n,d)})\in\mathbb M_k(\mathbb C)^{\bsigma(n,d)}\colon\|\sum c_j\otimes
b_j(\underline{a}_\tau)\|\le1\}$ is compact\footnote{It is important to remember that we work here under the hypothesis of faithfulness of  $\tau$. If, for 
instance, $\underline{a}_\tau$ were to satisfy some algebraic relation -- that is, if there were a $P\in\mathbb C\langle\uX\rangle$ such that 
$P(\underline{a}_\tau)=0$ -- 
the bound in \eqref{bound} would easily fail. In that case, one would need to work on a space of dimension strictly smaller than $\bsigma(n,d)$ in order for
this same result to hold -- see \cite{beckermann2020perturbations}. The 
reader may keep in mind the most extreme case, when all coordinates of $\underline{a}_\tau$ are complex multiples of the algebra's unit.} in the 
Euclidean space $\mathbb M_k(\mathbb C)^{\bsigma(n,d)}$. Indeed, all bases are equivalent via an invertible, finite-dimensional linear transformation, 
so it is enough to show this for $\{b_1,\dots,b_{\bsigma(n,d)}\}=\{P_w\}_{w\in\langle\uX\rangle_d},$ the basis of orthonormal polynomials. In that case, 
$\sum_wc_wc_w^\star=\sum_{v,w}c_wc_v^\star\otimes\text{tr}(P_v^\star(\underline{a}_\tau)P_w(\underline{a}_\tau))=
(\text{Id}_{\mathbb M_k(\mathbb C)}\otimes\text{tr})(\sum_{v,w}c_wc_v^\star\otimes P_v^\star(\underline{a}_\tau)
P_w(\underline{a}_\tau))=(\text{Id}_{\mathbb M_k(\mathbb C)}\otimes\text{tr})((\sum_{w}c_w\otimes P_w(\underline{a}_\tau))(\sum_{w}c_w\otimes 
P_w(\underline{a}_\tau))^\star)=(\text{Id}_{\mathbb M_k(\mathbb C)}\otimes\text{tr})(\mathbf P(\underline{a}_\tau)
\mathbf P(\underline{a}_\tau)^\star)$ for any $\mathbf P(\underline{X})=\sum_{w}c_w\otimes P_w(\underline{X})\in
\mathbb M_k(\mathbb C\langle\uX\rangle_d).$ Thus, the condition $\|\mathbf P(\underline{a}_\tau)\|\le1$ implies
\begin{equation}\label{bound}
1\ge\|\mathbf P(\underline{a}_\tau)\|^2\!=\|\mathbf P(\underline{a}_\tau)\mathbf P(\underline{a}_\tau)^\star\|\ge
\|(\text{Id}_{\mathbb M_k(\mathbb C)}\otimes\text{tr})(\mathbf P(\underline{a}_\tau)\mathbf P(\underline{a}_\tau)^\star)\|=\left\|\sum_{w\in\langle
\uX\rangle_d}\!c_wc_w^\star\right\|
\end{equation}
so that $\|c_w\|^2\le1$ for all ${w\in\langle\uX\rangle_d}$. This, together with the continuity of the functions involved, guarantees that the suprema in 
both \eqref{330} and \eqref{331} are in fact maxima.

\subsubsection{Matrix convexity} 
Recall that a set $K$ of matrices over $\C$ consists of subsets $K_k \subset \Mbb_k(\C)$, for all $k \in \N$.
Such a set $K$ is called {\em matrix convex} \cite[Section 3]{EW} if for all $A \in K_k$ and $B \in  K_{k'}$, one has $A\oplus B \in K_{k+k'}$, 
and for all $A \in K_k$ and $\alpha \in \Mbb_{k,k'}(\C)$ with $\alpha^\star \alpha = I_{k'}$, one has $\alpha^\star A \alpha \in K_k$.
The set $\left\{\mathbf P\in\mathbb M_k(\mathbb C\langle\uX\rangle_d)\colon
\|\mathbf P(\underline{a}_\tau)\|\le1\right\}$ is convex, and  in fact it is also matrix convex, which is verified as follows: if 
$\mathbf P\in\mathbb M_k(\mathbb C\langle\uX\rangle),\mathbf P'\in\mathbb M_{k'}(\mathbb C\langle\uX\rangle)$ satisfy the desired norm 
inequality, then $\|\mathbf P(\underline{a}_\tau)\oplus\mathbf P'(\underline{a}_\tau)\|\!\le\!1$ as well in 
$\mathbb M_{k+k'}(\mathbb C)\otimes\mathcal M$ (to be clear, if $\mathbf P(\uX)\!=\!\sum_w\!c_w\otimes P_w(\uX),\mathbf P'(\uX)=
\sum_wc_w'\otimes P_w(\uX)$, we view $\mathbf P(\underline{X})\oplus\mathbf P'(\underline{X})=\sum_w(c_w\oplus c_w')\otimes P_w(\uX)\in
\mathbb M_{k+k'}(\mathbb C\langle\uX\rangle)$). If $\alpha\in\mathbb M_{k\times k'}(\mathbb C)$ satisfies $\alpha^\star\alpha=I_{k'}$ (so implicitly 
$k'\le k$), the{\o}n 
$$
\left\|\sum_w(\alpha^\star c_w\alpha)\otimes P_w(\underline{a}_\tau)\right\|=\left\|(\alpha^\star\otimes1)\mathbf P(\underline{a}_\tau)(\alpha
\otimes1)\right\|\le\|\alpha\|^2\left\|\mathbf P(\underline{a}_\tau)\right\|\le1.
$$
Obviously, this implies that the maximizers providing equality in  \eqref{330} and \eqref{331} are extremal points in these (classically) convex sets.

\subsubsection{Invariance under left unitary action of the set of maximality}\label{ssunitary}
Equally relevant, let us observe that if $\mathbf P\in\mathbb M_k(\mathbb C\langle\underline{X}\rangle_d)$ is a polynomial such that one of the equalities
$\Phi_{\tau,d}^2(\uA)={\rm tr}_k(\mathbf P(\uA)(I_k)^\star\mathbf P(\uA)(I_k))$ or $\Phi_{\tau,d}^\infty(\uA)=\|\mathbf P(\uA)(I_k)\|^2$ takes place, 
then the corresponding equality still holds for the polynomial $(U\otimes1)\mathbf P$ for any $U\in\mathbb M_k(\mathbb C)$ satisfying $UU^\star=I_k$,
i.e. for any unitary matrix $U$. Indeed, this is straightforward since ${\rm tr}_k([(U\otimes1)\mathbf P](\uA)(I_k)^\star[(U\otimes1)\mathbf P](\uA)
(I_k))={\rm tr}_k\left(\,\left(\sum_w(Uc_w)\uA^w\right)^\star\ \left(\sum_v (Uc_v)\uA^v\right)\,\right)={\rm tr}_k\left(\,\left(\sum_w (\uA^w)^\star
c_w^\star\right)\ U^\star U\ \left(\sum_v c_v\uA^v\right)\,\right)={\rm tr}_k(\mathbf P(\uA)(I_k)^\star\mathbf P(\uA)(I_k))$, 
$\|[(U\otimes1\!)\mathbf P](\uA)(I_k)\|=\left\|\sum_w Uc_w\uA^w\right\|=\left\|U(\sum_w c_w\uA^w)\right\|\!=\left\|\sum_w c_w\uA^w\right\|
=\|\mathbf P(\uA)(I_k)\|$. Moreover, we have that if $\|\mathbf P(\underline{a}_\tau)\|\le1$, then  
$\|[(U\otimes1]\mathbf P)(\underline{a}_\tau)\|\le1$ as well. This is simply because $U\otimes 1$ is unitary in $\mathbb M_k(\mathbb C)\otimes 
W^*(\tau)$ whenever $U$ is in $\mathbb M_k(\mathbb C)$ and $\|[(U\otimes1)\mathbf P](\underline{a}_\tau)\|=
\|(U\otimes1)(\mathbf P(\underline{a}_\tau))\|=\|\mathbf P(\underline{a}_\tau)\|$. Thus, the two sets on which maxima are achieved in 
\eqref{330}, \eqref{331} are invariant under the left action of the unitary group of the $k$ by $k$ complex matrices. As an immediate consequence of the 
polar decomposition of operators, we may assume without loss of generality that the maximizer $\mathbf P$ satisfies the condition $\mathbf P(\uA)(I_k)
\succeq0$: indeed, if $\mathbf P(\uA)(I_k)=U|\mathbf P(\uA)(I_k)|$, then we replace $\mathbf P$ by $(U^\star\otimes1)\mathbf P$. Unfortunately that 
does not mean $\mathbf P\ge0$.

With the same methods, one shows that $\Phi^\bullet_{\tau,d}(\uA)\le\Phi^\bullet_{\tau,d}(U\uA U^\star),\bullet\in\{2,\infty\},$ for all unitary 
$k\times k$  complex matrices $U$, so that $\Phi^\bullet_{\tau, d}$ is constant on unitary orbits. Indeed, picking $\mathbf P(\underline{X})=
\sum_{w\in\langle\underline{X}\rangle_d}c_w\otimes\underline{X}^w$ for which the maximum $\Phi^\bullet_{\tau,d}(\underline{A})=
\|\mathbf P(\underline{A})(I_k)\|^2_\bullet$ is reached, one has 
\begin{eqnarray*}
\mathbf Q(U\underline{A}U^\star)(I_k)\mathbf Q(U\underline{A}U^\star)(I_k)^\star\!\!
& = & \!\!\sum_{v\in\langle\underline{X}\rangle_d}(c_vU^\star)U\underline{A}^vU^\star
\left(\sum_{w\in\langle\underline{X}\rangle_d}\!(c_wU^\star)U\underline{A}^wU^\star\right)^\star\\
& = & \!\!\sum_{v,w\in\langle\underline{X}\rangle_d}c_v(U^\star U)\underline{A}^v\left(\underline{A}^w\right)^\star(U^\star U)c_w^\star\\
& = & \!\! \mathbf P(\underline{A})(I_k)\mathbf P(\underline{A})(I_k)^\star,
\end{eqnarray*}
where $\mathbf Q(\underline{X})=\mathbf P(\underline{X})[U^\star\otimes1]$. As $1=\|\mathbf P(\underline{a}_\tau)\|=
\|\mathbf P(\underline{a}_\tau)[U^\star\otimes1]\|=\|\mathbf Q(\underline{a}_\tau)\|$, it follows that $\Phi_{\tau,d}^\bullet(U\uA U^\star)
\ge\|\mathbf Q(U\underline{A}U^\star)(I_k)\|_\bullet^2=\|\mathbf P(\underline{A})(I_k)\|_\bullet^2=\Phi_{\tau,d}^\bullet(\uA)$, for any
$k\times k$ unitary matrix $U$.

The proofs of the two statements above can be put together to yield a slightly stronger technical result: 
if $\mathbf P(\uX)=\displaystyle\sum_{w\in\langle\uX\rangle_d}c_w\otimes\underline{X}^w$ and $V,W$ are $k\times k$ unitaries, then 
$$
V\sum_{w\in\langle\uX\rangle_d}c_w\underline{A}^wW=\sum_{w\in\langle\uX\rangle_d}Vc_wW(W^\star\underline{A}W)^w=[(V\otimes1)\mathbf P(W\otimes1)](W^\star
\underline{A}W)(I_k);
$$
since $\|[(V\otimes1)\mathbf P(W\otimes1)](\underline{a}_\tau)\|=\|\mathbf P(\underline{a}_\tau)\|$, it follows that if $\Phi^\bullet_{\tau,d}(\uA)=\|\mathbf P(\uA)(I_k)\|_\bullet^2$
is achieved on $\mathbf P$, then $\Phi^\bullet_{\tau,d}(\uA)=\Phi^\bullet_{\tau,d}(W^\star\!\uA W)\ge\|[(V\otimes1)\mathbf P(W\otimes1)](W^\star\!
\underline{A}W)(I_k)\|_\bullet^2=\|\mathbf P(\underline{A})(I_k)\|^2_\bullet=\Phi^\bullet_{\tau,d}(\uA)$, that is,
\begin{equation}
\Phi^\bullet_{\tau,d}(\uA)=\left\|V\sum_{w\in\langle\uX\rangle_d}c_w\underline{A}^wW\right\|_\bullet^2=\|[(V\otimes1)\mathbf P(W\otimes1)](W^\star
\underline{A}W)(I_k)\|_\bullet^2.
\end{equation}
In particular, $\Phi^\bullet_{\tau,d}(\uA)=\Phi^\bullet_{\tau,d}(W^\star
\underline{A}W)$ can be achieved at an element that is diagonal in whichever basis it is desired.

\subsubsection{Domination}\label{ssminimal}
 The condition $\|\mathbf P(\underline{a}_\tau)\|\le1$ in relations \eqref{330} -- \eqref{331} is equivalent to either of 
$\mathbf P(\underline{a}_\tau)^\star\mathbf P(\underline{a}_\tau)\preceq I_k\otimes1$, $\mathbf P(\underline{a}_\tau)
\mathbf P(\underline{a}_\tau)^\star\preceq I_k\otimes1$ in $\mathbb M_k(\mathbb C)\otimes W^*(\tau)$. If there exists a $0\prec C\preceq I_k$ such 
that, say, $\mathbf P(\underline{a}_\tau)\mathbf P(\underline{a}_\tau)^\star\preceq C\otimes1$, then $(C^{-1/2}\otimes1)
\mathbf P(\underline{a}_\tau)\mathbf P(\underline{a}_\tau)^\star(C^{-1/2}\otimes1)\preceq I_k\otimes1,$ ensuring that the polynomial
$(C^{-1/2}\otimes1)\mathbf P(\underline{X})$ belongs to the acceptable set in the right hand side of \eqref{330} and \eqref{331}.
Then $((C^{-1/2}\otimes1)\mathbf P(\underline{A}))(I_k)=C^{-1/2}\mathbf P(\underline{A})(I_k)$, so that 
${\rm tr}_k(\mathbf P(\uA)(I_k)^\star C^{-1}\mathbf P(\uA)(I_k))\ge{\rm tr}_k(\mathbf P(\uA)(I_k)^\star\mathbf P(\uA)(I_k)),
\|\mathbf P(\uA)(I_k)^\star C^{-1}\mathbf P(\uA)(I_k)\|\!\ge\!\|\mathbf P(\uA)(I_k)^\star\mathbf P(\uA)(I_k)\|.$ (Since $C^{-1}\succneqq I_k$,
equality in the previous inequality of traces above can only take place if $\ker\mathbf P(\uA)(I_k)^\star\mathbf P(\uA)(I_k)\neq\{0\}$.)
This makes it clear that without loss of generality we may assume in both \eqref{330} and \eqref{331} that there is no element $C\precneqq I_k$ 
which is not a projection such that $\mathbf P(\underline{a}_\tau)\mathbf P(\underline{a}_\tau)^\star\preceq C\otimes1$.

However, we note that one may obtain a ``minimal maximizer'': assume that $\Phi^\infty_{\tau,d}(\uA)$ is achieved at $\mathbf P_{\uA}(\uX)
=\sum_{w\in\langle\uX\rangle_d}c_w\otimes\uX^w$. Then there exists a rank-one projection $p\in\mathbb M_k(\mathbb C)$ such that 
$\|\mathbf P_{\uA}(\uA)(I_k)\|^2=\|p\mathbf P_{\uA}(\uA)(I_k)\|^2$ (one simply picks $p$ to be the projection onto a norm-one vector 
$\xi\in\mathbb C^k$ such that $\|\mathbf P_{\uA}(\uA)(I_k)\|^2=\langle\mathbf P_{\uA}(\uA)(I_k)\mathbf P_{\uA}(\uA)(I_k)^\star\xi,\xi\rangle$). 
Since $\mathbf P_{\uA}(\uA)(I_k)=\sum_{w\in\langle\uX\rangle_d}c_w\uA^w$ it follows that
\begin{align*}
\Phi^\infty_{\tau,d}(\uA)&=\|p\mathbf P_{\uA}(\uA)(I_k)\|^2
=\left\|p\left(\sum_{w\in\langle\uX\rangle_d}c_w\uA^w\right)\left(\sum_{w\in\langle\uX\rangle_d}c_w\uA^w\right)^\star p\right\|\\
&=\left\|\left(\sum_{w\in\langle\uX\rangle_d}pc_w\uA^w\right)\left(\sum_{w\in\langle\uX\rangle_d}pc_w\uA^w\right)^\star\right\|=\left\|[(p\otimes1)
\mathbf P_{\uA}](\uA)(I_k)\right\|^2
\end{align*}
and of course $[(p\otimes1)\mathbf P_{\uA}](\uX)=\sum_{w\in\langle\uX\rangle_d}(pc_w)\otimes\uX^w$, so 
$\|[(p\otimes1)\mathbf P_{\uA}](\underline{a}_\tau)\|=1$, with $[(p\otimes1)\mathbf P_{\uA}](\underline{a}_\tau)[(p\otimes1)\mathbf P_{\uA}]
(\underline{a}_\tau)^\star\preceq p\otimes1$. On the other hand, for $\Phi_{\tau,d}^2$ it is clear that the orthogonal complement of the kernel of 
$\mathbf P_{\uA}(\uA)(I_k)^\star$ is the biggest projection with which we can multiply $\mathbf P_{\uA}$ on the left.

\subsubsection{Linear change of coordinates}\label{coordchange}
Assume that $\mathbf u\in GL_n(\mathbb R)$. We consider the $n$-tuple of selfadjoint operators $\underline{\gimel}_\theta=\mathbf u\underline{a}_\tau$. 
That is, $\underline{\gimel}_\theta=({\gimel}_1,{\gimel}_2,\dots,{\gimel}_n)$ is a tuple of selfadjoint variables such that $\gimel_j=\sum_{k=1}^n\mathbf u_{jk}a_k$,
so that the law $\theta\colon\mathbb C\langle\underline{X}\rangle\to\mathbb C$ of the $n$-tuple $\underline{\gimel}_\theta$ is given by $\theta
(P(X_1,X_2,\dots,X_n))=\tau(P(\sum_{k=1}^n\mathbf u_{1k}X_k,\sum_{k=1}^n\mathbf u_{2k}X_k,\dots,\allowbreak\sum_{k=1}^n\mathbf u_{nk}X_k))$. It is obvious 
that $W^*(\theta)\simeq W^*(\tau)$ as von Neumann algebras (in fact it is equally obvious the case for any invertible -- under composition -- nc map 
$\mathbf f$ defined on an nc neighborhood of $\underline{a}_\tau$ so that $\mathbf f(\underline{a}_\tau)$ is a tuple of selfadjoints). One can identify 
$\mathbb C\langle\uX\rangle$ with the Fock space of the $n$-dimensional Hilbert space $\mathbb C^n$ via the identification of $1$ with the vacuum vector
$\Omega$ and $X_j$ with $e_j,$ the $j^{\rm th}$ canonical basis  vector of $\mathbb C^n$. The product $X_iX_j$ corresponds to $e_i\otimes e_j$ etc. With 
this view, one sees immediately that the effect of $\mathbf u$ on $\underline{a}_\tau$ extends to all of $\mathbb C\langle\underline{a}_\tau\rangle$ 
by acting on $(\mathbb C^n)^{\otimes p}$ as $\mathbf u^{\otimes p}$. Thus, action on the subspace $\mathbb C\Omega\oplus\bigoplus_{j=1}^d
(\mathbb C^n)^{\otimes j}$ is done via ${\rm Id}_{\mathbb C\Omega}\oplus\bigoplus_{j=1}^d\mathbf u^{\otimes j}$. Obviously, this assigns (easily 
computable but long) coefficients ${\bf u}_{w,v}\in\mathbb R$ in the block-matrix expression of the operator: 
$\mathbf u^{\otimes j}e_w=\sum_{|v|=j}({\bf u}^{\otimes j})_{w,v}e_v.$ Given a matrix-valued polynomial $\mathbf P(\uX)=
\sum_{w\in\langle\uX\rangle_d}c_w\otimes \uX^w$, the change of variable $\underline{\gimel}_\theta=\mathbf u\underline{a}_\tau$ translates into 
\begin{align*}
\mathbf P(\underline{\gimel}_\theta)&=\mathbf P(\mathbf u\underline{a}_\tau)=\!\sum_{w\in\langle\uX\rangle_d}\!c_w\otimes (\mathbf u\underline{a}_\tau)^w\\
&=c_0\otimes 1+\sum_{j=1}^d[I_k\otimes\mathbf u^{\otimes j}]\!\!\sum_{w\in\langle\uX\rangle_j\setminus\langle\uX\rangle_{j-1}}\!\!c_w\otimes\underline{a}_\tau^w\\
&=c_0\otimes 1+\sum_{j=1}^d\sum_{w\in\langle\uX\rangle_j\setminus\langle\uX\rangle_{j-1}}\!\!c_w\otimes\left(\sum_{v
\in\langle\uX\rangle_j\setminus\langle\uX\rangle_{j-1}}({\bf u}^{\otimes j})_{w,v}\underline{a}_\tau^v\right)\\
&=c_0\otimes 1+\sum_{j=1}^d\sum_{v\in\langle\uX\rangle_j\setminus\langle\uX\rangle_{j-1}}\underbrace{\left[\sum_{w\in\langle\uX\rangle_j\setminus\langle\uX\rangle_{j-1}}\!\!({\bf u}^{\otimes j})_{w,v}c_w\right]}_{=d_v}\otimes\underline{a}_\tau^v=\mathbf Q(\underline{a}_\tau),
\end{align*}
where $\mathbf Q(\uX)=\sum_{v\in\langle\uX\rangle_d}\!d_v\otimes\uX^v$. Tautologically $\|\mathbf P(\underline{\gimel}_\theta)\|=1\iff\|\mathbf Q(\underline{a}_\tau)\|=1$. For
given $\uA\in\mathbb M_k(\mathbb C)^n$, assume $\mathbf Q$ is a polynomial on which the maximum in \eqref{330} (resp. \eqref{331}) is achieved for $\underline{a}_\tau$.
Thus, from the above, 
\begin{align*}
\Phi^\bullet_{\tau,d}(\uA)&=\|\mathbf Q(\uA)(I_k)\|_\bullet^2=\left\|\sum_{v\in\langle\uX\rangle_d}\!d_v\uA^v\right\|_\bullet^2\\
&=\left\|c_0+\sum_{j=1}^d\sum_{v\in\langle\uX\rangle_j\setminus\langle\uX\rangle_{j-1}}\left[\sum_{w\in\langle\uX\rangle_j\setminus\langle\uX\rangle_{j-1}}\!\!
({\bf u}^{\otimes j})_{w,v}c_w\right]\underline{A}^v\right\|_\bullet^2\\
&=\left\|c_0+\sum_{j=1}^d\sum_{w\in\langle\uX\rangle_j\setminus\langle\uX\rangle_{j-1}}\!\!c_w\left[\sum_{v
\in\langle\uX\rangle_j\setminus\langle\uX\rangle_{j-1}}({\bf u}^{\otimes j})_{w,v}\underline{A}_\tau^v\right]\right\|_\bullet^2\\
&=\left\|\sum_{w\in\langle\uX\rangle_d}c_w\left(\mathbf u\uA\right)^w\right\|_\bullet^2=\left\|\mathbf P(\mathbf u\uA)(I_k)\right\|_\bullet^2\le\Phi^\bullet_{\theta,d}(\mathbf u\uA).
\end{align*}
Since the transformation $\mathbf u$ is invertible, it follows that strict inequality in the last relation would lead to 
the existence of a $\tilde{\mathbf P}$ on which the supremum is realized,
and then the corresponding $\tilde{\mathbf Q}$ would achieve the obvious contradiction to the above 
$\|\tilde{\mathbf Q}(\uA)(I_k)\|_\bullet^2>\Phi^\bullet_{\tau,d}(\uA)$. Thus,
\begin{equation}\label{haromszashetven}
\Phi^\bullet_{\tau,d}(\uA)=\Phi^\bullet_{\theta,d}(\mathbf u\uA),\quad\uA\in\mathbb M_k(\mathbb C)^n.
\end{equation}
Clearly, faithfulness of $\tau$ does not play a role in this argument; moreover, faithfulness of $\tau$ is equivalent to faithfulness of $\theta$.
The case of invertible, analytic, but nonlinear $\mathbf u$ is more involved and leads to an acceptably strong result only for the objects 
defined in \eqref{320}--\eqref{321} below. We are not concerned with it at this moment.

\subsubsection{Growth}\label{grow}
The functions $\underline{A}\mapsto\frac1{2d}\log\Phi^\bullet_{\tau,d}(\underline{A})$ have logarithmic growth
at infinity, in the sense that there exists a real constant $c_{d,\bullet,\tau}$ such that
$$
{\frac1{2d}\log\Phi^\bullet_{\tau,d}(\underline{A})}\le{\log^+\|\uA\|_\bullet+c_{d,\bullet,\tau}}.
$$
(As it is usual, $\log^+t=\max\{\log t,0\}$.)
It is remarkably convenient that from this point of view, the two norms we use are the same. Indeed, $\|\uA\|_2^2={\rm tr}_k(A_1^\star A_1+\cdots+A_n^\star A_n)\le
n\max_{1\le i\le n}{\rm tr}_k(A_i^\star A_i)\le n\max_{1\le i\le n}\|A_i\|^2\!=n\|\uA\|^2$ and $\|\uA\|^2\!=\max_{1\le i\le n}\|A_i\|^2\!\le k\max_{1\le i\le n}
{\rm tr}_k(A_i^\star A_i)\le k{\rm tr}_k(A_1^\star A_1+\cdots A_n^\star A_n)=k\|\uA\|_2^2,$ so that $\frac{1}{\sqrt{n}}\|\uA\|_2\le\|\uA\|\le\sqrt{k}\|\uA\|_2$.
By taking log, we obtain $\log\|\uA\|_2-\log\sqrt{n}\le\log\|\uA\|\le\log\|\uA\|_2+\log\sqrt{k}$. Thus, we will compare everything to $\log\|\uA\|_2$.
Now, for any given polynomial $\mathbf P\in\mathbb M_k(\mathbb C\langle\uX\rangle_d)\setminus\mathbb M_k(\mathbb C\langle\uX\rangle_{d-1})$, it 
follows from the classical theory that 
$$
\limsup_{\|\uA\|_2\to\infty}\frac{\log{\rm tr}_k\left(\mathbf P(\uA)(I_k)^\star\mathbf P(\uA)(I_k)\right)}{2d\log\|\uA\|_2}\leq 1.
$$
Indeed, the numerator is the logarithm of a classical polynomial of degree $2d$ on $\mathbb C^{nk^2}$. According to Section \ref{comp},
$\{\log{\rm tr}_k\left(\mathbf P(\uA)(I_k)^\star\mathbf P(\uA)(I_k)\right)\colon{\bf P}\in\mathbb M_k(\mathbb C\langle\uX\rangle_d)$, $\|{\bf P}(\underline{a}_\tau)\|\le1\}$
is a locally bounded family, so \cite[Theorem 1.6, Appendix B]{ST} guarantees that $\frac1{2d}\log\Phi^2_{\tau,d}(\underline{A})$ satisfies the same growth condition.
By the same theorem, together with the description of the operator norm from Section \ref{sspsh}, so does $\frac1{2d}\log\Phi^\infty_{\tau,d}(\underline{A})$.

\subsubsection{Single variable bounds}\label{1var}
As noted before, our hypothesis that $\tau$ is positive guarantees that $\underline{a}_\tau$ can be chosen as a tuple of selfadjoint
random variables, $\underline{a}_\tau^\star=\underline{a}_\tau$, which generate a tracial $C^*$-algebra $(C^*(\tau),{\rm tr})$. For each coordinate of
$\underline{a}_\tau$, we consider, via functional calculus, the identification of the $C^*$-subalgebra it generates with the space of continuous functions on its spectrum, endowed with
the restriction of tr. Without loss of generality, pick the first coordinate, ${a}_{\tau,1}$, for our analysis. Denote by $\mu_1$ its distribution with respect to $\tau$,
which is a classical, compactly supported Borel probability measure on $\mathbb R$. Assume that supp$(\mu_1$) is an infinite set in
the real line, so that $C^*(\tau|_{\mathbb C\langle X_1\rangle})=C(\sigma(a_{\tau,1}),\mu_1)$ is an infinite-dimensional Abelian $C^*$-algebra. For each $k\in\mathbb N$, we may
choose $k$ distinct points in $\sigma(a_{\tau,1})$; the evaluation of an element from $C(\sigma(a_{\tau,1}),\mu_1)$ at those points yields an algebra morphism onto
$\mathbb C^k$, viewed as the diagonal of $\mathbb M_k(\mathbb C)$. In particular, for given matrix size $k$, one may pick pairwise distinct points 
$\mathcal P=\{p_1,\dots,p_k\}\subset\sigma(a_{\tau,1})$, and define the linear map
$\mathcal C_\mathcal P\colon C(\sigma(a_{\tau,1}),\mu_1)\to\mathbb C^k$ by $f\mapsto \mathcal C_\mathcal P(f)=(f(p_1),\dots,f(p_k))$. This is tautologically a surjective
$C^*$-algebra morphism, and in particular it is completely contractive, completely positive, and unit preserving.
Thus, for any $N\in\mathbb N$, the map ${\rm id}_N\otimes\mathcal C_\mathcal P$ is a positive contraction. In particular, if
$\mathbf P(X)=\sum_{w\in\langle X\rangle_d}c_w\otimes X^w\in\mathbb M_k(\mathbb C\langle X\rangle_d)$, then $
\left\|\sum_{w\in\langle X\rangle_d}c_w\otimes\mathcal C_\mathcal P(a_{\tau,1}^w)\right\|=\left\|({\rm id}_k\otimes\mathcal C_\mathcal P)(\mathbf P(a_{\tau,1}))\right\|\le
\left\|\mathbf P(a_{\tau,1})\right\|$. However, $\mathcal C_\mathcal P$ is an algebra morphism, so that $\mathcal C_\mathcal P(a_{\tau,1}^w)=\mathcal C_\mathcal P(a_{\tau,1})^w
$ for all $w\in\mathbb N$ (or $w\in\mathbb Z$, if $a_{\tau,1}$ is invertible). 
As noted after equations \eqref{HS}--\eqref{H}, this, together with the fact that ${\rm tr}_k(I_k)=1$, guarantees that 
$\mathrm{tr}_k(\mathbf P(\mathcal C_\mathcal P(a_{\tau,1}))(I_k)^\star\mathbf P(\mathcal C_\mathcal P(a_{\tau,1}))(I_k))\leq\left\|\mathbf P(a_{\tau,1})\right\|^2$
for any collection $\mathcal P\subset\sigma(a_{\tau,1})$ of $k$ mutually disjoint points. In particular, we have established that for any selfadjoint matrix $A\in\mathbb M_k
(\mathbb C)$ whose spectrum is included in the spectrum of $a_{\tau,1}$, the condition $\left\|\mathbf P(a_{\tau,1})\right\|\le1$ implies automatically
$\Phi^2_{\tau|_{\mathbb C\langle X_1\rangle},d}(A)\le1$.

\bigskip
We define
\begin{align}\label{320}
\Phi_{\tau}^2(\uA)=\left[\limsup_{d\to\infty}\left(\Phi_{\tau,d}^2(\uA)\right)^{1/d}\right]^*
,\\
\Phi_{\tau}^\infty(\uA)=\left[\limsup_{d\to\infty}\left(\Phi_{\tau,d}^\infty(\uA)\right)^{1/d}\right]^*
.\label{321}
\end{align}
(The asterisk denotes the upper semicontinuous regularization: the upper semicontinuous regularization $f^*$ for a function $f$ is defined by
$f^*(\zeta)=\limsup_{z\to\zeta}f(z)$, and is the smallest  upper semicontinuous function with the property that $f^*\ge f$.) We also define
\begin{align}\label{Sigma2}
\Sigma_{\tau}^2(\uA)=\left[\sup_{d\in\mathbb N}\left(\Phi_{\tau,d}^2(\uA)\right)^{1/d}\right]^*
,\\
\Sigma_{\tau}^\infty(\uA)=\left[\sup_{d\in\mathbb N}\left(\Phi_{\tau,d}^\infty(\uA)\right)^{1/d}\right]^*
.\label{Sigma8}
\end{align}
Clearly, $\Sigma_{\tau}^\bullet(\uA)\ge\Phi_{\tau}^\bullet(\uA)$, $\bullet\in\{2,\infty\},\uA\in\mathbb M_k(\mathbb C)^n,k\in\mathbb N$.
It is known (see, for instance, \cite[Equation (2.17) in Appendix B2]{ST}) that $\Sigma$ and $\Phi$ coincide in the classical 
case. 
That might not happen for us, as it can be easily seen from \eqref{mess} below and the comments following it.
Before enumerating below the consequences of the above-listed properties of $\Phi^\bullet_{\tau,d}$ on these functions, 
we recall (see Section \ref{PSHintro}) that a set is {\em pluripolar} if it is contained in the $-\infty$-level set of a non-constant plurisubharmonic function. 
A property holds {\em quasi everywhere} (q.e.) in a set $S$ if it holds on $S \backslash E$ where $E$ is a pluripolar set.
\begin{remark}
The following statements hold:
\label{lemma:bounds}
\begin{enumerate}
\item 
\begin{equation}
\Phi_{\tau}^2(\uA)\le\Phi_{\tau}^\infty(\uA),\ \ \Sigma_{\tau}^2(\uA)\le\Sigma_{\tau}^\infty(\uA)\ \ \text{q.e.},\quad\uA\in\mathbb M_k(\mathbb C)^n,k\in\mathbb N;
\end{equation}
\item As limits of plurisubharmonic functions (see \ref{sspsh}), according to \cite[Proposition 1.40 and Corollary 1.47]{GZ} or \cite[Theorem 2.9.14 and Proposition 2.9.17]{Klimek}, 
$\Sigma_{\tau}^2(\uA),\Sigma_{\tau}^\infty(\uA),\Phi_{\tau}^2(\uA)$, and $\Phi_{\tau}^\infty(\uA)$ are plurisubharmonic as functions on the complex Euclidean space 
$\mathbb C^{nk^2}$. More precisely, if the family $\uA\mapsto\left(\Phi_{\tau,d}^\infty(\uA)\right)^{1/d},d\in\mathbb N,$ is locally bounded from the above, then all 
the above functions are plurisubharmonic and the set where they differ from their upper semicontinuous regularization has Lebesgue measure zero in $\mathbb C^{nk^2}$ 
$($in fact it is of zero inner/outer capacity \cite[Corollary 4.43]{GZ}$)$. 
\end{enumerate}
In the following we assume implicitly that $\uA\mapsto\left(\Phi_{\tau,d}^\infty(\uA)\right)^{1/d}$ is 
locally bounded from the above in order for most of our statements to be non-vacuous (this happens in numerous cases as shown later on in Example \ref{L(F2)}). 
\begin{enumerate}
\setcounter{enumi}{2}
\item \label{max} If we denote $\tau_1=\tau|_{\mathbb C\langle X_1,\dots,X_n\rangle},\tau_2=\tau|_{\mathbb C\langle X_{n+1},\dots,X_{n+m}\rangle}$, then
$\Phi_{\tau}^\bullet(\underline{A},\underline{B})\ge\max\{\Phi_{\tau_1}^\bullet(\underline{A}),\Phi^\bullet_{\tau_2}(\underline{B})\}$,
$\Sigma_{\tau}^\bullet(\underline{A},\underline{B})\ge\max\{\Sigma_{\tau_1}^\bullet(\underline{A}),\Sigma^\bullet_{\tau_2}(\underline{B})\}$
for any $(\underline{A},\underline{B})\in\mathbb M_k(\mathbb C)^n\times\mathbb M_k(\mathbb C)^m,\bullet\in\{2,\infty\}$. This follows 
immediately from Section \ref{monot}.
\item By \ref{ssunitary}, for any $k\in\mathbb N,\bullet\in\{2,\infty\}$, we have
\begin{equation} \label{UnitaryInvariance}
\Phi_{\tau}^\bullet(\uA)=\Phi_{\tau}^\bullet(U\uA U^\star),\ \ \Sigma_{\tau}^\bullet(\uA)=\Sigma_{\tau}^\bullet(U\uA U^\star),\quad \uA\in\mathbb M_k(\mathbb C)^n,U\in\mathbb M_k(\mathbb C),UU^\star=I_k.
\end{equation}
As an immediate consequence, we record
$$
\Phi_{\tau}^\bullet\left(\begin{bmatrix}\uA&\uB\\\uB&\uA^\star\end{bmatrix}\right)=\Phi_{\tau}^\bullet\left(\begin{bmatrix}\uA^\star&\uB\\\uB&\uA\end{bmatrix}\right),\quad 
\Sigma_{\tau}^\bullet\left(\begin{bmatrix}\uA&\uB\\\uB&\uA^\star\end{bmatrix}\right)=\Sigma_{\tau}^\bullet\left(\begin{bmatrix}\uA^\star&\uB\\\uB&\uA\end{bmatrix}\right),
$$
for all $\uA,\uB\in\mathbb M_k(\mathbb C)^n,k\in\mathbb N,\bullet\in\{2,\infty\}.$
\item If $\mathbf u\in GL_n(\mathbb R)$, $\tau$ is the distribution of $\underline{a}_\tau$, and $\theta$ is the distribution of $\mathbf u\underline{a}_\tau$, then by
\eqref{haromszashetven},
\begin{equation}
\Phi^\bullet_{\tau}(\uA)=\Phi^\bullet_{\theta}(\mathbf u\uA),\ \ \Sigma^\bullet_{\tau}(\uA)=\Sigma^\bullet_{\theta}(\mathbf u\uA)\quad\uA\in\mathbb M_k(\mathbb C)^n, k\in\mathbb N,\bullet\in\{2,\infty\}.
\end{equation}
\end{enumerate}
\end{remark}

Let us record some obvious properties of $\Phi^\bullet_\tau,\bullet\in\{2,\infty\}$.
\begin{enumerate}

\item $\Phi^\bullet_\tau(\uA)\ge1$ for all $\bullet\in\{2,\infty\},\uA\in\mathbb M_k(\mathbb C)^n,k\in\mathbb N$.  This follows easily by fixing a 
polynomial $\mathbf P$ in the right hand side of either of \eqref{330}, \eqref{331} and noting that $\Phi_{\tau,d}^\bullet(\uA)\ge
\|\mathbf P(\uA)(I_k)\|_\bullet^2$; we conclude $\Phi_\tau^\bullet(\uA)\ge\lim_{d\to\infty}\|\mathbf P(\uA)(I_k)\|_\bullet^{2/d}=1$. The conclusion for $\Sigma$ is obvious;

\item In the classical theory of several complex variables, it is known (see \cite[Appendix B.2, (2.17)]{ST}) that the sequence $\{\Phi_{\tau,d}^2(\underline{A})^\frac1d\}_{d\in
\mathbb N}$ is increasing for any $\underline{A}\in\mathbb M_1(\mathbb C)^n\simeq\mathbb C^n$. When the matrix size $k>1$, this monotonicity statement may turn out to
be false. This fact is essentially caused by the existence of nilpotent elements in $\mathbb M_k(\mathbb C)$ whenever $k\ge2$. Surprisingly, a counterexample can be found
even when $n=1$. Let $\mathbf P(X)=\sum_{j=0}^dc_j\otimes X^j=\begin{bmatrix}P_{1,1}(X) & P_{1,2}(X) \\ P_{2,1}(X) & P_{2,2}(X)\end{bmatrix}\in
\mathbb M_2(\mathbb C\langle X\rangle)$ and $A=\begin{bmatrix}0 & t \\ 0 & 0 \end{bmatrix}$. Then 
\begin{align*}
\mathbf P(A)(I_2)&=\mathbf P\left(\begin{bmatrix}0 & t \\ 0 & 0 \end{bmatrix}\right)\left(\begin{bmatrix}1 & 0 \\ 0 & 1 \end{bmatrix}\right)
=\sum_{j=0}^dc_j\begin{bmatrix}0 & t \\ 0 & 0 \end{bmatrix}^j\\
&=\begin{bmatrix}(c_0)_{1,1} & (c_0)_{1,2}+t(c_1)_{1,1}\\ (c_0)_{2,1} & (c_0)_{2,2}+t(c_1)_{2,1}\end{bmatrix}=\begin{bmatrix}P_{1,1}(0) & tP_{1,1}'(0)+P_{1,2}(0) \\ P_{2,1}(0) & tP_{2,1}'(0)+P_{2,2}(0) \end{bmatrix}.
\end{align*}
Pick $\tau$ to be Wigner's semicircle law, $\tau(P(X))=\frac2\pi\int_{-1}^1P(x)\sqrt{1-x^2}\,{\rm d}x$ (in fact, for our example, any distribution given by a
probability measure with infinite and compact support containing zero would do). By definition, $\Phi_{\tau,d}^2(A)$ is the maximum of ${\rm tr}_2(\mathbf P(A)(I_2)^\star
\mathbf P(A)(I_2))$ on the set of polynomials $\mathbf P(X)\in\mathbb M_2(\mathbb C\langle X\rangle)$ satisfying $\|\mathbf P(a_\tau)\|\le1$. 
With the above notation, $\mathrm{tr}_2(\mathbf P(A)(I_2)^\star\mathbf P(A)(I_2))=\frac{1}{2}\left(|P_{1,1}(0)|^2\right.+|tP_{1,1}'(0)+P_{1,2}(0)|^2+|P_{2,1}(0)|^2+
\left.|tP_{2,1}'(0)+P_{2,2}(0)|^2\right).$ Now, $P_{i,j}$ are classical polynomials in one variable, to which Bernstein's estimate \cite[(33), Section 4.8.7]{Tim} applies:
$|P_{i,j}'(0)|\le{\rm deg}P_{i,j}\|P_{i,j}\|\le d\|P_{i,j}\|$, where $\|\cdot\|$ is the sup, or $C^*$-algebra, norm on $[-1,1]$ (the estimate is optimal: an example can be found in the
Chebyshev orthogonal polynomials of the first kind, given by $T_d(x)=\cos(d\arccos(x))$, or $T_d(\cos\theta)=\cos(d\theta),\theta\in[0,\pi]$, so that $\|T_d\|=1$, $|T_d'(0)|=
d|\sin(d\pi/2)|=d$ whenever $d$ is odd). The requirement $\|\mathbf P(a_\tau)\|\le1$ imposes
$|P_{i,j}(0)|\le1$, $1\le i,j\le2$, so that $\mathrm{tr}_2(\mathbf P(A)(I_2)^\star\mathbf P(A)(I_2))<1+(1+td)^2$ holds for all degrees $d\in\mathbb N$ and parameters
$t\in\mathbb C.$ On the other hand, by choosing the other entries to be zero, one sees that the estimate $\mathrm{tr}_2(\mathbf P(A)(I_2)^\star\mathbf P(A)(I_2))\ge\frac{|tP_{2,
1}'(0)|^2}{2}\ge(t^2d^2-1)/2$ holds at least on a subsequence of degrees $d$. Since both sequences $(1+(1+td)^2)^{1/d}$ and $(\frac{t^2d^2-1}{2})^{1/d}$ are both eventually
decreasing and tend to the same limit as $d \to \infty$, one obtains that
\begin{equation}\label{mess}
\left\{\Phi^2_{\tau,d}(\uA)^\frac1d\right\}_{d\in\mathbb N}\text{ might not be an increasing sequence}.
\end{equation}
Thus, if one replaces limsup in  \eqref{320}--\eqref{321} by sup as in  \eqref{Sigma2}--\eqref{Sigma8}, one may get a different function whenever $k>1$ (and in fact one is 
guaranteed the existence of a context in which a different result occurs before performing the upper semicontinuous regularization).

\item $\{\uA\in\mathbb M_k(\mathbb C)^n\colon\Phi_\tau^\bullet(\uA)>1\}\neq\varnothing$ for all $k>1$. To find such an example, one considers $d>1$ and
a selfadjoint polynomial $P\in\mathbb C\langle\uX\rangle_d$ (identified with $I_k\otimes P\in\mathbb M_k(\mathbb C\langle\uX\rangle_d)$) such that
$\|P(\underline{a}_\tau)\|=1$ (for instance, $P(\underline{X})=\left(\frac{X_1}{\|{a_\tau}_1\|}\right)^d$), and one picks a selfadjoint tuple 
$\uA$ such that $\|P(\uA)\|_\bullet>1$ (if $P(\underline{X})=\left(\frac{{X}_1}{\|{a_\tau}_1\|}\right)^d$, then any tuple $\uA$ whose
first coordinate is a positive matrix whose spectrum is included in $(\|{a_\tau}_1\|,+\infty)$ will do). Then $\Phi_\tau^\bullet(\uA)
\ge\|(P^\star P)^{2^n}(\uA)\|^\frac{1}{2^nd}_\bullet\ge\|(P^\star P)(\uA)\|_\bullet^\frac1d>1$ for all sufficiently large $n$.


\item $\Phi^2_\tau(\uA\oplus\uB)\ge\max\{\Phi^2_\tau(\uA),\Phi^2_\tau(\uB)\},\uA\in\mathbb M_{k_1}(\mathbb C)^n,\uB\in
\mathbb M_{k_2}(\mathbb C)^n$. Indeed, this follows easily from the fact that if $\mathbf P_j\in\mathbb M_{k_j}(\mathbb C\langle\uX\rangle_d)$, 
$j=1,2$, are maximizing elements in \eqref{330} corresponding to $\uA$ and $\uB$, respectively, then $\mathbf P_1\oplus\mathbf P_2$ still satisfies 
$\|(\mathbf P_1\oplus\mathbf P_2)(\underline{a}_\tau)\|\le1$ and thus $\Phi^2_{\tau,d}(\uA\oplus\uB)\ge\text{tr}_{k_1+k_2}(|\mathbf P_1\oplus
\mathbf P_2)(\uA\oplus\uB)(I_{k_1+k_2})|^2)=\frac{k_1}{k_1+k_2}\text{tr}_{k_1}(|\mathbf P_1(\uA)(I_{k_1})|^2)+\frac{k_2}{k_1+k_2}\text{tr}_{k_2}
(|\mathbf P_2(\uB)(I_{k_2})|^2)=\frac{k_1}{k_1+k_2}\Phi^2_{\tau,d}(\uA)+\frac{k_2}{k_1+k_2}\Phi^2_{\tau,d}(\uB)$. If one of $\Phi^2_\tau(\uA),$ 
$\Phi^2_\tau(\uB)$ is infinite, then so is $\Phi^2_\tau(\uA\oplus\uB)$. Assume that $\Phi^2_\tau(\uA)>\Phi^2_\tau(\uB)$. Pick a subsequence 
$\{d_\ell\}_{\ell\in\mathbb N}$ on which $[\Phi^2_{\tau,d_\ell}(\uA)]^\frac{1}{d_\ell}\ge[\Phi^2_{\tau,d_\ell}(\uB)]^\frac{1}{d_\ell}$ and 
$\Phi^2_\tau(\uA)$ is achieved. 
Then 
\begin{align*}
\quad \ \qquad {\displaystyle\limsup_{d\to\infty}}(\Phi^2_{\tau,d}(\uA\oplus\uB))^{\frac{1}{d}} & \ge
{\displaystyle\limsup_{\ell\to\infty}}(\Phi^2_{\tau,d_\ell}(\uA\oplus\uB))^{\frac{1}{d_\ell}}\\
& \ge
{\displaystyle\lim_{\ell\to\infty}}(\Phi^2_{\tau,d_\ell}(\uA))^{\frac{1}{d_\ell}}\left(\frac{k_1}{k_1+k_2}+\frac{k_2}{k_1+k_2}
\frac{\Phi^2_{\tau,d}(\uB)}{\Phi^2_{\tau,d}(\uA)}\right)^{\frac{1}{d_\ell}}\\
& =\Phi^2_\tau(\uA),
\end{align*}
since the term under the parenthesis is 
bounded from above by 1. The case $\Phi^2_\tau(\uA)=\Phi^2_\tau(\uB)$ is equally easy 
(one picks a subsequence as above, but on which one of the two is no less than the other). The same may not hold for $\Sigma^2$ (the sup might be achieved at a
finite $d$).

\item $\Phi^\infty_\tau(\uA\oplus\uB)\ge\max\{\Phi^\infty_\tau(\uA),\Phi^\infty_\tau(\uB)\}$. As in the previous item, for each given $d$ 
we find maximizers $\mathbf P_j\in\mathbb M_{k_j}(\mathbb C\langle\uX\rangle_d)$, $j=1,2$, and then 
$\Phi^\infty_{\tau,d}(\uA\oplus\uB)\ge\|(\mathbf P_1\oplus\mathbf P_2)(\uA\oplus\uB)(I_{k_1+k_2})\|=\max\{\|\mathbf P_1(\uA)(I_{k_1})\|,
\|\mathbf P_2(\uB)(I_{k_2})\|\}$. The rest of the argument is identical to the one above.

\end{enumerate}

\subsection{Level sets of the noncommutative Christoffel-Darboux polynomial: comparison of $\Phi^\bullet_\tau$ and $\kappa_\tau$}\label{trentaquatro}
According the second item of Remark \ref{Remark33}, we find $\kappa_{\tau,d}(\uA,\uA^\star)(I_k)$ as the maximum of $\mathbf P(\uA)(I_k)^\star 
\mathbf P(\uA)(I_k)$ under the condition that $({\rm Id}_{\Mbb_k(\mathbb C)}\otimes\tau)(\mathbf{PP}^\star)=I_k$. For the element, call it 
$\mathbf P_{\uA}$, which satisfies this equality, we have tautologically $\|{\bf P}_{\uA}(\underline{a}_\tau)\|_{L^2(\text{tr}_k\otimes\text{tr})}=
\text{tr}_k(I_k)=1$. In order to productively connect this to our versions of Siciak's extremal function, we need
\begin{definition}\label{BM}
Consider a unit-preserving positive bounded linear map $\varphi\colon\mathbb C\langle\uX\rangle\!\to\mathbb C$ and a tuple of noncommutative 
selfadjoint random variables $\underline{a}_\varphi$ in the von Neumann algebra generated by $\varphi$ via the GNS construction such that the law of 
$\underline{a}_\varphi$ is $\varphi$. We say that $\varphi$ satisfies the {\em Bernstein-Markov property} if for all $d\in\mathbb N$ there exists an 
$M_d=M(\tau)_d\ge0$ $($possibly depending on $k)$ such that
$$
\|\mathbf P(\underline{a}_\varphi)\|\leq 
M_d(\mathrm{tr}_k\otimes\varphi)(\mathbf P(\underline{a}_\varphi)^\star\mathbf P(\underline{a}_\varphi))^\frac12=M_d\|\mathbf P
(\underline{X})\|_{L^2(\mathrm{tr}_k\otimes\varphi)}
$$
for all $\mathbf P\in\Mbb_k(\mathbb C\langle\uX\rangle_d)$ of degree $d$, and
$$
\limsup_{d\to\infty}M_d^\frac1d=1 \text{ for each }k.
$$
\end{definition}
As $\mathrm{tr}_k\otimes\varphi$ is a state, it has norm equal to one, so that $(\mathrm{tr}_k\otimes\varphi)(Y^\star Y)\le\|Y^\star Y\|=\|Y\|^2$, i.e.
$\|Y\|_{L^2(\mathrm{tr}_k\otimes\varphi)}\le\|Y\|$. Thus, automatically $M_d\ge1$ for all $d,k$.

{\color{green}
}

This actually {\em is} the classical definition of the  Bernstein-Markov property: if $k=1$ and $\varphi$ is a probability measure on a Euclidean space, then 
the above is the expression of the Bernstein-Markov property in terms of the (real) random variables whose joint distribution is $\varphi$. For numerous 
details on this matter, we refer to \cite{BLPW,Bloom}. In this paper we will only consider the case when $\varphi=\tau$ is a positive bounded tracial 
state (usually, but not always, faithful), and thus $\underline{a}_\varphi=\underline{a}_\tau$ is a tuple of selfadjoint operators in a finite von Neumann algebra.

The conditions in the definition have some obvious reformulations: given any sequence of polynomials $\{\mathbf P_d\}_{d\in\mathbb N},
\deg\mathbf P_d\le d$,
\begin{align}\label{limsup}
\limsup_{d\to\infty}&\left(\frac{\|\mathbf P_d(\underline{a}_\tau)\|}{\|\mathbf P_d(\underline{a}_\tau)\|_{L^2(\text{tr}_k\otimes\tau)}}
\right)^{\frac1d}\le1;\\
\forall\varepsilon>0\ \exists C(\varepsilon,\tau)>0\text{ such that }&\|\mathbf P_d(\underline{a}_\tau)\|\le C(\varepsilon,\tau)(1+\varepsilon)^d
\|\mathbf P_d(\underline{a}_\tau)\|_{L^2(\text{tr}_k\otimes\tau)}.
\end{align}
As of now, we are not aware of this notion having previously appeared in free probability or free analysis. In particular, it seems far from obvious 
the precise way it is related to established notions from free probability such as free entropy, conjugate variables, or even free independence. 
However, unlike in classical pluripotential theory, the Bernstein-Markov property appears to be somewhat less ubiquitous in the (highly)
noncommutative context. We shall thus introduce an obvious modification:

\begin{definition}\label{BM-C}
Consider a bounded positive unital trace $\tau\colon\mathbb C\langle\uX\rangle\to\mathbb C$ and a tuple of noncommutative 
selfadjoint random variables $\underline{a}_\tau$ in the von Neumann algebra $W^*(\tau)$ generated by $\tau$ via the GNS construction such that the law of 
$\underline{a}_\tau$ is $\tau$. 
Given a real constant $\mathfrak C \in [1, \infty)$, we say that $\tau$ satisfies the $\mathfrak C$-{\em Bernstein-Markov property} if for all $d\in\mathbb N$ there exists an 
$M_d=M(\tau)_d\ge0$ $($possibly depending on $k)$ such that
$$
\|\mathbf P(\underline{a}_\tau)\|\leq 
M_d(\mathrm{tr}_k\otimes\mathrm{tr})(\mathbf P(\underline{a}_\tau)^\star\mathbf P(\underline{a}_\tau))^\frac12=M_d\|\mathbf P
(\underline{X})\|_{L^2(\mathrm{tr}_k\otimes\tau)}
$$
for all $\mathbf P\in\Mbb_k(\mathbb C\langle\uX\rangle_d)$ of degree $d$, and
$$
\limsup_{d\to\infty}M_d^\frac1d=\mathfrak C \text{ for each }k.
$$
\end{definition}


The next proposition states the main results of this section 
(compare with \cite[Lemma 2.9]{beckermann2020perturbations} and its proof):

\begin{proposition}\label{Siciak}
For any given bounded 
positive trace $\tau$ on $\mathbb C\langle\uX\rangle$, we have
\begin{align}
\limsup_{d\to\infty}\|\kappa_{\tau,d}(\uA,\uA^\star)(I_k)^\frac1d\|&\ge\Phi^\infty_\tau(\uA),\quad\sup_{d\in\mathbb N}\|\kappa_{\tau,d}(\uA,\uA^\star)(I_k)^\frac1d\|\ge\Sigma^\infty_\tau(\uA),\\
\limsup_{d\to\infty}\mathrm{tr}_k(\kappa_{\tau,d}(\uA,\uA^\star)(I_k))^\frac1d&\ge\Phi^2_{\tau}(\uA),\quad\sup_{d\in\mathbb N}\mathrm{tr}_k(\kappa_{\tau,d}(\uA,\uA^\star)(I_k))^\frac1d\ge\Sigma^2_{\tau}(\uA),
\end{align}
for all $ \uA\in\mathbb M_k(\mathbb C)^n,k\in\mathbb N.$
If $\tau$ satisfies the $\mathfrak C$-Bernstein-Markov property \ref{BM-C}, then 
\begin{align}
\mathfrak C^2\Phi^\infty_\tau(\uA)\ge\limsup_{d\to\infty}\|\kappa_{\tau,d}(\uA,\uA^\star)(I_k)^\frac1d\|&\ge\Phi^\infty_\tau(\uA),\label{Cinfty}\\
\mathfrak C^2\Phi^2_\tau(\uA)\ge\limsup_{d\to\infty}{\rm tr}_k(\kappa_{\tau,d}(\uA,\uA^\star)(I_k))^\frac1d&\ge\Phi^2_{\tau}(\uA),\label{C2}\\
\Sigma^\infty_\tau(\uA)\sup_{d\in\mathbb N}M_d^{\frac2d}\ge\sup_{d\in\mathbb N}\|\kappa_{\tau,d}(\uA,\uA^\star)(I_k)^\frac1d\|&\ge\Sigma^\infty_\tau(\uA),\\
\Sigma^2_\tau(\uA)\sup_{d\in\mathbb N}M_d^{\frac2d}\ge\sup_{d\in\mathbb N}{\rm tr}_k(\kappa_{\tau,d}(\uA,\uA^\star)(I_k))^\frac1d&\ge\Sigma^2_\tau(\uA),\ \,
 \uA\in\mathbb M_k(\mathbb C)^n,k\in\mathbb N.
\end{align}
\end{proposition}
Before proving the proposition, we give some intuition for the Bernstein-Markov property by explaining how to verify it in practice.
Observe first that in order to verify the $\mathfrak C$-Bernstein-Markov condition, one only needs to check it in the case 
$k=1$. Indeed, if $\mathbf P(\underline{a}_\tau)=\left(P_{i,j}(\underline{a}_\tau)\right)_{i,j=1}^k$, then 
$\mathbf P^\star(\underline{a}_\tau)=\left(P_{j,i}^\star(\underline{a}_\tau)\right)_{i,j=1}^k$, so that 
$\mathbf P(\underline{a}_\tau)\mathbf P^\star(\underline{a}_\tau)=
\left(\sum_{l=1}^kP_{i,l}(\underline{a}_\tau)P^\star_{j,l}(\underline{a}_\tau)\right)_{i,j=1}^k$. Then
$$
\|\mathbf P
(\underline{X})\|_{L^2(\mathrm{tr}_k\otimes\tau)}^2=\frac1k\sum_{i,l=1}^k\mathrm{tr}\left(P_{i,l}(\underline{a}_\tau)P^\star_{i,l}(\underline{a}_\tau)\right)
=\frac1k\sum_{i,l=1}^k\|P_{i,l}(\underline{X})\|^2_{L^2(\tau)}.
$$
We write
$$
\frac{\|\mathbf P(\underline{a}_\tau)\|}{\|\mathbf P(\underline{X})\|_{L^2(\mathrm{tr}_k\otimes\tau)}}
=\frac{\sqrt{k}\|\mathbf P(\underline{a}_\tau)\|}{\left[\sum_{i,l=1}^k\|P_{i,l}(\underline{X})\|^2_{L^2(\tau)}\right]^\frac12}
\le\frac{\sqrt{k}\sum_{i,l=1}^k\|P_{i,l}(\underline{a}_\tau)\|}{\left[\sum_{i,l=1}^k\|P_{i,l}(\underline{X})\|^2_{L^2(\tau)}\right]^\frac12}.
$$
Assume $\mathfrak C=\mathfrak C_k$ depends on $k$. By considering diagonal polynomials, it is obvious that $\{\mathfrak C_k\}_{k}$ can only be non-decreasing.
Consider a subsequence of polynomials $\mathbf P=\mathbf P_d,d\in\mathbb N$  being the degree of $\mathbf P$,
so that $\lim_{d\to\infty}\left(\frac{\|\mathbf P(\underline{a}_\tau)\|}{\|\mathbf P(\underline{X})\|_{L^2(\mathrm{tr}_k\otimes\tau)}}\right)^\frac1d=\mathfrak C_k$.
That means,
\begin{eqnarray}
\quad\mathfrak C_k&=& \lim_{d\to\infty}\left(\frac{\|\mathbf P(\underline{a}_\tau)\|}{\|\mathbf P(\underline{X})\|_{L^2(\mathrm{tr}_k\otimes\tau)}}\right)^\frac1d \nonumber
\le\lim_{d\to\infty}\left(\frac{\sqrt{k}\sum_{i,l=1}^k\|P_{i,l}(\underline{a}_\tau)\|}{\left[\sum_{i,l=1}^k\|P_{i,l}(\underline{X})\|^2_{L^2(\tau)}\right]^\frac12}\right)^\frac1d\\
& = & \lim_{d\to\infty}\left(\frac{\|P_{i_0,l_0}(\underline{a}_\tau)\|}{\|P_{i_0,l_0}(\underline{X}_\tau)\|_{L^2(\tau)}}\right)^\frac1d
\left(\frac{1+\sum_{(i,l)\neq(i_0,l_0)}\frac{\|P_{i,l}(\underline{a}_\tau)\|}{\|P_{i_0,l_0}(\underline{a}_\tau)\|}}{\left[1+\sum_{(i,l)\neq(i_0,l_0)}\frac{\|P_{i,l}(\underline{X})\|^2_{L^2(\tau)}}{\|P_{i_0,l_0}(\underline{X})\|^2_{L^2(\tau)}}\right]^\frac12}\right)^\frac1d.\label{ceka}
\end{eqnarray}
This holds for any pair of indices $(i_0,l_0)\in\{1,\dots,k\}^2$ corresponding to a non-zero entry. Obviously, for each such pair, the sequence 
$\left(\frac{\|P_{i_0,l_0}(\underline{a}_\tau)\|}{\|P_{i_0,l_0}(\underline{X})\|_{L^2(\tau)}}\right)^\frac1d$ indexed by the degrees 
$d$ of the polynomials $\mathbf P$ has an uppermost limit point. We choose the pair $(i_0,l_0)\in\{1,\dots,k\}^2$ for which this uppermost
limit point is the largest; if there are two or more which reach this largest upper limit, we choose $(i_0,l_0)$ such that 
the proportion $\frac{\|P_{i_0,l_0}(\underline{a}_\tau)\|}{\|P_{i_0,l_0}(\underline{X})\|_{L^2(\tau)}}$, at least on a 
subsequence on which the upper limit is reached, grows no slower than $\frac{\|P_{i,l}(\underline{a}_\tau)\|}{\|P_{i,l}(\underline{X})\|_{L^2(\tau)}}$
for any other pair $(i,l)$. We pass to a subsequence of $d$'s such that this uppermost limit is reached along a  sequence which is fastest
growing among all pairs $(i,l)$. By an abuse of notation, we still write
limit after $d$, that is, $\lim_{d\to\infty}\left(\frac{\|P_{i_0,l_0}(\underline{a}_\tau)\|}{\|P_{i_0,l_0}(\underline{X})\|_{L^2(\tau)}}\right)^\frac1d$, 
for the limit along this subsequence. We claim that the limit points of
$$\left(\frac{1+\sum_{(i,l)\neq(i_0,l_0)}\frac{\|P_{i,l}(\underline{a}_\tau)\|}{\|P_{i_0,l_0}(\underline{a}_\tau)\|}}{\left[1+\sum_{(i,l)\neq(i_0,l_0)}\frac{\|P_{i,l}(\underline{X})\|^2_{L^2(\tau)}}{\|P_{i_0,l_0}(\underline{X})\|^2_{L^2(\tau)}}\right]^\frac12}\right)^\frac1d$$ along this same subsequence are bounded from above by one.
Indeed, the only way for this to fail is if at least one $\frac{\|P_{i_1,l_1}(\underline{a}_\tau)\|}{\|P_{i_0,l_0}(\underline{a}_\tau)\|}$ tends very fast to 
infinity while all of $\frac{\|P_{i,l}(\underline{X})\|_{L^2(\tau)}}{\|P_{i_0,l_0}(\underline{X})\|_{L^2(\tau)}}$ either stay bounded or tend much
slower to infinity. However,
$$
\frac{\|P_{i_1,l_1}(\underline{a}_\tau)\|}{\|P_{i_0,l_0}(\underline{a}_\tau)\|}=\frac{\|P_{i_1,l_1}(\underline{a}_\tau)\|}{\|P_{i_1,l_1}(\underline{X})\|_{L^2(\tau)}}
\cdot\frac{\|P_{i_1,l_1}(\underline{X})\|_{L^2(\tau)}}{\|P_{i_0,l_0}(\underline{X})\|_{L^2(\tau)}}\cdot
\frac{\|P_{i_0,l_0}(\underline{X})\|_{L^2(\tau)}}{\|P_{i_0,l_0}(\underline{a}_\tau)\|},
$$
and, according to our choice of $(i_0,l_0)$ and of the subsequence, for $d$ large enough, 
$\frac{\|P_{i_1,l_1}(\underline{a}_\tau)\|}{\|P_{i_1,l_1}(\underline{X})\|_{L^2(\tau)}}
\cdot\frac{\|P_{i_0,l_0}(\underline{X})\|_{L^2(\tau)}}{\|P_{i_0,l_0}(\underline{a}_\tau)\|}<2$. Thus,  
$\frac{\|P_{i_1,l_1}(\underline{a}_\tau)\|}{\|P_{i_0,l_0}(\underline{a}_\tau)\|}$ cannot grow more than twice as fast as
$\frac{\|P_{i_1,l_1}(\underline{X})\|_{L^2(\tau)}}{\|P_{i_0,l_0}(\underline{X})\|_{L^2(\tau)}}$. We conclude that 
$$\left(\frac{1+\sum_{(i,l)\neq(i_0,l_0)}\frac{\|P_{i,l}(\underline{a}_\tau)\|}{\|P_{i_0,l_0}(\underline{a}_\tau)\|}}{\left[1+\sum_{(i,l)\neq(i_0,l_0)}\frac{\|P_{i,l}(\underline{X})\|^2_{L^2(\tau)}}{\|P_{i_0,l_0}(\underline{X})\|^2_{L^2(\tau)}}\right]^\frac12}\right)^\frac1d$$ 
is indeed bounded from above by $1$. By \eqref{ceka}, we conclude that $\mathfrak C_k\leq\mathfrak C_1$ for all $k\in\mathbb N$, so $\mathfrak C_k=\mathfrak C$
indeed does not depend on $k$. It should also be noted that if $\tau$ satisfies the $\mathfrak C$-Bernstein-Markov property, then $\mathfrak C\le\sup_dM_d^\frac1d<\infty$,
and the first inequality may be strict.

Second, in order to verify the $\mathfrak C$-Bernstein-Markov property, it is enough to verify it on positive polynomials:
there exists $Q_d(\uX)\in\mathbb C\langle\uX\rangle_{2d},d\in\mathbb N$, with $Q_d \succeq 0$ in $\mathbb C\langle\uX\rangle$, 
\begin{equation}
\limsup_{d\to\infty}\left(\frac{\|Q_d(\underline{a}_\tau)\|}{\text{tr}(Q_d(\underline{a}_\tau))}\right)^{\frac1{2d}}=\mathfrak C, 
\end{equation}
with no value larger than $\mathfrak C$ being attainable. (Here we performed the substitution $Q_d = P_d^\star P_d$.)  
For the Bernstein-Markov property ($\mathfrak C=1)$, it is enough to simply show that 
for {\em any} sequence $Q_d$ as above, $\limsup_{d\to\infty}\left(\frac{\|Q_d(\underline{a}_\tau)\|}{\text{tr}(Q_d(\underline{a}_\tau))}\right)^{\frac1{2d}}=1$.
As an aside, we note that for any $d\in\mathbb N$ the optimization problem 
\[
\max\{\|Q_d(\underline{a}_\tau)\|\colon\text{tr}(Q_d(\underline{a}_\tau))\le1,Q_d\succeq 0,
Q_d(\uX)\in\mathbb C\langle\uX\rangle_d\} \] 
 is convex, on a convex finite-dimensional set, hence solvable in principle.

These facts should probably arouse no astonishment in view of Theorem \ref{th:christoffeloptim}. Indeed, the connection
of $\kappa_{\tau,d}$ to the Bernstein-Markov property (respectively the $\mathfrak C$-Bernstein-Markov property)
becomes obvious in view of this theorem and Remark \ref{Remark33} following it: recall that 
$$
\kappa_{\tau,d}(\uA,\uA^\star)(I_k)=\max_{\mathbf P\in\mathbb M_k(\mathbb C\langle\uX\rangle_d)}\mathbf P(\uA)(I_k)^*
({\rm Id}_{\Mbb_k(\mathbb C)}\otimes\tau)(\mathbf{PP}^\star)^{-1}\mathbf P(\uA)(I_k),
$$
where we assume the obvious conditions of invertibility. As noted in Remark \ref{Remark33}(2), the set of polynomials on which the
above maximum is achieved is invariant under multiplication to the left with a scalar invertible $k\times k$ complex matrix, so in the above 
we may pick a polynomial $\mathbf P$ so that $\|\mathbf P(\underline{a}_\tau)\|=1$ and 
$({\rm Id}_{\Mbb_k(\mathbb C)}\otimes\tau)(\mathbf{PP}^\star)=\|\mathbf P(\underline{X})\|_{L^2({\rm tr}_k\otimes\tau)}^2I_k$ (i.e. the $k\times k$ complex matrix 
$({\rm Id}_{\Mbb_k(\mathbb C)}\otimes\tau)(\mathbf{PP}^\star)$ is a multiple of the identity). Now consider a
polynomial $\tilde{\mathbf P}$ such that $\|\tilde{\mathbf P}(\underline{a}_\tau)\|=1$ and $\frac{1}{\|\tilde{\mathbf P}(\underline{a}_\tau)\|_{L^2({\rm tr}_k\otimes\tau)}}
=\sup_{\mathbf P\in\mathbb M_k(\mathbb C\langle\uX\rangle_d)}\frac{\|\mathbf P(\underline{a}_\tau)\|}{\|\mathbf P(\underline{X})\|_{L^2({\rm tr}_k\otimes\tau)}}:=M_d$.
Of course, $\kappa_{\tau,d}(\uA,\uA^\star)(I_k)\succeq\tilde{\mathbf P}(\uA)(I_k)^*
({\rm Id}_{\Mbb_k(\mathbb C)}\otimes\tau)(\tilde{\mathbf P}\tilde{\mathbf P}^\star)^{-1}\tilde{\mathbf P}(\uA)(I_k)$ if 
$({\rm Id}_{\Mbb_k(\mathbb C)}\otimes\tau)(\tilde{\mathbf P}\tilde{\mathbf P}^\star)$ is invertible. If not, we pick an $\epsilon>0$
arbitrarily small and choose $\tilde{\mathbf P}_\epsilon$ approximating $\tilde{\mathbf P}$ such that $\|\tilde{\mathbf P}_\epsilon(\underline{a}_\tau)\|=1$, 
$({\rm Id}_{\Mbb_k(\mathbb C)}\otimes\tau)(\tilde{\mathbf P}_\epsilon\tilde{\mathbf P}_\epsilon^\star)$ is invertible, and
$\frac{1}{\|\tilde{\mathbf P}_\epsilon(\underline{X})\|_{L^2({\rm tr}_k\otimes\tau)}}>M_d-\epsilon$. Then the above majorization by 
$\kappa_{\tau,d}(\uA,\uA^\star)(I_k)$ holds for $\tilde{\mathbf P}_\epsilon$ as well. This allows us to write
$$
\kappa_{\tau,d}(\uA,\uA^\star)(I_k)=\max_{\mathbf P\in\mathbb M_k(\mathbb C\langle\uX\rangle_d)}\frac{\mathbf P(\uA)(I_k)^*
\mathbf P(\uA)(I_k)}{\|\mathbf P(\underline{X})\|_{L^2({\rm tr}_k\otimes\tau)}^2},
$$
imposing the same invertibility condition as for the previously displayed relation.

\begin{proof}[Proof of Proposition \ref{Siciak}]
With Definition \ref{BM-C} in hand, we assume that $\tau$ does satisfy a $\mathfrak C$-Bernstein-Markov condition. Then 
$\left\|\frac{1}{M_d}\mathbf P_{\uA}(\underline{a}_\tau)\right\|\le1$ (recall the meaning of $\mathbf P_{\underline{A}}$ from the
very beginning of Section \ref{trentaquatro}), so that $\Phi_{\tau,d}^\infty(\uA)
\geq\left\|\frac{1}{M_d}\mathbf P_{\uA}(\underline{A})(I_k)\right\|^2$, which, according to Remark \ref{Remark33}(2), implies
$$
\Phi^\infty_\tau(\uA)\ge\limsup_{d\to\infty}\left\|\frac{1}{M_d}\mathbf P_{\uA}(\underline{A})(I_k)\right\|^\frac{2}{d}
\ge\frac{1}{\displaystyle\limsup_{d\to\infty}M_d^\frac2d}\limsup_{d\to\infty}\|\kappa_{\tau,d}(\uA,\uA^\star)(I_k)^\frac1d\|.
$$
Applying the hypothesis that $\tau$ satisfies the $\mathfrak C$-Bernstein-Markov property yields
\begin{equation}\label{Si}
\mathfrak C^2\Phi^\infty_\tau(\uA)\ge\limsup_{d\to\infty}\|\kappa_{\tau,d}(\uA,\uA^\star)(I_k)^\frac1d\|.
\end{equation}
Conversely, if we consider the embodiment of $\kappa_{\tau,d}(\uA,\uA^\star)(I_k)$ as the maximum after $\mathbf P$ of $\mathbf P(\uA)(I_k)^*
({\rm Id}_{\Mbb_k(\mathbb C)}\otimes\tau)(\mathbf{PP}^\star)^{-1}\mathbf P(\uA)(I_k)$ 
as in Theorem \ref{th:christoffeloptim}, then we may renormalize $\mathbf P$ with a scalar constant so that 
$\|\mathbf P(\underline{a}_\tau)\|=1$. 
It follows that 
$({\rm Id}_{\Mbb_k(\mathbb C)}\otimes\tau)(\mathbf{PP}^\star\!)^{-1}\!\!\succeq\! I_k$, so 
\begin{align*}
\|\kappa_{\tau,d}(\uA,\uA^\star\!)\!(I_k)\!\|\!&=\!\|\mathbf P(\uA)(I_k)^*
({\rm Id}_{\Mbb_k(\mathbb C)}\otimes\tau)(\mathbf{PP}^\star)^{-1}\mathbf P(\uA)(I_k)\| \\
& \ge\|\mathbf P(\uA)(I_k)^*\mathbf P(\uA)(I_k)\|
\end{align*}
for all $\mathbf P$, so $\|\kappa_{\tau,d}(\uA,\uA^\star)(I_k)\|\ge\Phi_{\tau,d}^\infty(\uA)$. This proves that
\begin{equation}\label{Sinfty}
\limsup_{d\to\infty}\|\kappa_{\tau,d}(\uA,\uA^\star)(I_k)^\frac1d\|\ge\Phi_{\tau}^\infty(\uA),\quad \uA\in\mathbb M_k(\mathbb C)^n,k\in\mathbb N.
\end{equation}
Under the assumption of $\mathfrak C$-Bernstein-Markov property for $\tau$, equation \eqref{Si} guarantees that 
$\mathfrak C^2\Phi^\infty_\tau(\uA)\ge\limsup_{d\to\infty}\|\kappa_{\tau,d}(\uA,\uA^\star)(I_k)^\frac1d\|\ge\Phi_{\tau}^\infty(\uA)$,
which, for $\mathfrak C=1$ (i.e. the Bernstein-Markov property), yields
\begin{equation}\label{Sinfinity}
\Phi^\infty_\tau(\uA)=\limsup_{d\to\infty}\|\kappa_{\tau,d}(\uA,\uA^\star)(I_k)^\frac1d\|,\quad \uA\in\mathbb M_k(\mathbb C)^n,k\in\mathbb N.
\end{equation}

Let us consider next the other noncommutative version of Siciak's extremal function, namely $\Phi^2_\tau$.  For fixed $d$, we have seen above that 
we may renormalize any polynomial $\mathbf P$ in Remark \ref{Remark33}(2) so that $\|\mathbf P(\underline{a}_\tau)\|=1$, which yields 
$\kappa_{\tau,d}(\uA,\uA^\star)(I_k)\succeq\mathbf P(\uA)(I_k)^*({\rm Id}_{\Mbb_k(\mathbb C)}\otimes\tau)(\mathbf{PP}^\star)^{-1}
\mathbf P(\uA)(I_k).$ As $({\rm Id}_{\Mbb_k(\mathbb C)}\otimes\tau)(\mathbf{PP}^\star)=({\rm Id}_{\Mbb_k(\mathbb C)}\otimes\text{tr})
(\mathbf P(\underline{a}_\tau)\mathbf P(\underline{a}_\tau)^\star),$ and $({\rm Id}_{\Mbb_k(\mathbb C)}\otimes\text{tr})$
is a conditional expectation and hence (completely) positive, $\|\mathbf P(\underline{a}_\tau)\|\le1\iff\mathbf P(\underline{a}_\tau)
\mathbf P(\underline{a}_\tau)^\star\preceq I_k\otimes1\implies({\rm Id}_{\Mbb_k(\mathbb C)}\otimes\text{tr})(\mathbf P(\underline{a}_\tau)
\mathbf P(\underline{a}_\tau)^\star)\preceq I_k\implies({\rm Id}_{\Mbb_k(\mathbb C)}\otimes\text{tr})(\mathbf P(\underline{a}_\tau)
\mathbf P(\underline{a}_\tau)^\star)^{-1}\succeq I_k$, which yields the conclusion
$\kappa_{\tau,d}(\uA,\uA^\star)(I_k)\succeq\mathbf P(\uA)(I_k)^*\mathbf P(\uA)(I_k)$. Thus, 
$\text{tr}_k(\kappa_{\tau,d}(\uA,\uA^\star)(I_k))\ge\text{tr}_k(\mathbf P(\uA)(I_k)^*\mathbf P(\uA)(I_k))$ for all $\mathbf P$ as above, so that 
$\text{tr}_k(\kappa_{\tau,d}(\uA,\uA^\star)(I_k))\ge\Phi^2_{\tau,d}(\uA)$. As $\uA$ is arbitrary, we obtain
$$
\limsup_{d\to\infty}\text{tr}_k(\kappa_{\tau,d}(\uA,\uA^\star)(I_k))^\frac1d\ge\Phi^2_{\tau}(\uA).
$$
Under the assumption of the $\mathfrak C$-Bernstein-Markov property for $\tau$, Remark \ref{Remark33}(2) allows us to consider a polynomial $\mathbf P_{\uA}\in
\mathbb M_k(\mathbb C\langle\uX\rangle_d)$ such that $\kappa_{\tau,d}(\uA,\uA^\star)(I_k)=\mathbf P_{\uA}(\uA)(I_k)^\star\mathbf P_{\uA}(\uA)
(I_k)$ and $({\rm Id}_{\mathbb M_k(\mathbb C)}\otimes\text{tr})(\mathbf P_{\uA}(\underline{a}_\tau)\mathbf P_{\uA}(\underline{a}_\tau)^\star)
=I_k$. This naturally yields $\Phi^2_{\tau,d}(\uA)^{1/d}\ge\text{tr}_k\left(\frac{\mathbf P_{\uA}(\uA)(I_k)^\star\mathbf P_{\uA}(\uA)(I_k)}{\|\mathbf 
P_{\uA}(\underline{a}_\tau)\|^2}\right)^\frac1d=\frac{1}{\|\mathbf P_{\uA}(\underline{a}_\tau)\|^{2/d}}\text{tr}_k(\kappa_{\tau,d}(\uA,\uA^\star)
(I_k))^\frac1d$. Definition \ref{BM-C} \ guarantees \  that \ $\ \frac{1}{\|\mathbf P_{\uA}(\underline{a}_\tau)\|^{2/d}}\ \ge\ 
\frac{1}{M_d^{2/d}(\text{tr}_k\otimes\text{tr})(\mathbf P_{\uA}(\underline{a}_\tau)^\star\mathbf P_{\uA}(\underline{a}_\tau))^{1/d}}\ 
=\newline\noindent\frac{1}{M_d^{2/d}\text{tr}_k(I_k)^{1/d}}=\frac{1}{M_d^{2/d}}$. As $\frac{1}{\limsup_d M^{2/d}}=\frac{1}{\mathfrak C^2}$, one has $\Phi^2_{\tau}(\uA)=
\limsup_d\Phi^2_{\tau,d}(\uA)^{1/d}\ge\frac{1}{\mathfrak C^2}\limsup_d\text{tr}_k(\kappa_{\tau,d}(\uA,\uA^\star)(I_k))^\frac1d$. Together with the above-displayed relation, 
it yields the following estimate for distributions $\tau$ that satisfy the $\mathfrak C$-Bernstein-Markov inequality
\begin{equation}\label{S2}
\mathfrak C^2\Phi^2_{\tau}(\uA)\ge\limsup_{d\to\infty}\text{tr}_k(\kappa_{\tau,d}(\uA,\uA^\star)(I_k))^\frac1d\ge\Phi^2_{\tau}(\uA)
\end{equation}
and the following equality for distributions that satisfy the Bernstein-Markov property:
\begin{equation}\label{Si2}
\Phi^2_{\tau}(\uA)=\limsup_{d\to\infty}\text{tr}_k(\kappa_{\tau,d}(\uA,\uA^\star)(I_k))^\frac1d,\quad \uA\in\mathbb M_k(\mathbb C)^n,k\in\mathbb N.
\end{equation}

Finally, let us establish the relation between the Christoffel-Darboux kernel of $\tau$ and the functions $\Sigma_\tau^\bullet,\bullet\in\{2,\infty\}$, introduced 
in \eqref{Sigma2}--\eqref{Sigma8}. First observe that if $\sup_{d\in\mathbb N}(\Phi^\bullet_{\tau,d}(\underline{A}))^{1/d}$ is {\em not} achieved at a 
finite $d$, then $\Sigma_\tau^\bullet=\Phi^\bullet_\tau$, and relations \eqref{Si}, \eqref{Sinfty}, and/or \eqref{S2} hold. Thus, assume that 
$\Sigma_\tau^\bullet(\uA)=\sup_{d\in\mathbb N}(\Phi^\bullet_{\tau,d}(\uA))^\frac{1}{d}=\max_{d\in\mathbb N}(\Phi^\bullet_{\tau,d}(\underline{A}))^\frac{1}{d}=
\Phi^\bullet_{\tau,d_0}(\underline{A}))^{1/d_0}$. As before, we pick $\mathbf P_0\in\mathbb M_k(\mathbb C\langle\uX\rangle_{d_0})$ such that 
$\|\mathbf P_0(\underline{a}_\tau)\|=1$ and $\Sigma_\tau^\bullet(\uA)=\Phi^\bullet_{\tau,d_0}(\underline{A}))^\frac{1}{d_0}=\|\mathbf P_0(\underline{A})(I_k)\|_2^2$. Then 
$\kappa_{\tau,d_0}(\uA,\uA^\star)(I_k)\succeq\mathbf P_0(\underline{A})(I_k)^\star({\rm Id}_{\mathbb M_k(\mathbb C)}\otimes\mathrm{tr})(\mathbf P_0^\star(\underline{a}_\tau)
\mathbf P_0(\underline{a}_\tau))^{-1}\mathbf P_0(\underline{A})(I_k)\succeq\mathbf P_0(\underline{A})(I_k)^\star\mathbf P_0(\underline{A})(I_k)$, which immediately
implies ${\rm tr}_k(\kappa_{\tau,d_0}(\uA,\uA^\star)(I_k))^\frac{1}{d_0}\ge\Sigma^2_\tau(\underline{A}),$
$\|\kappa_{\tau,d_0}(\uA,\uA^\star)(I_k)\|^\frac{1}{d_0}\ge\Sigma^\infty_\tau(\underline{A})$. 

For the opposite inequality, we again take a polynomial 
$\mathbf P_0\in\mathbb M_k(\mathbb C\langle\uX\rangle_{d_0})$, but this time chosen to achieve $\kappa_{\tau,d_0}$: $({\rm Id}_{\mathbb M_k(\mathbb C)}\otimes\mathrm{tr})
(\mathbf P_0^\star(\underline{a}_\tau)\mathbf P_0(\underline{a}_\tau))=\|\mathbf P_0\|^2_{L^2({\rm tr}_k\otimes\tau)}I_k$ and $\kappa_{\tau,d_0}(\uA,\uA^\star)(I_k)=
\mathbf P_0(\underline{A})(I_k)^\star({\rm Id}_{\mathbb M_k(\mathbb C)}\otimes\mathrm{tr})(\mathbf P_0^\star(\underline{a}_\tau)
\mathbf P_0(\underline{a}_\tau))^{-1}\mathbf P_0(\underline{A})(I_k)
=\frac{\mathbf P_0(\underline{A})(I_k)^\star\mathbf P_0(\underline{A})(I_k)}{\|\mathbf P_0\|^2_{L^2({\rm tr}_k\otimes\tau)}}.$ We assume that 
$\tau$ satisfies the $\mathfrak C$-Bernstein-Markov property for some $\mathfrak C\in[1,+\infty)$. Then, as noted before, 
$$
\mathfrak C\leq\sup_{d\in\mathbb N}M_d^\frac1d<\infty
$$
depends only on $\tau$. We thus have
$\|\kappa_{\tau,d_0}(\uA,\uA^\star)(I_k)\|^\frac{1}{d_0}=
\frac{\|\mathbf P_0(\underline{A})(I_k)^\star\mathbf P_0(\underline{A})(I_k)\|^\frac{1}{d_0}}{\|\mathbf P_0\|^{2/d_0}_{L^2({\rm tr}_k\otimes\tau)}},$
${\rm tr}_k(\kappa_{\tau,d_0}(\uA,\uA^\star)(I_k))^{1/d_0}=
\frac{{\rm tr}_k(\mathbf P_0(\underline{A})(I_k)^\star\mathbf P_0(\underline{A})(I_k))^{1/d_0}}{\|\mathbf P_0\|^{2/d_0}_{L^2({\rm tr}_k\otimes\tau)}}.$
Since by definition $M_{d_0}\ge\frac{\|\mathbf P_0(\underline{a}_\tau)\|}{\|\mathbf P_0\|_{L^2({\rm tr}_k\otimes\tau)}}$, one has
${\displaystyle\|\kappa_{\tau,d_0}(\uA,\uA^\star)(I_k)\|^\frac{1}{d_0}\le\sup_{d\in\mathbb N}M_d^\frac2d}
\left[\frac{\|\mathbf P_0(\underline{A})(I_k)^\star\mathbf P_0(\underline{A})(I_k)\|}{\|\mathbf P_0(\underline{a}_\tau)\|^{2}}\right]^\frac{1}{d_0}\le
{\displaystyle\sup_{d\in\mathbb N}M_d^{\frac2d}\Sigma^\infty_\tau(\uA)},$
${\rm tr}_k(\kappa_{\tau,d_0}(\uA,\uA^\star)(I_k))^\frac{1}{d_0}\le
{\displaystyle\sup_{d\in\mathbb N}M_d^{\frac2d}\Sigma^2_\tau(\uA)}.$
\end{proof}

It should be noted that the above proposition does {\em not} state that the quantities involved are finite: neither of $\Phi^\infty_\tau,\Phi^2_\tau,\Sigma_\tau^\infty,\Sigma_\tau^2$ 
is guaranteed to be finite anywhere. The usefulness of this proposition in the study of the noncommutative Christoffel-Darboux kernel $\kappa_\tau$ 
depends on how well we understand the Siciak functions $\Phi$. In the classical pluripotential theory of several complex variables, the importance of 
$\Phi$ springs from the fact that it usually equals Green's function associated to the support of the classical distribution $\tau$. In the noncommutative 
context, we do not even have a notion of support for $\tau.$ We will go around this inconvenient fact (and also indirectly define a notion of support - or, 
in a certain sense, closure of support - for $\tau$) by making use of the classical theory of plurisubharmonic functions, which, as noted above, applies to 
the functions $\Phi$.

We have noted that Definition \ref{BM} actually generalizes the classical Bernstein-Markov property. Let us put this in context through an example:

\begin{example}\label{Wigner}
Let $\tau$ be the semicircular (Wigner) distribution on $\mathbb R$: $\tau\colon\mathbb C\langle X\rangle\to\mathbb C,\tau(X^d)=\frac{2}{\pi}\int_{-1}^1t^d\sqrt{1-t^2}\,{\rm d}t$.
It is well-known (see for instance \cite{BLPW}), and easy to prove, that $\tau$ has the Bernstein-Markov property in the classical sense:
for instance, the orthonormal polynomials of $\tau$ are known to be the Chebyshev polynomials of the second kind, denoted $\{U_d\}_{d\in\mathbb N}$,
and $U_d(\cos\theta)=\frac{\sin\left((d+1)\theta\right)}{\sin\theta}$, $d\in\mathbb N$. The rate of growth of
their infinity norm is of order ${d}$ as the degree $d$ tends to infinity (more precisely, the norm is reached at the boundary of the
support, i.e. when $\theta=0$ or $\theta=\pi$, and then $|U_d(\pm1)|=d+1$). If one writes an arbitrary polynomial $P\in\mathbb C\langle X\rangle_d$
as $P(X)=\sum_{v=0}^dc_vU_v(X)$, then $\|P(a_\tau)\|\le\sum_{v=0}^d|c_v|\|U_v(a_\tau)\|\le\sum_{v=0}^d|c_v|{(v+1)}
\le(d+1)\sum_{v=0}^d|c_v|$, while $\|P(a_\tau)\|_2=\sqrt{\sum_{v=0}^d|c_v|^2}$. By the Schwarz-Cauchy inequality, 
$\sum_{v=0}^d|c_v|\le\sqrt{d+1}\sqrt{\sum_{v=0}^d|c_v|^2}$, so that 
$\|P(a_\tau)\|\le(d+1)^\frac32\|P(a_\tau)\|_2$ for all $P\in\mathbb C\langle X\rangle_d$. Since $\lim_{d\to\infty}(d+1)^\frac{3}{2d}=1$,
$\tau$ satisfies the classical Bernstein-Markov property. As shown above, this means $\tau$ satisfies the Bernstein-Markov property for 
matrix-valued polynomials as well. Now, Proposition \ref{Siciak} shows that $\Phi^\infty_\tau(A)=\limsup_{d\to\infty}\|\kappa_{\tau,d}(A,A^\star)(I_k)^\frac1d\|$
and $\Phi^2_\tau(A)=\limsup_{d\to\infty}{\rm tr}_k(\kappa_{\tau,d}(A,A^\star)(I_k))^\frac1d$, $A\in\mathbb M_k(\mathbb C)$.\\
Functional calculus shows that for fixed $k$ the growth of $\|\kappa_{\tau,d}(A,A^\star)(I_k)\|$ as $d\to\infty$ is governed by the eigenvalues of $A$,
so the existence of $\lim_{d\to\infty}\|\kappa_{\tau,d}(A,A^\star)(I_k)^\frac1d\|$ is simply guaranteed by the convergence of 
$\kappa_{\tau,d}(z,\overline{z})(1)^\frac1d$ when $z\in\mathbb C$, which is known from classical, one-variable potential theory.
Specifically, by using analytic functional calculus, one obtains a manageable expression for
$U_d(A)$ when $A\in\mathbb M_k(\mathbb C)$ is in upper triangular form $A=D+T,$ where $D$ is diagonal and $T$ strictly upper triangular. With the notation 
$D=\text{diag}(z_1,\dots,z_k),T=(t_{i,i+j})_{1\le i\le k,1\le j\le k-i}$, we make the assumptions that $z_i\neq z_l$ if $l\neq i$ (which excludes a set of Lebesgue
measure zero in $\mathbb M_k(\mathbb C)$), and that $|z_1|\le\cdots\le|z_k|$ (which implies no loss of generality, as any upper triangular matrix is unitarily 
equivalent to one whose eigenvalues are ordered increasingly). We then have the following formula for $U_d(A)$:
\begin{eqnarray*}
\lefteqn{\left(U_d(A)\right)_{i,i+j}=\left[\frac{U_d(z_i)}{z_i-z_{i+j}}+\frac{U_d(z_{i+j})}{z_{i+j}-z_i}\right]t_{i,i+j}+}\\
& &\!\!\sum_{l=1}^{j-1}\left[\sum_{i<i_1<\cdots<i_l<i+j}\left[\sum_{s=i,i_1,\dots,i_l,i+j}\!\frac{U_d(z_s)}{\displaystyle\prod_{r=i,i_1,\dots,i_l,i+j,r\neq s}\!\![z_s-z_{r}]}\right]\! 
t_{i,i_1}t_{i_1,i_2}\cdots t_{i_l,i+j}\right]\!,
\end{eqnarray*}
$\left(U_d(A)\right)_{i,i}=U_d(z_i)$, while $\left(U_d(A)\right)_{i,k}=0$ if $k<i$. The formula does not make sense if $z_r=z_s$ for some $r\neq s$, but 
the expression is Lipschitz in all its diagonal entries, with derivatives of $U_d$ occuring if two or more diagonal entries of $A$ happen to be equal (specifically,
if $q$ entries are equal, then the derivatives of order up to and including $q-1$ appear). 
The terms in the above sum can be easily seen to be indexed by the paths in the 
upper right corner between the diagonal entry $(i,i)$ and the diagonal entry $(i+j,i+j)$ which are formed by segments leaving from and returning to diagonal entries.
For instance, 
$$
U_d\left(\begin{bmatrix} z_1 & t_{1,2} & t_{1,3}\\ 0 & z_2 & t_{2,3} \\ 0 & 0 & z_3 \end{bmatrix}\right)=
$$
{\tiny{$$\!\!
\begin{bmatrix}\!U_d(z_1) & \!\!\left[\!\frac{U_d(z_1)}{z_1\!-\!z_2}\!+\!\frac{U_d(z_2)}{z_2\!-\!z_1}\!\right]\!t_{1\!,2} & 
\!\!\left[\!\frac{U_d(z_1)}{z_1\!-\!z_3}\!+\!\frac{U_d(z_3)}{z_3\!-\!z_1}\!\right]t_{1\!,3}\!+\!\left[\!\frac{U_d(z_1)}{[z_1\!-\!z_2][z_1\!-\!z_3]}\!
+\!\frac{U_d(z_2)}{[z_2\!-\!z_1][z_2\!-\!z_3]}\!+\!\frac{U_d(z_3)}{[z_3\!-\!z_1][z_3\!-\!z_2]}\!\right]\!t_{1\!,2}t_{2\!,3} \\ 
0 & U_d(z_2) & \left[\frac{U_d(z_2)}{z_2-z_3}+\frac{U_d(z_3)}{z_3-z_2}\right]t_{2,3} \\ 0 & 0 & U_d(z_3) \end{bmatrix}
$$}}\noindent
for any $z_1\neq z_2\neq z_3\neq z_1\in\mathbb C$. If any two of the three diagonal entries happen to be equal to each other, the first derivative of $U_d$ evaluated in that entry
is involved in the off-diagonal entries, and if all three are equal, so is the second derivative of $U_d$. (Probably the easiest way to see these facts is by applying analytic functional
calculus and the formulae for resolvents of upper triangular matrices: $U_d(A)=(2\pi i)^{-1}\int_\gamma U_d(\zeta)(\zeta I_k-A)^{-1}\,{\rm d}\zeta$ for any smooth curve $\gamma$
surrounding $\sigma(A)$ once.)

More general examples can be derived from the estimates obtained in Section \ref{1var}: specifically, assuming that the spectrum $\sigma(a_\tau)$ of our generating variable
is sufficiently large so that the compact subset $S_{a_\tau,k}:=\{A=A^\star\in\mathbb M_k(\mathbb C)\colon\sigma(A)\subset\sigma(a_\tau)\}$ of $\mathbb S_k$ is not pluripolar
in $\mathbb M_k(\mathbb C)=\mathbb C^{k^2}$ (which is the case, for instance, if $\sigma(a_\tau)$ has nonempty interior\footnote{This immediately implies that 
the set $S_{a_\tau,k}$ has nonempty interior in $\mathbb S_k$, so that it cannot be pluripolar.}, 
which happens trivially for the semicircular) we have seen that $\Phi_{\tau,d}^2(A)\le1$ for all $A\in S_{a_\tau,k}$. This, together with the estimate in Section \ref{grow},
implies that both $\Sigma_\tau^2$ and $\Phi_\tau^2$ are finite q.e. on $\mathbb M_k(\mathbb C)$
(see, for instance, \cite[Theorem 1.6, Appendix B]{ST}). Moreover, they are also no greater than one on $S_{a_\tau,k}$.
Since $\|\cdot\|_2$ and $\|\cdot\|$ are comparable on $\mathbb M_k(\mathbb C)$,
it follows that $\Sigma_\tau^\infty$ and $\Phi_\tau^\infty$ are finite q.e. on $\mathbb M_k(\mathbb C)$ as well.

\end{example}

A typical noncommutative context is provided by the case of free products. Let us consider next such an example.

\begin{example}\label{L(F2)}
Consider two copies $\tau_1,\tau_2$ of the semicircular (Wigner) distribution from Example \ref{Wigner} above, and take their {\em free product}
$\tau=\tau_1*\tau_2\colon\mathbb C\langle X_1,X_2\rangle\to\mathbb C$: $\tau$ is completely specified by the following evaluation rules: $\tau$ is $\mathbb C$-linear,
$\tau(P(X_j))=\tau_j(P(X)),j=1,2$, and $\tau(P_{i_1}(X_{i_1})P_{i_2}(X_{i_2})\cdots P_{i_r}(X_{i_r}))=0$ whenever $r\in\mathbb N,r\ge1,i_1,i_2,\dots,i_r\in\{1,2\}$
are such that $i_1\neq i_2\neq\cdots\neq i_r$ and $\tau(P_{i_j}(X_{i_j}))=0,1\le j\le r$. It is known that $\tau$ is a faithful, bounded, tracial state (see \cite{V1} for details).
As mentioned in Example \ref{ortex}(3), a family of orthonormal polynomials for $\tau$ is $\{U_w\}_{w\in\langle X_1,X_2\rangle}$,
$U_w(X_1,X_2)=U_{d_{i_1}}(X_{i_1})U_{d_{i_2}}(X_{i_2})\cdots U_{d_{i_r}}(X_{i_r})$ whenever $w=X_{i_1}^{d_{i_1}}X_{i_2}^{d_{i_2}}\cdots X_{i_r}^{d_{i_r}},$
$r\in\mathbb N,i_1,i_2,\dots,i_r\in\{1,2\},i_1\neq i_2\neq\cdots\neq i_r$. Here, as before, $U_{d_{i_j}}(X_{i_j})$ are the Chebyshev polynomials of the second kind introduced
in Example \ref{Wigner} above.

First, it follows from Section \ref{monot} (and 
has been argued in item \eqref{max} in the list following Section \ref{coordchange}) that 
$\Phi^\bullet_\tau(A_1,A_2)\ge\max\{\Phi^\bullet_{\tau_1}(A_1),\Phi^\bullet_{\tau_2}(A_2)\},A_1,A_2\in\mathbb M_k(\mathbb C)$ (with a similar statement for $\Sigma$),
so that $\Phi^\bullet_\tau,\Sigma^\bullet_\tau$ are not constantly $1$. 

Let us show that $\Sigma^\bullet_\tau$, and hence $\Phi^\bullet_\tau,\bullet\in\{2,\infty\}$, is finite q.e. 
For an arbitrary $\mathbf f\in\mathbb M_k(\mathbb C\langle X_1,X_2\rangle_d)$, 
the following Fourier-like expansion holds:
\begin{eqnarray*}
\mathbf f(X_1,X_2)\!&=&\!\sum_{w\in\langle X_1,X_2\rangle_d}\!\left\langle\mathbf f,I_k\otimes U_w\right\rangle\otimes U_w(X_1,X_2)\\
\!&=&\!\sum_{w\in\langle X_1,X_2\rangle_d}\!(\textrm{Id}_{\mathbb M_k(\mathbb C)}\otimes\tau)((I_k\otimes U_w^\star)\mathbf f)\otimes U_w(X_1,X_2).
\end{eqnarray*}
Then
\begin{eqnarray*}
\lefteqn{\mathbf f(A_1,A_2)(I_k)^\star \mathbf f(A_1,A_2)(I_k)}\\
&=&\!\!\!\!\!\!\!\!\sum_{v,w\in\langle X_1,X_2\rangle_d}\!\!\!\!\!\!\!U_v(A_1,A_2)^\star({\rm Id}_{\mathbb M_k(\mathbb C)}\otimes\tau)(\mathbf f^\star I_k\otimes U_v)
(\textrm{Id}_{\mathbb M_k(\mathbb C)}\!\otimes\!\tau)(I_k\!\otimes\! U_w^\star\mathbf f) U_w(A_1,A_2).
\end{eqnarray*}
For $v=w$, then 
\begin{eqnarray*}
\lefteqn{U_w(A_1,A_2)^\star(\textrm{Id}_{\mathbb M_k(\mathbb C)}\!\otimes\!\tau)(\mathbf f^\star I_k\!\otimes\! U_w)
(\textrm{Id}_{\mathbb M_k(\mathbb C)}\!\otimes\!\tau)(I_k\!\otimes\! U_w^\star\mathbf f)U_w(A_1,A_2)}\\
&\quad\quad\quad \le & \|\mathbf f({a}_{\tau_1},{a}_{\tau_2})\|^2U_w(A_1,A_2)^\star[\tau\left(U_wU_w^\star\right)I_k]U_w(A_1,A_2)\\
&\quad\quad\quad = & \|\mathbf f({a}_{\tau_1},{a}_{\tau_2})\|^2U_w(A_1,A_2)^\star U_w(A_1,A_2),
\end{eqnarray*}
for all $w\in\langle X_1,X_2\rangle_d$. Elements $v\neq w\in\langle X_1,X_2\rangle_d$ appear obviously in pairs, so we can group them accordingly:
$U_v(A_1,A_2)^\star({\rm Id}_{\mathbb M_k(\mathbb C)}\otimes\tau)(\mathbf f^\star I_k\otimes U_v)
(\textrm{Id}_{\mathbb M_k(\mathbb C)}\!\otimes\!\tau)(I_k\!\otimes\! U_w^\star\mathbf f) U_w(A_1,A_2)+
U_w(A_1,A_2)^\star({\rm Id}_{\mathbb M_k(\mathbb C)}\otimes\tau)(\mathbf f^\star I_k\otimes U_w)
(\textrm{Id}_{\mathbb M_k(\mathbb C)}\!\otimes\!\tau)(I_k\!\otimes\! U_v^\star\mathbf f) U_v(A_1,A_2)$.
The positivity of $(V\pm W)^\star(V\pm W)$ guarantees that $-V^\star V-W^\star W\le V^\star W+W^\star V\le V^\star V+W^\star W$. Applying this 
observation to our sums yields
\begin{eqnarray*}
\lefteqn{\pm[U_v(A_1,A_2)^\star({\rm Id}_{\mathbb M_k(\mathbb C)}\otimes\tau)(\mathbf f^\star I_k\otimes U_v)
(\textrm{Id}_{\mathbb M_k(\mathbb C)}\!\otimes\!\tau)(I_k\!\otimes\! U_w^\star\mathbf f) U_w(A_1,A_2)}\\
& & \mbox{}+U_w(A_1,A_2)^\star({\rm Id}_{\mathbb M_k(\mathbb C)}\otimes\tau)(\mathbf f^\star I_k\otimes U_w)
(\textrm{Id}_{\mathbb M_k(\mathbb C)}\!\otimes\!\tau)(I_k\!\otimes\! U_v^\star\mathbf f) U_v(A_1,A_2)]\\
& \leq & \|\mathbf f({a}_{\tau_1},{a}_{\tau_2})\|^2\left(U_w(A_1,A_2)^\star U_w(A_1,A_2)+U_v(A_1,A_2)^\star U_v(A_1,A_2)
\right).
\end{eqnarray*}
We obtain the (rather brutal) majorization 
\begin{eqnarray}
\lefteqn{\mathbf f(A_1,A_2)(I_k)^\star \mathbf f(A_1,A_2)(I_k)}\nonumber\\
& \quad\quad\quad\leq & \|\mathbf f(\underline{a}_{\tau_1},\underline{a}_{\tau_2})\|^2\left[
\sum_{w\in\langle X_1,X_2\rangle_d}U_w(A_1,A_2)^\star U_w(A_1,A_2)\right.\nonumber\\
& & \mbox{}+\left.
\sum_{\substack{v,w\in\langle X_1,X_2\rangle_d \\ v<_{\rm gl}w}}U_w(A_1,A_2)^\star U_w(A_1,A_2)+U_v(A_1,A_2)^\star U_v(A_1,A_2)\right].\label{344}
\end{eqnarray}
We estimate next the product $U_v(A_1,A_2)^\star U_v(A_1,A_2)$. Given the specific form of $U_v(X_1,X_2)$, let us first focus on $U_d(X)$.
As $\left\|\frac{\mathbf P(\uA)(I_k)}{\|\mathbf P({a}_{\tau_1})\|}\right\|_\bullet^2\leq\Phi^\bullet_{\tau_1,d}(\uA)$ according to 
the definition of $\Phi^\bullet_{\tau_1,d}$, by employing the embedding $U_d\mapsto I_k\otimes U_d$ we have the estimate $\|U_d(A_j)\|^2_\bullet\le\|U_d(a_{\tau_j})\|^2
\Phi^\bullet_{\tau_j,d}(A_j)$.
According to \eqref{Sigma2} and \eqref{Sigma8}, by taking power $1/d$, we obtain $\|U_d(A_j)\|^{2/d}_\bullet\le\|U_d(a_{\tau_j})\|^{2/d}
\Sigma^\bullet_{\tau_j}(A_j),$  $A_j\in\mathbb M_k(\mathbb C), k\in\mathbb N$.

Consider the case $\bullet=\infty$. Given $w=X_{i_1}^{d_{i_1}}X_{i_2}^{d_{i_2}}\cdots X_{i_r}^{d_{i_r}}\in\langle X_1,X_2\rangle_d,$ we have 
\begin{eqnarray*}
\lefteqn{\|U_w(A_1,A_2)^\star U_w(A_1,A_2)\|}\\
&\leq&\|U_{d_{i_1}}(A_{i_1})\|^2\cdots\|U_{d_{i_r}}(A_{i_r})\|^2\\
&\le&\|U_{d_{i_1}}(a_{\tau_{i_1}})\|^2\cdots\|U_{d_{i_r}}(a_{\tau_{i_r}})\|^2\Sigma^\infty_{\tau_{i_1}}(A_{i_1})^{d_{i_1}}\cdots\Sigma^\infty_{\tau_{i_r}}(A_{i_r})^{d_{i_r}}\\
&\le&(d_{i_1}+1)^3\cdots(d_{i_r}+1)^3\Sigma^\infty_{\tau_{i_1}}(A_{i_1})^{d_{i_1}}\cdots\Sigma^\infty_{\tau_{i_r}}(A_{i_r})^{d_{i_r}}.
\end{eqnarray*}
We choose the $w$ for which $\displaystyle\max_{v\in\langle X_1,X_2\rangle_d}\|U_v(A_1,A_2)\|$ is achieved. By a most brutal estimate in \eqref{344},
\begin{align}
\|\mathbf f(A_1,A_2)&(I_k)\|^2\nonumber
 \leq 2^{2d}\|\mathbf f({a}_{\tau_1},{a}_{\tau_2})\|^2\max_{v\in\langle X_1,X_2\rangle_d}\|U_v(A_1,A_2)\|\nonumber\\
&\le4^d\|\mathbf f({a}_{\tau_1},{a}_{\tau_2})\|^2(d_{i_1}+1)^3\cdots(d_{i_r}+1)^3\Sigma^\infty_{\tau_{i_1}}(A_{i_1})^{d_{i_1}}\cdots\Sigma^\infty_{\tau_{i_r}}(A_{i_r})^{d_{i_r}},\nonumber
\end{align}
so that, by dividing with $\|\mathbf f({a}_{\tau_1},{a}_{\tau_2})\|^2$ and taking power $1/d$,
\begin{eqnarray}
\Sigma^\infty_{\tau}(A_1,A_2)\nonumber
& \leq & 4[(d_{i_1}+1)\cdots(d_{i_r}+1)]^{3/d}\Sigma^\infty_{\tau_{i_1}}(A_{i_1})^{d_{i_1}/d}\cdots\Sigma^\infty_{\tau_{i_r}}(A_{i_r})^{d_{i_r}/d}\nonumber\\
& < & 32\max\{\Sigma^\infty_{\tau_1}(A_1),\Sigma^\infty_{\tau_2}(A_2)\}.\label{345}
\end{eqnarray}
This shows that all of $\Sigma^\infty_{\tau}(A_1,A_2),\Sigma^2_{\tau}(A_1,A_2),\Phi^\infty_{\tau}(A_1,A_2),\Phi^2_{\tau}(A_1,A_2)$ are non-trivial and finite q.e.


By Remark \ref{Remark33}(4), the characterization of $\kappa_{\tau,d}$ as minimum over certain sets of polynomials as in Theorem \ref{th:christoffeloptim}
(and parts (2) and (3) of the same Remark \ref{Remark33}) holds also when $\kappa_{\tau,d}$ is evaluated on tuples of elements in a finite von Neumann algebra.
As the entire proof of Proposition \ref{Siciak} is based on this characterization, it follows that it holds also for $\uA$ replaced by $\underline{a}_\tau\in W^*(\tau)^n$.
In particular, in our case, ${\rm tr}(\kappa_{\tau,d}(({a}_{\tau_1},{a}_{\tau_2}),({a}_{\tau_1},{a}_{\tau_2}))(1))=2^d$ from the definition of 
$\kappa_{\tau,d}$ and of orthonormal polynomials. On the other hand, by definition $\Phi_{\tau,d}^2(({a}_{\tau_1},{a}_{\tau_2}))=
\sup\{{\rm tr}(\mathbf P(({a}_{\tau_1},{a}_{\tau_2}))(1)^\star\mathbf P(({a}_{\tau_1},{a}_{\tau_2}))(1))\colon
\|\mathbf P(({a}_{\tau_1},{a}_{\tau_2}))\|\le1,\mathbf P\in W^*(\tau_1*\tau_2)\otimes\mathbb C\langle\uX\rangle_d\}\leq1$ because 
${\rm tr}(\mathbf P(({a}_{\tau_1},{a}_{\tau_2}))(1)^\star\mathbf P(({a}_{\tau_1},{a}_{\tau_2}))(1))\le
\|\mathbf P(({a}_{\tau_1},{a}_{\tau_2}))\|^2$. Thus,
$\Phi_\tau^2(({a}_{\tau_1},{a}_{\tau_2}))=1,$ while $\limsup_{d\to\infty}
{\rm tr}(\kappa_{\tau,d}(({a}_{\tau_1},{a}_{\tau_2}),({a}_{\tau_1},{a}_{\tau_2}))(1))=2$. Thus, the free product 
$\tau=\tau_1*\tau_2$ does {\em not} satisfy the Bernstein-Markov property, even though it does satisfy a $\mathfrak C$-Bernstein-Markov property
for some $2\le\mathfrak C\le8\sqrt{2}$, as it will be seen in Remark \ref{BMBound}.

\end{example}

The reader might legitimately hope that some of the inequalities in \eqref{Cinfty}--\eqref{C2} can be eventually shown to be equalities even
if $\mathfrak C>1$. It will follow from Theorem \ref{free} below, the remark following it, and Examples \ref{L(F2)} and \ref{Wigner} that this is unlikely to happen in many cases.

Some of the methods from Example \ref{L(F2)} can be applied in a more general context in order to provide some estimates the noncommutative Siciak function. In view of 
\eqref{UnitaryInvariance}, the following theorem appears to be intimately related to Voiculescu's asymptotic freeness result for independent unitarily invariant random matrices, 
and can be viewed as a partial free analogue of Siciak's Theorem \cite[Theorem 5.1.8]{Klimek}.

\begin{theorem}\label{free}
Let $\uX_1=(X_1,\dots,X_n)$ and $\uX_2=(X_{n+1},\dots,X_{n+m})$ be selfadjoint non-commuting indeterminates. Assume  $\tau_1\colon\mathbb C\langle\uX_1
\rangle\to\mathbb C$ and $\tau_2\colon\mathbb C\langle\uX_2\rangle\to\mathbb C$ are two faithful positive bounded traces, and denote $\tau=\tau_1*
\tau_2 \colon\mathbb C\langle\uX_1,\uX_2\rangle\to\mathbb C$ their free product. Suppose that $\tau_j$ satisfies the $\mathfrak C_j$-Bernstein-Markov property for some $\mathfrak
C_j\in[1,+\infty)$, $j=1,2$. Then for each $k\in\mathbb N,$ one has
\begin{align}
\max\{\Sigma_{\tau_1}^\infty(\underline{A}_1),&\Sigma^\infty_{\tau_2}(\underline{A}_2)\}\le\Sigma^\infty_\tau(\uA_1,\uA_2)\nonumber\\  
&\le (m+n)^2\max\left\{\Sigma^\infty_{\tau_{1}}(\uA_{1})\!\sup_{q\in\mathbb N}M(\tau_1)_q^\frac{2}{q},\ 
\Sigma^\infty_{\tau_2}(\uA_{2})\!\sup_{q\in\mathbb N}M(\tau_{2})_q^\frac{2}{q}\right\},
\end{align}
for all $(\underline{A}_1,\underline{A}_2)\in\mathbb M_k(\mathbb C)^n\times\mathbb M_k(\mathbb C)^m,$ and
\begin{align}
\max\{\Phi_{\tau_1}^\infty(\underline{A}_1),&\Phi^\infty_{\tau_2}(\underline{A}_2)\}\le\Phi^\infty_\tau(\uA_1,\uA_2)\nonumber\\ 
&\le (m+n)^2\max\left\{\Phi^\infty_{\tau_{1}}(\uA_{1})\!\sup_{q\in\mathbb N}M(\tau_1)_q^\frac{2}{q},\ 
\Phi^\infty_{\tau_2}(\uA_{2})\!\sup_{q\in\mathbb N}M(\tau_{2})_q^\frac{2}{q}\right\},
\end{align}
for all $(\underline{A}_1,\underline{A}_2)=(\underline{A}_1,\underline{A}_2)^\star\in\mathbb M_k(\mathbb C)^n\times\mathbb M_k(\mathbb C)^m$.
In particular, under the hypotheses of the theorem, the functions $\Sigma^\infty_\tau$, $\Sigma^2_\tau$, $\Phi^2_\tau,$ and $\Phi^\infty_\tau$ are finite q.e.
on the Euclidean space $\mathbb M_k(\mathbb C)^n\times\mathbb M_k(\mathbb C)^m$ and nontrivial.
\end{theorem}

\begin{proof}
As in Example \ref{L(F2)}, our proof parallels to some extent the proof provided in \cite{Klimek} for Theorem 5.1.8. We start by recalling that 
$\Phi_{\tau_1*\tau_2}^\bullet(\underline{A}_1,\underline{A}_2)\ge\max\{\Phi_{\tau_1}^\bullet(\underline{A}_1),\Phi^\bullet_{\tau_2}(\underline{A}_2)\}$,
$\Sigma_{\tau_1*\tau_2}^\bullet(\underline{A}_1,\underline{A}_2)\ge\max\{\Sigma_{\tau_1}^\bullet(\underline{A}_1),\Sigma^\bullet_{\tau_2}(\underline{A}_2)\}$.
Indeed, this has nothing to do with freeness, and is a conclusion of Section \ref{monot}, as it has been argued in item \eqref{max} in the list following Section \ref{1var}.
The proof of the other inequality is more challenging.

The third item of Example \ref{ortex} (as shown in \cite[Section 3]{A1}) provides an explicit formula for orthonormal polynomials corresponding to free products of 
distributions. If $\{P_{w,\tau_1}\}_{w\in\langle\uX_1\rangle}$ and $\{P_{w,\tau_2}\}_{w\in\langle\uX_2\rangle}$
are the orthonormal polynomials associated to noncommutative distributions $\tau_1$ and $\tau_2$, respectively, then  the orthonormal polynomials
$\{P_{w,\tau}\}_{w\in\langle\uX_1,\uX_2\rangle}$ corresponding to $\tau=\tau_1*\tau_2$ are indexed by words in letters from $\underline{X}_1$ and $\underline{X}_2$, 
and if $w=u_1u_{2}u_{3}\cdots u_{r}\in\langle\uX_1,\uX_2\rangle,u_j\in\langle\uX_{i_j}\rangle,1\le j\le r,r\in\mathbb N$, where $i_1\neq i_2\neq i_3\neq\cdots\neq i_r$, then
$$
P_{w,\tau}(\uX_1,\uX_2)=P_{u_1,\tau_{i_1}}(\uX_{i_1})P_{u_2,\tau_{i_2}}(\uX_{i_2})P_{u_3,\tau_{i_3}}(\uX_{i_3})\cdots
P_{u_r,\tau_{i_r}}(\uX_{i_r}).
$$
Pick $d\in\mathbb N$ and $\mathbf f\in\mathbb M_k(\mathbb C\langle\uX_1,\uX_2\rangle_d)$ such that $\|\mathbf f(\underline{a}_{\tau_1},\underline{a}_{\tau_2})\|=1$. Then
\begin{eqnarray*}
\mathbf f(\uX_1,\uX_2)&=&\sum_{w\in\langle\uX_1,\uX_2\rangle_d}\left\langle\mathbf f,I_k\otimes P_{w,\tau}\right\rangle\otimes P_{w,\tau}(\uX_1,\uX_2)\\
&=&\sum_{w\in\langle\uX_1,\uX_2\rangle_d}(\textrm{Id}_{\mathbb M_k(\mathbb C)}\otimes\tau)((I_k\otimes P_{w,\tau}^\star)\mathbf f)\otimes P_{w,\tau}(\uX_1,\uX_2).
\end{eqnarray*}
It then follows that
\begin{eqnarray*}
\lefteqn{\mathbf f(\uA_1,\uA_2)(I_k)^\star \mathbf f(\uA_1,\uA_2)(I_k)}\\
&\!\!\!\! = &\!\!\left[\sum_{w\in\langle\uX_1,\uX_2\rangle_d}\!\!\!(\textrm{Id}_{\mathbb M_k(\mathbb C)}\otimes\tau)(I_k\otimes P_{w,\tau}^\star\mathbf f)
P_w(\uA_1,\uA_2)\right]^\star\\
& & \mbox{}\times\left[\sum_{w\in\langle\uX_1,\uX_2\rangle_d}\!\!\!(\textrm{Id}_{\mathbb M_k(\mathbb C)}\otimes\tau)(I_k\otimes P_{w,\tau}^\star\mathbf f)
P_w(\uA_1,\uA_2)\right]\\
&\!\!\!\!=&\!\!\!\!\!\!\!\sum_{v,w\in\langle\uX_1\!,\uX_2\rangle_d}\!\!\!\!\!\!\!\!P_{v,\tau}\!(\uA_1\!,\!\uA_2)^\star\!({\rm Id}_{\mathbb M_k(\mathbb C)}\!\otimes\!\tau)\!
({\bf f}^\star I_k\!\otimes\!P_{v,\tau})\!(\textrm{Id}_{\mathbb M_k(\mathbb C)}\!\otimes\!\tau)\!(I_k\!\otimes\! P_{w,\tau}^\star\mathbf f)\! P_{w,\tau}\!(\uA_1\!,\!\uA_2).
\end{eqnarray*}
As in Example \ref{L(F2)}, we obtain 
\begin{eqnarray}
\lefteqn{\mathbf f(\uA_1,\uA_2)(I_k)^\star \mathbf f(\uA_1,\uA_2)(I_k)}\nonumber\\
& \quad\leq & \|\mathbf f(\underline{a}_{\tau_1},\underline{a}_{\tau_2})\|^2\left[
\sum_{w\in\langle\uX_1,\uX_2\rangle_d}P_{w,\tau}(\uA_1,\uA_2)^\star P_{w,\tau}(\uA_1,\uA_2)\right.\nonumber\\
& & \mbox{}+\left.\sum_{\substack{v,w\in\langle\uX_1,\uX_2\rangle_d\\v<_{\rm gl}w}}
P_{w,\tau}(\uA_1,\uA_2)^\star P_{w,\tau}(\uA_1,\uA_2)+P_{v,\tau}(\uA_1,\uA_2)^\star P_v(\uA_1,\uA_2)\right].\label{three47}
\end{eqnarray}
This holds for any polynomial $\mathbf f\in\mathbb M_k(\mathbb C)\otimes\mathbb C\langle\uX_1,\uX_2\rangle_d$ of degree at most $d$ which has 
norm one on $(\underline{a}_{\tau_1},\underline{a}_{\tau_2})$.

Again as in Example \ref{L(F2)},  
\begin{eqnarray*}
\lefteqn{P_{w,\tau}(\uA_1,\uA_2)^\star P_{w,\tau}(\uA_1,\uA_2)}\\
& = & P_{u_1,\tau_{i_1}}(\uA_{i_1})P_{u_2,\tau_{i_2}}(\uA_{i_2})P_{u_3,\tau_{i_3}}(\uA_{i_3})\cdots
P_{u_r,\tau_{i_r}}(\uA_{i_r})\\
& & \mbox{}\times P_{u_r,\tau_{i_r}}(\uA_{i_r})^\star\cdots P_{u_3,\tau_{i_3}}(\uA_{i_3})^\star P_{u_2,\tau_{i_2}}(\uA_{i_2})^\star P_{u_1,\tau_{i_1}}(\uA_{i_1})^\star,
\end{eqnarray*}
where we remind the reader that
$w=u_1u_{2}u_{3}\cdots u_{r}\in\langle\uX_1,\uX_2\rangle,u_j\in\langle\uX_{i_j}\rangle,1\le j\le r,r\in\mathbb N$, where $i_1\neq i_2\neq i_3\neq\cdots\neq i_r$ and
$|u_1|+|u_2|+\cdots+|u_r|\leq d$ ($|u_j|$ denotes the length of the word $u_j$). This implies
\begin{eqnarray*}
\lefteqn{P_{w,\tau}(\uA_1,\uA_2)^\star P_{w,\tau}(\uA_1,\uA_2)}\\
& \le & \|P_{w,\tau}(\uA_1,\uA_2)^\star P_{w,\tau}(\uA_1,\uA_2)\|\\
& \le & \|P_{u_1,\tau_{i_1}}(\uA_{i_1})\|^2\|P_{u_2,\tau_{i_2}}(\uA_{i_2})\|^2\|P_{u_3,\tau_{i_3}}(\uA_{i_3})\|^2\cdots
\|P_{u_r,\tau_{i_r}}(\uA_{i_r})\|^2,
\end{eqnarray*}
and in particular $\|P_{w,\tau}(\uA_1,\uA_2)\|_\bullet\le\|P_{u_1,\tau_{i_1}}(\uA_{i_1})\|\|P_{u_2,\tau_{i_2}}(\uA_{i_2})\|\|P_{u_3,\tau_{i_3}}(\uA_{i_3})\|\cdots$ $
\|P_{u_r,\tau_{i_r}}(\uA_{i_r})\|$. 

By definition, $\left\|\frac{\mathbf P(\uA_j)(I_k)}{\|\mathbf P(\underline{a}_{\tau_j})\|}\right\|_\bullet^2\leq
\Phi^\bullet_{\tau_j,d}(\uA_j)$ for any $\mathbf P(\underline{X}_j)\in\mathbb M_k(\mathbb C\langle\underline{X}_j\rangle_d)$, that is, for any 
polynomial $\mathbf P$ with $k\times k$ scalar matrix coefficients of degree at most $d$. We embed the orthonormal polynomials in 
$\mathbb M_k(\mathbb C\langle\underline{X}_j\rangle_d)$ via the obvious identification $P(\uX_j)\mapsto I_k\otimes P(\uX_j)$, which is isometric 
on operators. Thus, $\|P_{u_j,\tau_{i_j}}(\uA_{i_j})\|^2_\bullet=\|P_{u_j,\tau_{i_j}}(\uA_{i_j})^\star\|^2_\bullet
\leq\|P_{u_j,\tau_{i_j}}(\underline{a}_{\tau_{i_j}})\|^2\Phi^\bullet_{\tau_{i_j},|u_j|}(\uA_{i_j})$, which implies 
$\|P_{u_j,\tau_{i_j}}(\uA_{i_j})\|^2\le\|P_{u_j,\tau_{i_j}}(\underline{a}_{\tau_{i_j}})\|^2\Sigma^\infty_{\tau_{i_j}}(\uA_{i_j})^{|u_j|}$.
The cardinality of $\langle\underline{X}_1,\underline{X}_2\rangle_d$ is $(m+n)^d$. The number of summands in \eqref{three47}
is thus $\frac{(m+n)^d[(m+n)^d+1]}{2}$, which we agree to majorize by $(m+n)^{2d}$. We majorize the norm of the right-hand
side of \eqref{three47} by $(m+n)^{2d}$ times the maximum after $w$ of $\|P_{w,\tau}(\uA_1,\uA_2)\|^2.$ We have
\begin{align}
\|P_{w,\tau}(\uA_1,\uA_2)\|^2&\leq\|P_{u_1,\tau_{i_1}}(\uA_{i_1})\|^2\|P_{u_2,\tau_{i_2}}(\uA_{i_2})\|^2\|P_{u_3,\tau_{i_3}}(\uA_{i_3})\|^2\!\cdots
\|P_{u_r,\tau_{i_r}}(\uA_{i_r})\|^2\nonumber\\
&\leq\|P_{u_1,\tau_{i_1}}(\underline{a}_{\tau_{i_1}})\|^2\|P_{u_2,\tau_{i_2}}(\underline{a}_{\tau_{i_2}})\|^2\|P_{u_3,\tau_{i_3}}(\underline{a}_{\tau_{i_3}})\|^2\!\cdots\!
\|P_{u_r,\tau_{i_r}}(\underline{a}_{\tau_{i_r}})\|^2\nonumber\\
&\quad\ \mbox{}\times\Sigma^\infty_{\tau_{i_1}}(\uA_{i_1})^{|u_1|}\Sigma^\infty_{\tau_{i_2}}(\uA_{i_2})^{|u_2|}\Sigma^\infty_{\tau_{i_3}}(\uA_{i_3})^{|u_3|}\cdots
\Sigma^\infty_{\tau_{i_r}}(\uA_{i_r})^{|u_r|}\label{Ort}.
\end{align}
Since our traces $\tau_i,i=1,2$ satisfy a $\mathfrak C_i$-Bernstein-Markov property, the norms of the individual orthonormal polynomials are bounded from above:
\begin{equation}\label{Sanders}
\|P_{u_j,\tau_{i_j}}(\underline{a}_{i_j})\|^{2/|u_j|}\leq M(\tau_{i_j})_{|u_j|}^{2/|u_j|}\le\sup_{q\in\mathbb N}M(\tau_{i_j})_q^{2/q}<\infty.
\end{equation}
(As in Definitions \ref{BM}, \ref{BM-C}, it is necessary to specify here the dependence of the parameter $M_q$ on the tracial state $\tau_{i_j}$.)
Together with the above estimate for $\|P_{w,\tau}(\uA_1,\uA_2)\|$, this yields
\begin{align}
\|P_{w,\tau}(\uA_1,\uA_2)\|^\frac2d&\!
\le\|P_{u_1,\tau_{i_1}}(\underline{a}_{\tau_{i_1}})\|^\frac2d\|P_{u_2,\tau_{i_2}}(\underline{a}_{\tau_{i_2}})\|^\frac2d
\cdots\|P_{u_r,\tau_{i_r}}(\underline{a}_{\tau_{i_r}})\|^\frac2d\nonumber\\
&\quad\ \mbox{}\times\Sigma^\infty_{\tau_{i_1}}(\uA_{i_1})^\frac{|u_1|}{d}\Sigma^\infty_{\tau_{i_2}}(\uA_{i_2})^\frac{|u_2|}{d}
\cdots\Sigma^\infty_{\tau_{i_r}}(\uA_{i_r})^\frac{|u_r|}{d}\nonumber\\
&\le\left[\sup_{q\in\mathbb N}M(\tau_{i_1})_q^\frac{2}{q}\right]^{\!\frac{|u_1|}{d}}\left[\sup_{q\in\mathbb N}M(\tau_{i_2})_q^\frac{2}{q}\right]^{\!\frac{|u_2|}{d}}\!\cdots
\left[\sup_{q\in\mathbb N}M(\tau_{i_r})_q^\frac{2}{q}\right]^{\!\frac{|u_r|}{d}}\nonumber\\
&\quad\ \mbox{}\times\Sigma^\infty_{\tau_{i_1}}(\uA_{i_1})^\frac{|u_1|}{d}\Sigma^\infty_{\tau_{i_2}}(\uA_{i_2})^\frac{|u_2|}{d}
\cdots\Sigma^\infty_{\tau_{i_r}}(\uA_{i_r})^\frac{|u_r|}{d}\nonumber\\
&\le\max\left\{\Sigma^\infty_{\tau_{1}}(\uA_{1})\sup_{q\in\mathbb N}M(\tau_1)_q^\frac{2}{q},\ 
\Sigma^\infty_{\tau_2}(\uA_{2})\sup_{q\in\mathbb N}M(\tau_{2})_q^\frac{2}{q}\right\}.
\end{align}
Combining this estimate with \eqref{three47} yields
\begin{eqnarray*}
\lefteqn{\|\mathbf f(\uA_1,\uA_2)(I_k)\|^\frac2d}\\
&\!\le&\!\|\mathbf f(\underline{a}_{\tau_1},\underline{a}_{\tau_2})\|^\frac2d\!(m+n)^2
\max\left\{\!\Sigma^\infty_{\tau_{1}}(\uA_{1})\!\sup_{q\in\mathbb N}M(\tau_1)_q^\frac{2}{q},\ 
\Sigma^\infty_{\tau_2}(\uA_{2})\!\sup_{q\in\mathbb N}M(\tau_{2})_q^\frac{2}{q}\right\}
\end{eqnarray*}
which, by the definition of $\Sigma^\infty$, provides the desired estimate
\begin{equation}
\Sigma_\tau^\infty(\uA_1,\uA_2)\le(m+n)^2
\max\left\{\!\Sigma^\infty_{\tau_{1}}(\uA_{1})\!\sup_{q\in\mathbb N}M(\tau_1)_q^\frac{2}{q},\ 
\Sigma^\infty_{\tau_2}(\uA_{2})\!\sup_{q\in\mathbb N}M(\tau_{2})_q^\frac{2}{q}\right\},
\end{equation}
for all $(\uA_1,\uA_2)\in\mathbb M_k(\mathbb C)^{m+n}$.

Note that $\Phi^\infty_\theta(\uA)$ is finite if and only if $\Sigma^\infty_\theta(\uA)$ is. Thus, the above estimate guarantees that $\Sigma^\infty_\tau(\uA_1,\uA_2)<+\infty$,
$\Sigma^2_\tau(\uA_1,\uA_2)<+\infty$, $\Phi^2_\tau(\uA_1,\uA_2)<+\infty,$ $\Phi^\infty_\tau(\uA_1,\uA_2)<+\infty$ whenever $\Phi^\infty_{\tau_j}(\uA_j)<\infty$ and 
$\tau_j$ satisfy a $\mathfrak C_j$-Bernstein-Markov property, $j=1,2$. In particular, if $\Phi^\infty_{\tau_j}$ are finite q.e., then so are all four functions associated to $\tau$.

Assume next that the tuples $\uA_1$ and $\uA_2$ are of selfadjoint matrices (or, equivalently for our purposes, of skew-selfadjoint matrices, or of any scalar multiples of
selfadjoint matrices). Then for any polynomial $P\in\mathbb C\langle\uX_{i_j}\rangle$, one has $P(\uA_{i_j})^\star=P^\star(\uA_{i_j}),$ so that 
$P(\uA_{i_j})P(\uA_{i_j})^\star=(PP^\star)(\uA_{i_j})$. In particular, if $P\in\mathbb C\langle\uX_{i_j}\rangle_d$, then $(PP^\star)^p\in\mathbb C\langle\uX_{i_j}\rangle_{2dp}.$
In light of this observation, relation $\|P_{u_j,\tau_{i_j}}(\uA_{i_j})\|^2_\bullet\leq\|P_{u_j,\tau_{i_j}}(\underline{a}_{\tau_{i_j}})\|^2\Phi^\bullet_{\tau_{i_j},|u_j|}(\uA_{i_j})$
immediately implies 
\begin{eqnarray*}
\lefteqn{\|P_{u_j,\tau_{i_j}}(\uA_{i_j})\|^{2^{N+1}}_\bullet}\\
&\leq&\|P_{u_j,\tau_{i_j}}(\uA_{i_j})P_{u_j,\tau_{i_j}}(\uA_{i_j})^\star\|^{2^{N}}_\bullet=\|(P_{u_j,\tau_{i_j}}P_{u_j,\tau_{i_j}})^\star(\uA_{i_j})\|^{2^{N}}_\bullet\\
&\leq&\left\|\left[(P_{u_j,\tau_{i_j}}P_{u_j,\tau_{i_j}})^\star\right]^{2^{N-1}}(\uA_{i_j})\right\|_\bullet^2\\
& \le & \left\|\left[(P_{u_j,\tau_{i_j}}P_{u_j,\tau_{i_j}})^\star\right]^{2^{N-1}}(\underline{a}_{\tau_{i_j}})\right\|^2\Phi^\bullet_{\tau_{i_j},2^N|u_j|}(\uA_{i_j})\\
& = &  \left\|(P_{u_j,\tau_{i_j}}P_{u_j,\tau_{i_j}})^\star(\underline{a}_{\tau_{i_j}})\right\|^{2^{N}}\Phi^\bullet_{\tau_{i_j},2^N|u_j|}(\uA_{i_j})\\
& = &  \left\|P_{u_j,\tau_{i_j}}(\underline{a}_{\tau_{i_j}})\right\|^{2^{N+1}}\Phi^\bullet_{\tau_{i_j},2^N|u_j|}(\uA_{i_j});
\end{eqnarray*}
for $\bullet=\infty,$ we have used the equality $\|YY^\star\|=\|Y\|^2$, and for $\bullet=2$, the inequality ${\rm tr}_k(Y){\rm tr}_k(Y^\star)\leq{\rm tr}_k(YY^\star)$.
It follows that
\begin{eqnarray*}
\|P_{u_j,\tau_{i_j}}(\uA_{i_j})\|^2_\bullet
& \le &  \left\|P_{u_j,\tau_{i_j}}(\underline{a}_{\tau_{i_j}})\right\|^{2}\limsup_{N\to\infty}\Phi^\bullet_{\tau_{i_j},2^N|u_j|}(\uA_{i_j})^\frac{1}{2^N}\\
& \leq & \left\|P_{u_j,\tau_{i_j}}(\underline{a}_{\tau_{i_j}})\right\|^{2}\limsup_{q\to\infty}\Phi^\bullet_{\tau_{i_j},q}(\uA_{i_j})^\frac{|u_j|}{q}\\
& \leq & \left\|P_{u_j,\tau_{i_j}}(\underline{a}_{\tau_{i_j}})\right\|^{2}\Phi^\bullet_{\tau_{i_j}}(\uA_{i_j})^{|u_j|},
\end{eqnarray*}
for all $1\le j\le r$. Relation \eqref{Ort} becomes
\begin{align}
\|P_{w,\tau}(\uA_1,\uA_2)\|^2&\leq\|P_{u_1,\tau_{i_1}}(\uA_{i_1})\|^2\|P_{u_2,\tau_{i_2}}(\uA_{i_2})\|^2\|P_{u_3,\tau_{i_3}}(\uA_{i_3})\|^2\!\cdots
\|P_{u_r,\tau_{i_r}}(\uA_{i_r})\|^2\nonumber\\
&\leq\|P_{u_1,\tau_{i_1}}(\underline{a}_{\tau_{i_1}})\|^2\|P_{u_2,\tau_{i_2}}(\underline{a}_{\tau_{i_2}})\|^2\|P_{u_3,\tau_{i_3}}(\underline{a}_{\tau_{i_3}})\|^2\!\cdots\!
\|P_{u_r,\tau_{i_r}}(\underline{a}_{\tau_{i_r}})\|^2\nonumber\\
&\quad\ \mbox{}\times\Phi^\infty_{\tau_{i_1}}(\uA_{i_1})^{|u_1|}\Phi^\infty_{\tau_{i_2}}(\uA_{i_2})^{|u_2|}\Phi^\infty_{\tau_{i_3}}(\uA_{i_3})^{|u_3|}\cdots
\Phi^\infty_{\tau_{i_r}}(\uA_{i_r})^{|u_r|}\nonumber.
\end{align}
As for $\Sigma,$ by taking limsup as $d$ tends to infinity and picking maximizing elements $\mathbf f\in\mathbb M_k(\mathbb C\langle\uX_1,\uX_2\rangle_d)$ for each $d$,
we obtain 
\begin{eqnarray*}
\Phi^\infty_\tau(\uA_1,\uA_2) & \le & (m+n)^2
\max\left\{\Phi^\infty_{\tau_{1}}(\uA_{1})\!\sup_{q\in\mathbb N}M(\tau_1)_q^\frac{2}{q},\ 
\Phi^\infty_{\tau_2}(\uA_{2})\!\sup_{q\in\mathbb N}M(\tau_{2})_q^\frac{2}{q}\right\},
\end{eqnarray*}
for all {\em selfadjoint} tuples $(\uA_1,\uA_2)\in\mathbb M_k(\mathbb C)^{m+n}$.
\end{proof}

\begin{remark}\label{BMBound}
Under the hypotheses of Theorem \ref{free}, $\tau$ satisfies the $\mathfrak C$-Bernstein-Markov property for some constant $\mathfrak C$ satisfying the inequalities
$\min\{\mathfrak C_1,\mathfrak C_2\}\le\mathfrak C\le\sqrt{m+n}\max\left\{\sup_{q\in\mathbb N}M(\tau_1)_q^\frac1q,\sup_{q\in\mathbb N}M(\tau_{2})_q^\frac1q\right\}$.
This is indeed a very rough estimate: if $P\in\mathbb C\langle\uX_1,\uX_2\rangle_d$, then
\begin{eqnarray*}
\|P(\underline{a}_{\tau_1},\underline{a}_{\tau_2})\|&=&\left\|\sum_{w\in\langle\uX_1,\uX_2\rangle_d}c_wP_{w,\tau}(\underline{a}_{\tau_1},\underline{a}_{\tau_2})\right\|\\
& \le & \sum_{w\in\langle\uX_1,\uX_2\rangle_d}|c_w|\|P_{w,\tau}(\underline{a}_{\tau_1},\underline{a}_{\tau_2})\|\\
& \le & \left(\max_{w\in\langle\uX_1,\uX_2\rangle_d}\|P_{w,\tau}(\underline{a}_{\tau_1},\underline{a}_{\tau_2})\|\right)\sum_{w\in\langle\uX_1,\uX_2\rangle_d}|c_w|.
\end{eqnarray*}
By Cauchy-Schwarz, $\sum_{w\in\langle\uX_1,\uX_2\rangle_d}|c_w|\le\left(\sum_{w\in\langle\uX_1,\uX_2\rangle_d}|c_w|^2\right)^\frac12(m+n)^\frac{d}{2}=(m+n)^\frac{d}{2}
\|P\|_{L^2(\tau)}$ (where $(m+n)^\frac{d}{2}$ is the two-norm of the vector of ones of length equal to the cardinality of $\langle\uX_1,\uX_2\rangle_d$). The operator norm 
$\|P_{w,\tau}(\underline{a}_{\tau_1},\underline{a}_{\tau_2})\|$ of orthonormal polynomials associated to $\tau=\tau_1*\tau_2$ has been estimated in the proof of Theorem
\ref{free}: if $w=u_1u_{2}\cdots u_{r}\in\langle\uX_1,\uX_2\rangle,u_j\in\langle\uX_{i_j}\rangle,1\le j\le r,r\le d$, where $i_1\neq i_2\neq i_3\neq\cdots\neq i_r$ and
$|u_1|+|u_2|+\cdots+|u_r|= d$, then
\begin{eqnarray*}
\|P_{w,\tau}(\underline{a}_{\tau_1},\underline{a}_{\tau_2})\|^\frac{2}{d}&\le&\left[
\|P_{u_1,\tau_{i_1}}(\underline{a}_{\tau_{i_1}})\|^2\|P_{u_2,\tau_{i_2}}(\underline{a}_{\tau_{i_2}})\|^2\cdots\|P_{u_r,\tau_{i_r}}(\underline{a}_{\tau_{i_r}})\|^2\right]^{1/d}\\
& \le &\max\left\{\sup_{q\in\mathbb N}M(\tau_1)_q^\frac{2}{q},\ \sup_{q\in\mathbb N}M(\tau_{2})_q^\frac{2}{q}\right\} .
\end{eqnarray*}
Combining the last two estimates yields
$$
\|P(\underline{a}_{\tau_1},\underline{a}_{\tau_2})\|^\frac1d\le\sqrt{m+n}\max\left\{\sup_{q\in\mathbb N}M(\tau_1)_q^\frac1q,\sup_{q\in\mathbb N}M(\tau_{2})_q^\frac1q\right\}
\|P\|_{L^2(\tau)}^\frac1d.
$$
The lower bound is obvious.

This shows that Theorem \ref{free} applies to any finite free products $\tau=\tau_1*\tau_2*\cdots*\tau_s$, $s\in\mathbb N$, of bounded tracial states $\tau_j$
which satisfy a $\mathfrak C_j$-Bernstein-Markov property, $1\le j\le s$.
\end{remark}

We conclude this subsection by pointing out that $\frac12\log\Phi^\bullet_\tau(\uA),\frac12\log\Sigma^\bullet_\tau(\uA)$ have logarithmic growth as plurisubharmonic 
functions on the Euclidean space $\mathbb C^{nk^2}\simeq\mathbb M_k(\mathbb C)^n$ whenever they are non-trivial. Indeed, it follows from Section \ref{grow}
that if the sequence $\left\{\left(\Phi_{\tau,d}^\infty(\uA)\right)^{1/d}\right\}_{d\in\mathbb N}$ is locally bounded, then, as $\log$ is increasing,
$\frac12\log\Sigma_{\tau}^\infty(\uA)=\left[\sup_{d\in\mathbb N}\frac{1}{2d}\log\left(\Phi_{\tau,d}^\infty(\uA)\right)\right]^*$ has logarithmic growth, according to \cite[Theorem 1.6,
Appendix B]{ST}. Since it dominates the other three functions, it follows that $\frac12\log\Phi^\bullet_\tau(\uA),\frac12\log\Sigma^\bullet_\tau(\uA)$
have logarithmic growth as $\|\uA\|_\bullet\to\infty$.

\subsection{Derivatives}
One of the equivalent characterizations of plurisubharmonicity is via positivity: if $\Omega$ is an open subset of $\mathbb C^m$, then for any 
plurisubharmonic function $u$ on $\Omega$ and $\xi\in\mathbb C^m$, $\displaystyle\sum_{j,k=1}^m\xi_j\overline{\xi}_k\frac{\partial^2u}{\partial z_j
\partial\bar{z}_k}$ is a positive distribution in $\Omega$ (that is, it is a positive Borel measure on $\Omega$, possibly $\equiv0$) - see 
\cite[Proposition 1.43]{GZ}. We have established in the previous section that $\frac12\log\Phi^\bullet_\tau(\uA),\frac12\log\Sigma^\bullet_\tau(\uA)$,
$\bullet\in\{2,\infty\}$ are classical plurisubharmonic functions when viewed as functions on $\mathbb C^{nk^2}$, and that they have
at most logarithmic growth at infinity. Thus, according to, for instance, \cite[Proposition 3.34]{GZ}, 
$$
\int_{\mathbb C^{nk^2}}(d\, d^c\log\Phi^\bullet_\tau(\uA)^\frac12)^{nk^2}\le1,\quad\int_{\mathbb C^{nk^2}}(d\, d^c\log\Sigma^\bullet_\tau(\uA)^\frac12)^{nk^2}\le1,
$$
where $\varphi\mapsto (d\, d^c\varphi)^{nk^2}=(2\pi)^{-nk^2}\det\left[(\partial_i\overline{\partial}_j\varphi)_{1\le i,j\le nk^2}\right]\,{\rm d}V$ is the Monge-Amp\`ere operator
if $\varphi$ is of class $C^2$ ($V$ is the Lebesgue measure). Otherwise this should be understood in the weak sense. Thus, $(d\, d^c\log\Phi^\bullet_\tau(\uA)^\frac12)^{nk^2},
(d\, d^c\log\Sigma^\bullet_\tau(\uA)^\frac12)^{nk^2}$ are nonnegative measures on $\mathbb C^{nk^2}$ of mass at most one. For the purposes of this article, we
agree to call such a measure the {\em Monge-Amp\`ere measure }of the corresponding function.
In the same \cite[Proposition 3.34]{GZ} it is shown that if $\frac12\log\Phi^\bullet_\tau(\uA),\frac12\log\Sigma^\bullet_\tau(\uA)$
have exactly logarithmic growth at infinity, then these measures have total mass precisely one, i.e. the inequalities in the above displayed relations are equalities.

We show next that indeed $\uA\mapsto\frac12\log\Sigma^\infty_\tau(\uA)$ is exactly of logarithmic growth for a large class of distributions $\tau$. In order to do this,
we consider first the case of $n=1$ and then use item \eqref{max} from Lemma \ref{lemma:bounds} to conclude for arbitrary $n$. 
We look at polynomials $\mathbf L_{A,d}(X)=\frac{1}{\|A^d\|}(A^\star)^d\otimes X^d$. We assume without loss of generality that $\|a_\tau\|=1$ and, with loss of generality,
that $\sigma(a_\tau)$ has non-empty interior. We clearly have $\|\mathbf L_{A,d}(a_\tau)\|=\frac{1}{\|A^d\|}\|(A^\star)^d\|\|a_\tau\|^d=1$ 
and $\|\mathbf L_{A,d}(A)(I_k)^\star\mathbf L_{A,d}(A)(I_k)\|=\frac{1}{\|A^d\|^2}\left\|[(A^\star)^dA^d]^2\right\|=\|A^d\|^2.$
It follows that $\Phi^\infty_{\tau,d}(A)\ge\|A^d\|^2$. By definition, $\Sigma_\tau^\infty(A)=\left[\sup_{d\in\mathbb N}\Phi^\infty_{\tau,d}(A)^\frac1d\right]^*\ge
\sup_{d\in\mathbb N}\|A^d\|^\frac2d\ge\|A\|^2$. Taking logarithms and recalling Example \ref{Wigner} shows that $\frac12\log\Sigma_\tau^\infty(A)\ge\log^+\|A\|$
for $\|A\|$ sufficiently large.

Recall that if $\tau\colon\mathbb C\langle\uX_1,\uX_2\rangle\to\mathbb C$ is a bounded positive trace, then 
$\Sigma_\tau^\bullet(\uA_1,\uA_2)\ge\max\{\Sigma_{\tau|_{\mathbb C\langle\uX_1\rangle}}^\bullet(\uA_1),\Sigma_{\tau|_{\mathbb C\langle\uX_2\rangle}}^\bullet(\uA_2)\}$,
with a similar statement for $\Phi^\bullet_\tau,\bullet\in\{2, \infty\}$. Thus, if $\frac12\log\Sigma_{\tau|_{\mathbb C\langle\uX_j\rangle}}^\bullet(\uA_j)\ge
\log^+\|\uA_j\|_\bullet-K$ for some $K\ge0$ not depending on $\uA_j,j=1,2$, then $\frac12\log
\Sigma_\tau^\bullet(\uA_1,\uA_2)\ge\max\{\log^+\|\uA_1\|_\bullet-K,\log^+\|\uA_2\|_\bullet-K\}$. Since $\|\uA_1\|_2^2+\|\uA_2\|^2_2=\|(\uA_1,\uA_2)\|_2^2$,
one has 
\begin{eqnarray*}
\log\|(\uA_1,\uA_2)\|_2&=&\log\max\{\|\uA_1\|_2,\|\uA_2\|_2\}+\log\sqrt{1+\frac{\min\{\|\uA_1\|^2_2,\|\uA_2\|^2_2\}}{\max\{\|\uA_1\|^2_2,\|\uA_2\|^2_2\}}}\\
&\le&\max\{\log\|\uA_1\|_2,\log\|\uA_2\|_2\}+\log\sqrt2.
\end{eqnarray*}
Similarly, one has $\|(\uA_1,\uA_2)\|\,=\,\max\{\,\|\uA_1\|,\,\|\uA_2\|\,\},$ so that $\log\|(\uA_1,\uA_2)\|\,=\,\max\{\log\|\uA_1\|,\,\log\|\uA_2\|\}$. Putting these together, 
it follows that
$$
\frac12\log\Sigma_\tau^\bullet(\uA_1,\uA_2)\ge\log\|(\uA_1,\uA_2)\|_\bullet-\sqrt2-K
$$
whenever $\frac12\log\Sigma_{\tau|_{\mathbb C\langle\uX_j\rangle}}^\bullet(\uA_j)\ge
\log^+\|\uA_j\|_\bullet-K$ for some $K\ge0$, $j=1,2$.

We conclude that whenever $\tau\colon\mathbb C\langle\uX\rangle\to\mathbb C$ is a positive bounded tracial state such that $\tau|_{\mathbb C\langle X_j\rangle}$
is a probability distribution whose support contains a nonempty open set in $\mathbb R$, the function $\Sigma_\tau^\infty(\uA)$ has precisely logarithmic growth
at infinity, i.e. belongs to the Lelong class $\mathcal L^+(\mathbb C^{nk^2})\,=\,\{u\colon\mathbb C^{nk^2}\to[-\infty,+\infty)\colon u\text{ plurisubharmonic, }\newline
\log^+|z|-K_u\le u(z)\le\log^+|z|+K_u,z\in\mathbb C^{nk^2},\text{ for some }K_u>0\}$. Proposition \ref{Siciak} allows us to obtain the following result:

\begin{proposition}\label{Meas}
Assume that $\tau\colon\mathbb C\langle\uX\rangle\to\mathbb C$ is a positive bounded tracial state such that $\tau|_{\mathbb C\langle X_j\rangle}$
is a probability distribution whose support contains a nonempty open set in $\mathbb R$, $1\le j\le n$. If any of $\Sigma_\tau^\bullet,\Phi^\bullet_\tau,\bullet\in\{2,\infty\}$ 
is well-defined and finite q.e., then $\Sigma_\tau^\infty\in\mathcal L^+(\mathbb C^{nk^2})$. If $\tau$ satisfies in addition the $\mathfrak C$-Bernstein-Markov property \ref{BM-C},
then all of
$$
\left[\sup_{d\in\mathbb N}\log\|\kappa_{\tau,d}(\uA^\star,\uA)(I_k)\|^\frac1{2d}\right]^*,\ 
\left[\sup_{d\in\mathbb N}\log{\rm tr}_k(\kappa_{\tau,d}(\uA^\star,\uA)(I_k))^\frac1{2d}\right]^*,
$$
$$
\frac12\log\Sigma_\tau^\bullet\in\mathcal L^+(\mathbb C^{nk^2}),\quad\bullet\in\{2,\infty\}.
$$
In particular, applying the Monge-Amp\`ere operator to any of these plurisubharmonic functions yields a probability measure on $\mathbb C^{nk^2}$ for any $k\in\mathbb N$.
\end{proposition}
We remind the reader that, according to Theorem \ref{free}, $\Sigma_\tau^2$ is finite q.e. whenever $\tau|_{\mathbb C\langle X_j\rangle}$ are as in the above proposition and 
$\mathbb C\langle X_j\rangle$, $1\le j\le n$, are free with respect to $\tau$. This means that a large class of distributions is covered by the result above.

It is also known that in many cases the plurisubharmonic functions of the kind described in Proposition \ref{Meas} reach their minimum precisely at the points of the support of their
Monge-Amp\`ere measure. This, together with the considerations from Section \ref{Sec34}, guarantees that the supports of these measures are compact and invariant under conjugation by a 
unitary $k\times k$ matrix for each $k$.

\bigskip

{\bf Aside:} To conclude the discussion  of plurisubharmonic functions in the context of noncommutative sets, polynomials, distributions, and functions, we make some 
remarks regarding connections between plurisubharmonicity and the difference-differential operators, which play the role of the derivatives from classical analysis. This has no 
direct bearing on the rest of our paper at this time. In the case of noncommutative polynomials (or, more generally, functions), one has the free difference quotient standing 
in for the derivative \cite{V3,V2}. Specifically, let us consider now a polynomial $F\in\mathbb C\langle\underline{Z},\underline{Z}^\star\rangle.$ We identify 
$Z_j=X_j+iY_j,Z_j^\star=X_j-iY_j$ where $X_j,Y_j$ are selfadjoint indeterminates algebraically free from each other, $j\in\{1,\dots,n\}$.
That is, we double the number of noncommuting indeterminates. The rule for ``differentiation'' is simple in the case of polynomials in noncommuting 
selfadjoint indeterminates: $\partial_{X_i}X_j=\delta_{i,j}1\otimes1$, $\partial_{X_i}Y_k=0$, $\partial_{X_i}$ is linear from $\mathbb C\langle
\underline{Z},\underline{Z}^\star\rangle$ to $\mathbb C\langle\underline{Z},\underline{Z}^\star\rangle\otimes\mathbb C\langle\underline{Z},
\underline{Z}^\star\rangle$, and $\partial_{X_i}$ respects the Leibniz rule with respect to products of polynomials. For instance, $\partial_{X_1}
(7iX_1X_2Y_2X_1^2Y_2)=7i(1\otimes X_2Y_2X_1^2Y_2+X_1X_2Y_2X_1\otimes Y_2+X_1X_2Y_2\otimes X_1Y_2)$. To parallel classical analysis, we work 
with the derivative with respect to $Z$ and with respect to ${Z}^\star$. Specifically, $\partial_{Z_j}=\frac12(\partial_{X_j}-i\partial_{Y_j}),
\partial_{{Z}^\star_j}=\frac12(\partial_{X_j}+i\partial_{Y_j})$, a simple change of variable. Not surprisingly,
\begin{align*}
\partial_{Z_j}Z_j&=\frac12(\partial_{X_j}-i\partial_{Y_j})({X_j}+i{Y_j})=\frac12(1\otimes1+i0-i0-i\cdot i1\otimes1)=1\otimes1,\\
\partial_{Z_j}Z^\star_j&=\frac12(\partial_{X_j}-i\partial_{Y_j})({X_j}-i{Y_j})=\frac12(1\otimes1-i0-i0+i\cdot i1\otimes1)=0,\\
\partial_{Z_j^\star}Z_j&=\frac12(\partial_{X_j}+i\partial_{Y_j})({X_j}+i{Y_j})=\frac12(1\otimes1+i0+i0+i\cdot i1\otimes1)=0,\\
\partial_{Z_j^\star}Z^\star_j&=\frac12(\partial_{X_j}+i\partial_{Y_j})({X_j}-i{Y_j})=\frac12(1\otimes1-i0+i0-i\cdot i1\otimes1)=1\otimes1.
\end{align*}
This clearly implies that $\partial_{Z_j}(P(\underline{Z})^\star)=0=\partial_{Z_j^\star}P(\underline{Z})$ for all 
$P\in\mathbb C\langle\underline{Z}\rangle,1\le j\le n$. That is, an analytic polynomial is ``killed'' by $\partial_{Z_j^\star}$ and a 
conjugate-analytic one is ``killed'' by $\partial_{Z_j}$. We extend the definition of the derivative
precisely the same way to polynomials with matrix coefficients $\mathbf F\in\mathbb M_k(\mathbb C\langle\underline{Z},
\underline{Z}^\star\rangle)$.

When evaluated, the ``derivative'' becomes a proper derivative in the sense of Fr\'echet. Indeed, direct computation shows that if $\mathcal A$ is a 
star-algebra and $F\in\mathbb C\langle\underline{Z},\underline{Z}^\star\rangle$, then $\partial_{Z_j}F(\underline{a},\underline{a}^\star)(c)$ 
replaces in the above formula the tensor symbol $\otimes$ with the variable $c$:
$$
\partial_{Z_j}F(\underline{a},\underline{a}^\star)(c)=(\partial_{Z_j}F)(\underline{a},\underline{a}^\star)\circ m_c,
$$
where $m_c\colon\mathcal A\otimes\mathcal A^{\rm op}\to\mathcal A$, $m_c(p\otimes q)=pcq$. This makes the above equivalent definition of 
the classical plurisubharmonic functions ``immitable'' in the noncommutative context. Consider the easiest case, namely functions of 
the type $\phi\colon\uA\mapsto\text{tr}_k(F(\uA)^\star F(\uA))$ defined on $\mathbb M_k(\mathbb C)^n$ for each given polynomial 
$F\in\mathbb C\langle\uX\rangle$. We view $F$ as living in $\mathbb C\langle\underline{Z}\rangle$, that is, as an analytic polynomial. As mentioned 
before, when viewed as an element in $\mathbb C\langle\underline{X}\rangle$, the evaluation $F(\uA)^\star=F^\star(\uA^\star)$, and when viewed 
as an element in $\mathbb C\langle\underline{Z}\rangle$, we have $F^\star$ as an element in $\mathbb C\langle\underline{Z}^\star\rangle$.
The problem of differentiating becomes now very classical. For any given direction $C\in\mathbb M_k(\mathbb C)$, 
$$
\partial_{Z_i}\phi(\uA)(C)=\left.\frac{\partial\phi(\uA+\delta_izC)}{\partial z}\right|_{z=0},\quad\partial_{Z_j^\star}
\phi(\uA)(C)=\left.\frac{\partial\phi(\uA+\delta_j{z}C)}{\partial\bar{z}}\right|_{z=0},
$$
where we have denoted $\uA+\delta_izC=(A_1,\dots,A_{i-1},A_i+zC,A_{i+1},\dots,A_n)$.
Moreover, for a vector $\underline{C}=(C_1,\dots,C_n)^{\rm t}\in\mathbb M_k(\mathbb C)^n$, we write
$$
\sum_{i,j=1}^n\partial_{Z_i^\star}\partial_{Z_j}\phi(\uA)(C_j)(C_i)
$$
for the Levi form of $\phi$ at the point $\uA$ evaluated in $\underline{C}$ \cite[(1.3.6)]{GZ}. (It might seem normal to rather write 
$\partial_{Z_i^\star}\partial_{Z_j}\phi(\uA)(C_j)(C_i^\star)$ in the formula above; however, with the convention we have introduced, we actually 
have $[\partial_{Z_i^\star}\partial_{Z_j}\phi(\uA)](C_j)(C^\star_i)=\partial_{Z_i^\star}\partial_{Z_j}\phi(\uA)(C_j)(C_i)$ when 
$\partial_{Z_i^\star}\partial_{Z_j}\phi(\uA)$ is viewed as an operator. We shall keep in mind, and follow, this convention.) In terms of directional 
derivatives, one writes this relation naturally as $\sum\partial_{Z_i^\star}\partial_{Z_j}\phi(\uA)(C_j)(C_i)=\partial_z\partial_{\bar{z}}\phi(\uA+z
\underline{C})$.
To avoid confusion as much as possible, we will keep the notation with iterated arguments. 

Starting with Section \ref{Sec34}, we introduced several examples of classically plurisubharmonic functions obtained from noncommutative functions. These functions,
while not being themselves noncommutative functions, have been seen to posses many properties reminiscent of those of noncommutative functions. Based on these properties 
and the machinery of the free difference quotient described above, we propose next some more general function spaces that accept our examples as members.  For any 
noncommutative function $H$ defined on a noncommutative set $\Omega\subseteq\mathcal A^n$ in a $C^*$-algebra $\mathcal A$ endowed with a state $\varphi$, we may define 
$\mathscr H=\varphi\circ(H^\star H$) -- that is, $\mathscr H(\underline{a})=(\varphi\otimes\text{tr}_k)(H(\underline{a})^\star H(\underline{a}))$ for any 
$\underline{a}\in\Omega^{(k)}$ -- and then the Levi form of $\mathscr H$ at the point $\underline{a}$ evaluated in the direction $\underline{c}$ is
\begin{equation}\label{Levi}
\partial_z\partial_{\bar{z}}\mathscr H(\underline{a}+z\underline{c})|_{z=0}=\sum_{i,j=1}^n\partial_{Z_i^\star}\partial_{Z_j}\mathscr H(\underline{a})
(c_j)(c_i).
\end{equation}
The positivity of the above for all $(\underline{a},\underline{c})\in(\Omega\times\mathcal A^n)_{\rm nc}$ could be taken as a definition of an nc 
plurisubharmonic function. Indeed, as in the classical case one can find in references \cite{GZ,ST}, the above means that the correspondence
$z\mapsto\mathscr H(\underline{a}+z\underline{c})$ subharmonic as a  function from a complex neighborhood of zero into $[-\infty,+\infty)$ 
\cite[Definition 1.27]{GZ}, which seems to be a fair definition of plurisubharmonicity in our case too. However, as in the classical case, such a definition
is likely to be less general than needed: here we have assumed $H$ to exist and be a free noncommutative function, hence analytic. In \eqref{Levi}, derivatives 
exist in the strongest possible sense. It is undoubtedly necessary to consider at a minimum level-by-level weak closures of spaces of such functions (this is 
somewhat reminiscent of \cite[Example 1.41]{GZ}, but in our case analyticity should be replaced by the property of being a noncommutative function).

We should emphasize that the functions $\Phi_{\tau,d}^2$ defined in \eqref{330} belong to the space of plurisubharmonic functions of the form $\varphi\circ(H^\star H)$
for some nc function $H$, while the functions $\Phi_{\tau,d}^\infty$ from \eqref{331} generally do not, as they might not be classically differentiable at some levels. However, 
$\Phi_{\tau,d}^\infty$ is a well-defined (and monotone) limit of fractional powers of functions of the type $\Phi_{\tau,p}^2$ with $p$ tending to infinity along a properly 
chosen subsequence, hence plurisubharmonic in the sense that $z\mapsto\Phi_{\tau,d}^\infty(\uA+z\underline{C})$ is classically subharmonic on some neighborhood of zero in 
$\mathbb C$ for all $k\in\mathbb N,\underline{A},\underline{C}\in(\mathbb M_k(\mathbb C))^n$ (see \cite[Theorem 1.46]{GZ}). For these functions, \eqref{Levi} does hold, but 
in a weak sense. This same statement remains true for $\Phi_\tau^2,\Sigma_\tau^2,\Sigma_\tau^\infty,$ and $\Phi_\tau^\infty$ {\em if} they are less than $+\infty$. Thus,
\begin{align}
\sum_{i,j=1}^n\partial_{Z_i^\star}\partial_{Z_j}\Phi_{\tau,d}^2(\underline{A})(C_j)(C_i) & \ge0\text{ pointwise, and}\\
\sum_{i,j=1}^n\partial_{Z_i^\star}\partial_{Z_j}\Phi_{\tau,d}^\infty(\underline{A})(C_j)(C_i) & \ge0\text{ in the sense of distributions},
\end{align}
for all $k,d\in\mathbb N,$ and all $\uA,\underline{C}\in\mathbb M_k(\mathbb C)^n$, with similar statements for $\Sigma$. By positivity in the sense of distributions we 
mean that $\sum_{i,j=1}^n\partial_{Z_i^\star}\partial_{Z_j}\Phi_{\tau,d}^\infty(\underline{A})(C_i)(C_j)$ is a positive distribution on the Euclidean
space $\mathbb C^{nk^2}$ - the variables considered being $\uA$, while $\underline{C}$ should be viewed as parameters. 
We expect that eventually the correct use of the difference-differential operators will reveal a tighter connection between $\tau$ and the Monge-Amp\`ere
measures associated to $\frac12\log\Phi_\tau^2,\frac12\log\Sigma_\tau^2,\frac12\log\Sigma_\tau^\infty,$ and $\frac12\log\Phi_\tau^\infty$, as well as
to the logarithm of the norms of the Christoffel-Darboux kernels.

\section{Conclusion, Perspectives, and Numerical Experiments}
\label{sec:bench}

We have seen in the previous section that to a bounded, positive, tracial noncommutative distribution $\tau\colon\mathbb C\langle\underline{X}\rangle\to\mathbb C$ 
one associates, under some reasonable assumptions, sequences of Borel probability measures supported on 
$\mathbb M_k(\mathbb C)^n\simeq\mathbb C^{nk^2}.$ We focus here on those associated
to the logarithm of the Christoffel-Darboux kernel, which we denote by $\{\mu_{\tau,k}\}_{k\in\mathbb N}$. 
While we have investigated some properties of these measures, it is not yet clear 
how strongly they are related to $\tau$ itself besides being determinated by it. We conjecture that these measures might be good approximants 
for $\tau$ in a very precise way, thus providing another means to find random matrix approximants to (some) noncommutative disributions.

\begin{conjecture}\label{supp}
Assume that $\tau\colon\mathbb C\langle\uX\rangle\to\mathbb C$ is a bounded faithful tracial state. 
Suppose that $\{X_1,\dots,X_n\}$ are free with respect to $\tau$ and $\tau|_{\mathbb C\langle X_j\rangle}$ has a support with nonempty interior, $1\le j\le n$.
For all $f \in \SymRX$, one has 
\begin{align}
\label{eq:momentcvg}
& \!\lim_{d\to\infty}\lim_{k\to\infty}\!\!\sup_{\uA\in\Omega_{\tau,d}^{(k)}}
|\mathrm{tr}_k(f(\uA))-\tau(f(\underline{X}))|\!=\!\!\lim_{d\to\infty}\lim_{k\to\infty}\!\!\sup_{\uA\in\Omega_{\tau,d}^{(k)}}
|\mathrm{tr}_k(f(\uA))-\mathrm{tr}(f(\underline{a}_\tau)) |=0,\\
& \lim_{d \to \infty}\lim_{k\to\infty}\!\!\sup_{\uA \in\Omega_{\tau,d}^{(k)}} 
 \|f(\uA)\|_{\mathbb M_k(\mathbb C)}=\|f(\underline{a}_\tau)\|_{W^*(\tau)},
\end{align}
where $\Omega_{\tau,d}^{(k)} = \{\uA :  [\limsup_{d\to\infty}{\rm tr}_k (\kappa_{\tau,d}(\uA,\uA^\star)(I_k) )^{1/d}]^* \leq n\}$.
\end{conjecture}

Let us note some immediate consequences in this context of Theorem \ref{th:christoffeloptim}. Observe 
that in the above conjecture, the matrix variables $\uA$ are taken as approximants of $\underline{a}_\tau$: for this to be possible, it is necessary that the algebra $W^*(\tau)$
satisfies Connes' embedding property, such approximants exist. Viewed this way, our conjecture states that the approximation is
good along the subspaces generated by the first $\bsigma(n,d)$ vectors in the Hilbert module generated by the orthogonal polynomials of 
$\tau$ as $d$ tends to infinity.

In commutative analysis, the level sets of the Christoffel polynomial are used to approximate the support of the distribution of the variables
$\underline{a}_\tau$. 
From this perspective, Conjecture \ref{supp} can be viewed as specifying the notion of support for noncommutative distributions.

Here, we illustrate our theoretical framework for the computation of Christoffel-Darboux kernels for free semicircular and Poisson laws and attempt to provide some numerical
data to support our Conjecture \ref{supp}. Our experiments are performed with Mathematica 12, and we heavily rely on the NCAlgebra package \cite{helton1996ncalgebra} to handle 
noncommutative polynomials.  All results were obtained on an Intel Xeon(R) E-2176M CPU (2.70GHz $\times$ 6) with 32Gb of RAM.
Our code is available online\footnote{\url{http://homepages.laas.fr/vmagron/files/NCCD.zip}}.

Given $\tau$ and $d$, we first compute the Christoffel polynomial $\kappa_{\tau,d}$ from the knowledge of the $d$-th order moment matrix associated 
to $\tau$. Then, we consider approximations $\widetilde \Phi_{\tau,d}^{(k)}(\uA) = {\rm tr}_k (\kappa_{\tau,d}(\uA,\uA^\star)(I_k) )^\frac1d$ of 
\begin{align*}
\left[\limsup_{d\to\infty}{\rm tr}_k (\kappa_{\tau,d}(\uA,\uA^\star)(I_k) )^\frac1d\right]^*
.\\
\end{align*}
%
For several values of the matrix size $k$, we sample $N$ matrices from the following set
\[
\widetilde \Omega_{\tau,d}^{(k)} = \{\uA :  n - \varepsilon \leq \widetilde \Phi_{\tau,d}^{(k)}(\uA) \leq n + \varepsilon\}
\,,\]
with respect to a given matrix distribution.
Overall, it boils down to sampling in an $\varepsilon$-perturbation of the $n$-level set of the approximate Siciak function $\widetilde \Phi_{\tau,d}^{(k)}$.
Then, we choose an $f \in \SymRX$ and empirically check if \eqref{eq:momentcvg} from Conjecture \ref{supp} holds on average. This verification is 
performed by computing the expectation $E (\mathrm{tr}_k(f(\uA)))$ over the $N$ selected samples $\uA$ from $\widetilde \Omega_{\tau,d}^{(k)}$ 
and comparing it with $\tau (f(\uA))$.
For all experiments, we select $N = 10^5$.

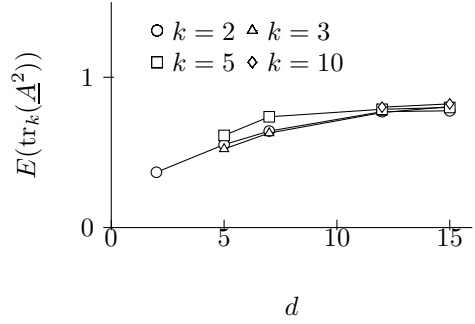
\begin{figure}
\begin{tikzpicture}[y=2cm, x=.3cm]
	\draw (0,0) -- coordinate (x axis mid) (16,0);
    	\draw (0,0) -- coordinate (y axis mid) (0,1.5);
    	\foreach \x in {0,5,...,15}
     		\draw (\x,-1pt) -- (\x,3pt)
			node[anchor=north] {\x};
    	\foreach \y in {0,1}
     		\draw (1pt,\y) -- (-3pt,\y) 
     			node[anchor=east] {\y}; 
	\node[below=0.8cm] at (x axis mid) {$d$};
	\node[rotate=90, above=0.8cm] at (y axis mid) {$E( {\rm tr}_k (\uA^2))$};

	\draw plot[mark=*, mark options={fill=white}] 
		file {1k2.data};
	\draw  plot[mark=triangle*, mark options={fill=white} ] 
		file {1k3.data}; 
	\draw  plot[mark=square*, mark options={fill=white} ] 
		file {1k5.data}; 
	\draw  plot[mark=diamond*, mark options={fill=white} ] 
		file {1k10.data}; 				
    
	\begin{scope}[shift={(6,1.1)}] 
	\draw  
		plot[mark=diamond*, mark options={fill=white}] (0.25,0) -- (0.5,0) 
		node[right]{$k=10$};	
	\draw[xshift=-3*\baselineskip]
		plot[mark=square*, mark options={fill=white}] (0.25,0) -- (0.5,0) 
		node[right]{$k=5$};
	\draw[yshift=\baselineskip]  
		plot[mark=triangle*, mark options={fill=white}] (0.25,0) -- (0.5,0)
		node[right]{$k=3$};	\draw[yshift=\baselineskip,xshift=-3*\baselineskip]  
		plot[mark=*, mark options={fill=white}] (0.25,0) -- (0.5,0)
		node[right]{$k=2$};		
	\end{scope}
\end{tikzpicture}
\caption{Averaging the normalized traces of $\uA^2$ over $N = 10^5$ samples of $\widetilde \Omega_{\tau,d}^{(k)}$ for a single free semicircular of variance $1$.}
\label{fig:onedim}
\end{figure}

\subsection{Single free semicircular}
Here, we consider a single standard free semicircular $\uA$ of variance 1. 
By \cite[Definition 4]{mingo2017free}, the odd moments of $\uA$ are 0 and the even moments are given by the Catalan numbers, i.e., 
$\tau(\uA^{2d+1}) = 0$ and $\tau(\uA^{2d}) = \frac{1}{d+1} \binom{2 d}{d}$.
The first four corresponding values of $\kappa_{\tau,d}$ are:
\begin{align*}
\kappa_{\tau,1}(\uA,\uA)(1) &=  1+\uA^2 \,, \\\kappa_{\tau,2}(\uA,\uA)(1) &=  2-\uA^2 +\uA^4  \,,\\
\kappa_{\tau,3}(\uA,\uA)(1) &=  2+3\uA^2- 3 \uA^4 + \uA^6 \,, \\
\kappa_{\tau,4}(\uA,\uA)(1) &=  3-3\uA^2 +8\uA^4 -5\uA^6+ \uA^8 \,.\\
\end{align*}
In Figure \ref{fig:onedim}, we present the numerical experiments obtained after sampling $\widetilde \Omega_{\tau,d}^{(k)}$ with respect to the 
Gaussian orthogonal matrix distribution with unit scale parameter, for degree $d \in \{1,\dots,15\}$ and matrix size $k \in \{2,3,5,10\}$. 
We choose a perturbation of $\varepsilon = 0.7$ to sample $\widetilde \Omega_{\tau,d}^{(k)}$ and $f(\uA) = \uA^2$.
The figure confirms that the average $E ({\rm tr}_k(f(\uA)))$ over the samples of $\widetilde\Omega_{\tau,d}^{(k)}$ gets closer to $\tau(f(\uA))=1$, 
when $d$ and $k$ increase.

\begin{figure}
\begin{tikzpicture}[y=2cm, x=.7cm]
	\draw (0,1.5) -- coordinate (x axis mid) (8.5,1.5);
    	\draw (0,1.5) -- coordinate (y axis mid) (0,3.2);
    	\foreach \x in {1,...,8}
     		\draw (\x,1.48) -- (\x,1.52)
			node[anchor=north] {\x};
    	\foreach \y in {2,3}
     		\draw (1pt,\y) -- (-3pt,\y) 
     			node[anchor=east] {\y}; 
	\node[below=0.8cm] at (x axis mid) {$d$};
	\node[rotate=90, above=0.8cm] at (y axis mid) {$E( {\rm tr}_k (A_1 A_1 A_2 A_2 A_1 A_1))$};

	\draw plot[mark=*, mark options={fill=white}] 
		file {k2.data};
	\draw  plot[mark=triangle*, mark options={fill=white} ] 
		file {k3.data}; 
	\draw  plot[mark=square*, mark options={fill=white} ] 
		file {k4.data}; 		
    
	\begin{scope}[shift={(6,2.8)}] 
	\draw  
		plot[mark=square*, mark options={fill=white}] (0.25,0) -- (0.5,0) 
		node[right]{$k=4$};
	\draw[yshift=\baselineskip]  
		plot[mark=triangle*, mark options={fill=white}] (0.25,0) -- (0.5,0)
		node[right]{$k=3$};
	\draw[yshift=2*\baselineskip]  
		plot[mark=*, mark options={fill=white}] (0.25,0) -- (0.5,0)
		node[right]{$k=2$};		
	\end{scope}
\end{tikzpicture}
\caption{Averaging the normalized traces of $A_1 A_1 A_2 A_2 A_1 A_1$ over $N = 10^5$ samples of $\widetilde \Omega_{\tau,d}^{(k)}$ for a pair of free semicirculars of variance $1$.}
\label{fig:semicircular}
\end{figure}
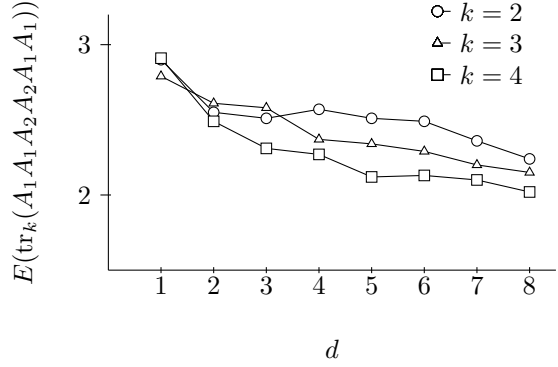

\subsection{Pair of free semicircular distributions}
Here, we consider a pair $\uA = (A_1,A_2)$ of standard free semicircular distributions of variance 1. 
The odd moments of $\uA$ are 0 and the even moments are given by the number of non-crossing pairings which respect the indices, e.g., 
$\tau(A_1 A_2 A_1) = 0$ and $\tau(A_1 A_1 A_2 A_2 A_1 A_1) = 2$.
The first three corresponding values of $\kappa_{\tau,d}$ are:
\begin{align*}
\kappa_{\tau,1}(\uA,\uA)(1) &=  1+A_1^2 + A_2^2 \,, \\\kappa_{\tau,2}(\uA,\uA)(1) &=  3-A_1^2 -A_2^2 + A_1^4 + A_1 A_2^2 A_1 + A_2 A_1^2 A_2 + A_2^4 \,,\\
\kappa_{\tau,3}(\uA,\uA)(1) &=  3+5A_1^2+5A_2^2 - 3 A_1^4 - 2 A_1^2 A_2^2 - A_1 A_2^2 A_1 - A_2 A_2^1 A_2 -2 A_2^2 A_1^2 \\
& \quad - 3 A_2^4 + A_1^6 + A_1^2 A_2^2 A_1^2 
 + A_1 A_2 A_1^2 A_2 A_1 + A_1 A_2^4 A_1 
+ A_2 A_1^4 A_2 \\
& \quad + A_2 A_1 A_2^2 A_1 A_2 + A_2^6 \,. \\
\end{align*}
In Figure \ref{fig:semicircular}, we present the numerical experiments obtained after sampling $\widetilde \Omega_{\tau,d}^{(k)}$ with respect to the Gaussian orthogonal matrix distribution with unit scale parameter, for degree $d \in \{1,\dots,8\}$ and matrix size $k \in \{2,3,4\}$. 
We choose a perturbation of $\varepsilon = 0.7$ to sample $\widetilde \Omega_{\tau,d}^{(k)}$ and $f(\uA) = A_1 A_1 A_2 A_2 A_1 A_1$.
As for the single case, the average of normalized traces over the samples of $\widetilde \Omega_{\tau,d}^{(k)}$ gets closer to $\tau(f(\uA)) = 2$, when $d$ and $k$ increase.

\subsection{Pair of free Poisson distributions}
Finally, we consider a pair $\uA = (A_1,A_2)$ of standard free Poisson distributions of rate $c$.
Since the cumulants of each law are equal to $c$, one can get the moments thanks to the moment-cumulant formula, by computing the non-crossing partitions.
For instance $\tau(A_1) = \tau(A_2) = c$, $\tau(A_1^2) = \tau(A_2^2) = c+c^2$, $\tau(A_1 A_2) = \tau(A_2 A_1) = c^2$ and $\tau(A_1^3) = c+3 c^2+ c^3$.
The first and second values of $\kappa_{\tau,d}$ are:
\begin{align*}
\kappa_{\tau,1}(\uA,\uA)(1) &=  1 + 2 c - 2 A_1 - 2 A_2 + \frac{A_1^2 + A_2^2}{c} \,, \\
\kappa_{\tau,2}(\uA,\uA)(1) &=  1 + 2 c + 4 c^2 - (4 + 8c) (A_1 + A_2) 
+ \left(8 +\frac{5}{c} +  \frac{1}{c^2} \right) (A_1^2+ A_2^2) \\
& \quad +  4 (A_1 A_2 + A_2 A_1) - \left(\frac{2}{c^2} + \frac{4}{c} \right) (A_1^3+A_2^3) \\
& \quad - \frac{A_1^2 A_2 + A_2 A_1^2+ A_1 A_2^2 + A_2^2 A_1}{c} - \frac{2 (A_1 A_2 A_1 + A_2 A_1 A_2)}{c}  \\
& \quad + \frac{A_1^4+A_1 A_2^2A_1 + A_2 A_1^2 A_2 + A_2^4}{c^2} \,. \\
\end{align*}
Recall that a Wishart matrix is of the form $A = \frac{1}{k} G G^\star $, where $G$ is a $k \times M$ matrix with entries being standard complex Gaussian random variables with mean $0$ and $E(|G_{ij}|^2) = 1$.
According to \cite[\S~4.5.1]{mingo2017free}, if one sends $k$ and $M$ to infinity so that the ratio $M/k$ is kept fixed to the value $c$, the limiting distribution is the free Poisson distribution of rate $c$.
In our experiments, we sample $\widetilde \Omega_{\tau,d}^{(k)}$ with respect to the Wishart matrix distribution, for $c = k = 5$,  with a perturbation of $\varepsilon = 10$. 
For $f(\uA) = A_1+ A_2$, we obtain the successive averages of normalized traces: $10.42$, $10.24$, $10.16$, $10.14$ and $10.1 \simeq 10 = \tau(A_1 + A_2)$, for $d \in \{1,\dots,5\}$, respectively.
For $f(\uA) = A_1 A_2$, we obtain  $26.78$, $25.78$, $25.7$, $25.6$ and $25.4$, while $\tau(A_1 A_2) = 5^2 = 25$.
Eventually, for $f(\uA) = A_1^3$, we obtain  $264$, $255$, $244$, $242$ and $239$, while $\tau(A_1^3) = 5 + 3 \times 5^2 + 5^3 = 205$.

Overall, the empirical convergence behavior of the sequence of averaged normalized traces seems to match with the theoretical moment values. 
It is important to emphasize that due to the computational burden associated to the rapidly growing number of non-crossing partitions, our method is currently limited to 
problems of modest size. 
We plan to overcome these scalability issues by exploiting the sparsity and symmetry properties of the noncommutative Christoffel-Darboux kernels, e.g., by relying on the framework derived in  \cite{helton2008measures}.

\subsection{Application to trace polynomial optimization}
We recall the tracial version of Lasserre's hierarchy \cite{nctrace} to minimize the trace of a noncommutative polynomial on a noncommutative semialgebraic set.
Given $f \in \SymRX$, a positive integer $m$ and $G = \{g_1,\dots,g_m \} \subseteq \SymRX$,
 the {\em semialgebraic} set $D_G$ associated to $G$ is defined as follows:
\begin{align}
\label{eq:DG}
D_G := \bigcup_{k \in \N} \{ \underline{A} = (A_1,\dots,A_n) \in \Sbb_k^n : g_j(\underline{A}) \succeq 0 \,, \quad j=1 \dots, m  \} \,.
\end{align}
Now fix a possibly infinite dimensional complex Hilbert space $\mathcal H$ endowed with a scalar product $\langle\cdot\mid\cdot\rangle$ and let 
$\mathcal{B}(\mathcal{H})$ be the algebra of all bounded linear operators on $\mathcal H$. As usual, positivity for an operator $C\in\mathcal{B}
(\mathcal{H})$ means that $\langle C\xi\mid\xi\rangle\ge0$ for all $\xi\in\mathcal H$, and, as for matrices, it is denoted $C\succeq0$. Strict positivity 
means that in addition $C$ is also invertible, and is denoted by $C\succ0$. For any set $\mathcal M$ of bounded operators on $\mathcal H$, we define
$$
D_G^\mathcal M:=\{\uA=(A_1,\dots,A_n)\in \mathcal M^n : \uA=\uA^\star, g_j(\uA)\succeq0\text{ on }\mathcal H\,, \ j=1 \dots, m  \},
$$
and its noncommutative extension
\begin{equation}\label{eq:NCsemi}
D_{G,\rm nc}^\mathcal M:=\bigcup_{k\in\mathbb N}\{\uA=(A_1,\dots,A_n)\in\mathbb M_k^{\rm sa}(\mathcal M)^n : g_j(\uA)\succeq0\text{ on }
\mathcal H^k,\  j=1, \dots, m  \}.
\end{equation}
(We denote by $\mathbb M_k^{\rm sa}(\mathcal M)$ the set of selfadjoint matrices with entries from $\mathcal M$.)
The important case for us is when $\mathcal M$ is a finite von Neumann algebra endowed with a normal faithful tracial state $\Trace$ (see \cite[Chapter 
V.2]{Takesaki01}). In particular, the set $D_G$ defined in \eqref{eq:DG} corresponds, with the notation from \eqref{eq:NCsemi}, to 
$D_{G,\rm nc}^\mathbb C$.

Given an arbitrary algebra $\mathcal A$ and $g,h\in\mathcal A$, we denote by $[g, h] := gh - hg$ the {\em commutator} of $g$ and $h$. 
In the particular case when $\mathcal A=\mathbb C\langle\uX\rangle$, two nc polynomials $g, h$ are called cyclically equivalent
(denoted $g \cyc h$) if $g-h$ is a sum of commutators. In particular, $\tau(g)=\tau(h)$.

Given the trace $\tau$ and the polynomial $g_j$, the matrix $\M_d (g_j\tau)$ is the localizing matrix associated to the nc  polynomial $g_j$. It is defined 
the following way: let $d_j = \lceil \text{deg}\,g_j / 2 \rceil\in\mathbb N$. The localizing matrix $\M_d (g_j\tau)$ is indexed on 
$\langle\uX\rangle_{d-d_j}$ with entry $(v, w)$ being equal to $\tau (v^\star g_j w)$. Note that the moment matrix $\M_d(\tau)$ is the localizing 
matrix associated to $g =1.$

Let us define  $\Trace_{\min} (f, G)$ as follows:
\begin{align}
\label{eq:constr_trace}
\Trace_{\min}(f,G) := \inf \{\Trace f(\underline{A}) : \uA \in D_G \} \,.
\end{align}

For an arbitrary finite von Neumann algebra $\mathcal M$ on a {\em separable} Hilbert space $\mathcal H$, we define  
$\Trace_{\min}(f,G)^{\mathcal M}$ as the trace-minimum of $f$ on $D_{G,\rm nc}^{\mathcal M}$:
\begin{align}
\label{eq:constr_trace21}
\Trace_{\min}(f,G)^{\mathcal M} := \inf \{\Trace f(\underline{A}) : \uA\in D_{G,\rm nc}^{\mathcal M} \} \,.
\end{align}

We approximate $\Trace_{\min}(f,G)^{\mathcal M}
$ from below via the following hierarchy of semidefinite programs, indexed by $d$: 
\begin{equation}
\label{eq:constr_trace_primal}
\begin{aligned}
\tau_{d}(f,K) := \inf\limits_{\tau} \quad  & \tau(f) \\	
\text{s.t.} 
\quad & \M_d(\tau)_{u,v} = \M_d(\tau)_{w,z}  \,, \quad \text{for all } u^\star v \cyc w^\star z \,, \\
\quad & \M_d(\tau)_{1,1} = 1 \,, \\
\quad & \M_d(\tau) \succeq 0 \,, \quad  \M_d(\tau) \in \Sbb_{\bsigma(n,d)} \,. \\
\quad & \M_d(g_j \tau) \succeq 0 \,,  \quad  \M_d(g_j \tau) \in \Sbb_{\bsigma(n,d - d_j)} \,,  \quad j = 1,\dots,m \,.
\end{aligned}
\end{equation}
The optimization variables of the semidefinite program \eqref{eq:constr_trace_primal} are the entries of the moment and localizing matrices. 
If at a given relaxation order $d$, the computed  tracial moment matrix $\M_d(\tau)$ is \emph{flat}, then finite convergence of the hierarchy is guaranteed and one can extract the minimizers of the corresponding optimization problem, see, e.g.,~\cite[Theorem~1.69]{burgdorf16}. 
One research investigation would be to  approximate such minimizers when the flatness condition does not hold.
The algorithmic scheme would be to build an approximation of the noncommutative Christoffel-Darboux kernel associated to $\tau$ from the inverse moment matrix (as stated in Proposition \ref{prop:inverse}) and perform a sampling procedure as done in the three former subsections.
%
%
\section*{Acknowledgements} 
STB would like to thank Ahmed Zeriahi and Norm Levenberg for very helpful discussions regarding plurisubharmonic functions and the Bernstein-Markov property.
All authors are grateful to Jean-Bernard Lasserre for encouraging and promoting the line of research that led to this paper.
Part of this work was conducted in November 2019 while STB was a visiting fellow supported by the Faculty of Natural Sciences 
Distinguished Scientist Visitors Program of the Ben Gurion University of the Negev, Beer-Sheva, Israel, 
and during the three authors' visit at the MFO workshop I.D.: 2010, ``Real Algebraic Geometry with a View Toward Hyperbolic Programming and Free Probability'' (1--7 March 2020).
The second-named author was supported by the Tremplin ERC Stg Grant ANR-18-ERC2-0004-01 (T-COPS project), the European Union's Horizon 2020 research and innovation programme under the Marie Sklodowska-Curie Actions, grant agreement 813211 (POEMA), as well as the FMJH Program PGMO (EPICS project). 
We would also like to thank the anonymous referee for a careful reading and numerous helpful comments that allowed us to improve the paper.
\bibliographystyle{alpha}
\newcommand{\etalchar}[1]{$^{#1}$}

\end{document}